\documentclass[final,3p,times]{elsarticle}

\usepackage{amssymb}
\usepackage{amsmath}
\usepackage{amsthm}

\usepackage[]{graphicx}
\usepackage[]{color}
\usepackage{alltt}
\usepackage{array}
\usepackage{amsfonts}
\usepackage{bbm}
\usepackage{todonotes}
\usepackage{mathtools}
\usepackage{enumitem}
\usepackage{tikz}
\usetikzlibrary{shapes,arrows,decorations.pathmorphing,graphs,positioning,backgrounds}
\usepackage{subfigure}

\usepackage{hyperref}
\hypersetup{
  colorlinks = true,
  linkcolor ={blue},
  urlcolor = {blue},
  citecolor ={blue}}
\usepackage[capitalize,nameinlink]{cleveref}
  
\crefname{exmp}{Example}{Examples}
\Crefname{exmp}{Example}{Examples}
\usepackage{crossreftools}

\usepackage{amsthm,thmtools}

\theoremstyle{plain}

\newtheorem{theorem}{Theorem}

\newtheorem{proposition}{Proposition}
\newtheorem{lemma}{Lemma}

\theoremstyle{definition}
\newtheorem{defn}{Definition}

\newcommand{\myparagraph}[1]{\paragraph{#1}}

\usepackage{ifthen}
\let\oldtextcolor\textcolor
\renewcommand{\textcolor}[2]{\ifthenelse{\equal{#1}{red}}{\oldtextcolor{black}{#2}}{\oldtextcolor{#1}{#2}}}
\let\oldcolor\color
\renewcommand{\color}[1]{\ifthenelse{\equal{#1}{red}}{\oldcolor{black}}{\oldcolor{#1}}}

\usepackage{bm}
\usepackage{dsfont}

\newcommand{\ind}{\mathds{1}}
\newcommand{\R}{\mathbb{R}}

\providecommand{\P}{}
\renewcommand{\P}{\mathbb{P}}

\newcommand{\E}{\mathbb{E}}

\newcommand{\bnull}{\bm{0}}
\newcommand{\ba}{\bm{a}}

\newcommand{\bu}{\bm{u}}

\newcommand{\bx}{\bm{x}}

\newcommand{\bU}{\bm{U}}

\newcommand{\bX}{\bm{X}}
\newcommand{\bY}{\bm{Y}}

\newcommand{\Bcal}{\mathcal{B}}

\newcommand{\Fcal}{\mathcal{F}}
\newcommand{\Gcal}{\mathcal{G}}

\newcommand{\Ncal}{\mathcal{N}}

\newcommand{\Vcal}{\mathcal{V}}

\newcommand{\Xcal}{\mathcal{X}}

\newcommand{\eps}{\varepsilon}

\newcommand{\bDelta}{\bm{\Delta}}

\newcommand{\bdeta}{\bm{\eta}}
\newcommand{\btheta}{\bm{\theta}}

\newcommand{\bnu}{\bm{\nu}}

\newcommand{\bxi}{\bm{\xi}}

\newcommand{\brho}{\bm{\rho}}

\newcommand{\0}{\bm{0}}

\newcommand{\sign}{\operatorname{sign}}

\renewcommand{\bar}{\overline}

\newcommand{\hbtheta}{\bm{\hat{\theta}}}
\newcommand{\hbdeta}{\bm{\hat{\eta}}}

\newcommand{\htheta}{\hat{\theta}}

\newcommand{\hbu}{\bm{\hat{u}}}
\newcommand{\hbU}{\bm{\widehat{U}}}

\newcommand{\btu}{\bm{\tilde{u}}}
\newcommand{\btU}{\bm{\tilde{U}}}

\newcommand{\bttheta}{\bm{\tilde{\theta}}}

\newcommand{\lf}{\left}
\newcommand{\ri}{\right}
 \providecommand{\Pr}{}
\renewcommand{\Pr}{\mathbb{P}}

\newcommand{\cov}{{\mathds{C}\mathrm{ov}}}

\newcommand{\sumin}{\sum_{i = 1}^n}

\makeatletter
\def\blfootnote{\gdef\@thefnmark{}\@footnotetext}
\makeatother

\renewcommand{\citet}{\citep}

\usepackage{lineno}

\begin{document}

\begin{frontmatter}

\title{Properties of stepwise parameter estimation in high-dimensional vine copulas}

\author[label1]{Jana Gauss} 

\author[label1,label2]{Thomas Nagler}

\affiliation[label1]{organization={Department of Statistics},city={Munich},
country={Germany}}
            
\affiliation[label2]{organization={Munich Center for Machine Learning},city={Munich},
country={Germany}}

\begin{abstract}
The increasing use of vine copulas in high-dimensional settings, where the number of parameters is often of the same order as the sample size, calls for asymptotic theory beyond the traditional fixed-$p$, large-$n$ framework.
We establish consistency and asymptotic normality of the stepwise maximum likelihood estimator for vine copulas when the number of parameters diverges as $n \to \infty$.
Our theoretical results cover both parametric and nonparametric estimation of the marginal distributions, as well as truncated vines, and are also applicable to general estimation problems, particularly other sequential procedures.
Numerical experiments suggest that the derived assumptions are satisfied if the pair copulas in higher trees converge to independence copulas sufficiently fast.
A simulation study substantiates these findings and identifies settings in which estimation becomes challenging.
In particular, the vine structure strongly affects estimation accuracy, with D-vines being more difficult to estimate than C-vines, and estimates in Gumbel vines exhibit substantially larger biases than those in Gaussian vines.
 \end{abstract}


\end{frontmatter}

\section{Introduction}\label{sec:intro}
 
The increasing availability of complex, high-dimensional data necessitates flexible methods to describe and estimate dependence between random variables.
Copulas enable the modeling of the dependence structure independently of the marginal distributions, thereby providing a framework for constructing multivariate distributions beyond the normal and Student's $t$-distribution (see \citet{Nelsen2007} and \citet{Genest2007} for introductions).
This decomposition allows the combination of arbitrary, possibly heavy-tailed or skewed, univariate distributions into a multivariate distribution with an arbitrary depedence structure.
Although numerous bivariate parametric copula families can capture asymmetric dependence and heavy tails, such flexibility is not readily available in higher dimensions.

Vine copulas provide a versatile approach for modeling multivariate distributions by decomposing a copula into a hierarchical structure of bivariate copulas \citep{Joe96, Bedford2001, Bedford2002}.
Applications include finance and insurance \citep{Erhardt2012, Dissmann2013, Aas2016}, medical sciences \citep{Stober2015}, weather forecasting \citep{Moller2018}, and hydrology \citep{Haff2015}.
Most of them use the sequential maximum-likelihood-based approach proposed by \citet{Aas2009} for the parametric estimation of vine copulas.
While the joint maximum-likelihood estimator becomes computationally infeasible in high dimensions, the stepwise estimator exploits the hierarchical structure of the model, thereby substantially reducing the computational complexity.
The asymptotic properties of this estimator in finite dimensions have been studied in \citet{Haff13} and \citet{Stober2013}.

In classical settings, including many of the aforementioned references, the copula dimension typically does not exceed $d = 15$, while the sample size is larger than $n = 1000$.
As the dimension increases, estimation rapidly becomes challenging, since the number of parameters to estimate is approximately $p \approx 0.5 d^2$ if each pair copula has one parameter.
Comparatively high dimensions are not uncommon, see, for instance, \citet[$d = 64$, $p \approx 2000$, $n \approx 7500$]{Haff2015} and the extreme value copulas in \citet[$d = 29$, $p \approx 380$, $n < 500$]{Kiri24}.
In such applications, the validity of theoretical guarantees for finite-dimensional inference remains unclear.
Existing work on high-dimensional vine copulas \citep[see][Section 6.2 for an overview]{CzadoNagler} primarily focuses on sparse models such as truncated vines, especially the choice of an appropriate truncation level \citep{Kurowicka2010, Brechmann2012, Brechmann2015b}, and thresholded vine copulas \citep{Nagler2019}, as well as connections between vines and graphical models, e.g., directed acyclic graphs \citep{Mueller2018, Mueller2019, Mueller2019b}.
The estimation of covariance matrices of Gaussian copulas when the dimension exceeds the sample size is analyzed in \citet{Xue2014}, however, their results provide limited insight into the estimation of vine copulas.
The results on model selection in \citet{Nagler2019} rely on the assumption of consistent parameter estimation when the dimension $d$ of the vine copula diverges, which has yet not been proven. 

Our main contribution is to provide theoretical results for the sequential estimator by \citet{Aas2009} when $p \to \infty$.
The necessary background on (vine) copulas is presented in \cref{sec:background}.
In \cref{sec:TheoResults}, we establish consistency and asymptotic normality of the sequential estimator, including both parametric and nonparametric estimation of the margins and a discussion of the required assumptions.
Although our primary motivation lies in the estimation of parametric, simplified vines, the results are also applicable to non-simplified vines and nonparametric approaches involving splines for example.
Moreover, the results are not specific to vine copulas and can also be applied to other sequential procedures as well as general estimation problems.
The theoretical results are complemented by a simulation study in \cref{sec:Sim}.
The article concludes with a discussion and some final remarks in \cref{sec:conclusion}.
All proofs can be found in the Appendix, while the supplement contains both additional theoretical considerations and a numerical validation of the assumptions as well as additional figures for the simulation study.
 
\section{Background}

\label{sec:background}

\subsection{Copulas}

Consider an \textcolor{red}{absolutely} continuous random vector $\bX = (X_1, \ldots, X_d) \in \Xcal \subseteq  \R^d$ with joint distribution $F$ and marginal distributions $F_1, \ldots, F_d$. 
Sklar's theorem \citep{Sklar} states that any multivariate distribution can be expressed as
\[
F(x_1, \ldots, x_d) = C\left(F_1(x_1), \ldots, F_d(x_d) \right),
\]
where $C \colon [0,1]^d \mapsto [0,1]$ is a \emph{copula} that captures the dependence structure among the components of $\bX$.
A copula is a multivariate distribution function on $[0,1]^d$ with uniform marginals \citep{Czado2019}.
\textcolor{red}{Assuming the joint density $f$ exists everywhere}, it can be written as $f(x_1, \ldots, x_d) = c\lf(F_1(x_1), \ldots, F_d(x_d)\ri) f_1(x_1) \cdots f_d(x_d)$,
where $c(u_1, \ldots, u_d) \coloneqq \frac{\partial^d}{\partial u_1 \cdots \partial u_d} C(u_1, \ldots, u_d)$ denotes the \emph{copula density} and $f_1, \ldots, f_d$ are the marginal densities \citep{Czado2019}.

\subsection{Regular Vines and Vine Copulas}

Vine copulas provide a flexible framework for modeling multivariate dependence structures by decomposing a copula into multiple bivariate copulas (also referred to as \emph{pair copulas}) using conditioning.
For a comprehensive introduction, we refer to \citet{Czado2019}, while \citet{CzadoNagler} provide a concise overview.
The underlying idea goes back to \citet{Joe96} and was formalized by \citet{Bedford2001, Bedford2002}:
any valid factorization of a multivariate density into marginal densities and bivariate conditional densities can be represented as a \emph{regular vine (R-vine) tree structure}.
Informally, an R-vine on $d$ elements is a nested sequence of trees $\Vcal = (T_1, \ldots, T_{d-1})$, where the edges of each tree become the nodes of the subsequent tree and $T_1$ has node set $N_1 = \{ 1, \ldots, d\}$.
\textcolor{red}{In addition, the \emph{proximity condition} must be satisfied, which states that an edge between nodes $a, b$ in $T_j$ is only possible if $a$ and $b$ (which are edges in $T_{j-1}$) share a common node in $T_{j-1}$ \citep{Czado2019}.}
This condition ensures the validity of the following pair copula construction:

\begin{defn}[R-vine specification \citep{Bedford2002}]
A \emph{regular vine} (\emph{R-vine}) specification is a triplet $(\Fcal, \Vcal, \Bcal)$, where
\begin{enumerate}
\item $\Fcal = (F_1, \ldots, F_d)$ is a vector of continuous invertible distribution functions.
\item $\Vcal$ is an R-vine tree sequence on $d$ elements with edge sets $E_t, t = 1, \ldots, d_n - 1$.
\item \textcolor{red}{$\Bcal = \{ \Bcal_e: e \in E_t, t = 1, \ldots d-1\}$, where $\Bcal_e$ is a collection $\{ C_e(\cdot, \cdot; \bx ) : \bx \in \R^{t-1}\}$ of bivariate copulas.}  
\end{enumerate}
\end{defn}
In the R-vine tree sequence, each edge $e$ can be labeled as $e = (a_e, b_e ; D_e)$, where $a_e$ and $b_e$ are the \emph{conditioned variables} and $D_e$ is the \emph{conditioning set}.
Denote by $c_{a_e, b_e; D_e}(\cdot, \cdot ;\bx_{D_e})$ the density of the copula $C_e(\cdot, \cdot ;\bx_{D_e})$, also written as $C_{a_e, b_e ; D_e}$.

It was shown by \citet{Bedford2002} that for a triplet $(\Fcal, \Vcal, \Bcal)$ satisfying the above conditions, there exists a valid $d$-dimensional distribution $F$ with density 
\[
f(x_1, \ldots, x_d) =  f_1(x_1) \cdots f_d(x_d) \prod_{t = 1}^{d-1} \prod_{e \in E_t} c_{a_e, b_e; D_e} \left(F_{a_e | D_e}(x_{a_e} | \bx_{D_e}), F_{b_e|D_e}(x_{b_e} | \bx_{D_e}); \bx_{D_e} \right),
\]
such that for each edge $e\in E_t,t = 1, \ldots, d-1$, the copula of the bivariate conditional distribution $(X_{a_e}, X_{b_e}) \mid \bX_{D_e} = \bx_{D_e}$ is given by $C_{a_e, b_e ; D_e}$, i.e.,
\[
F_{a_e, b_e | D_e}(x_{a_e}, x_{b_e} | \bx_{D_e}) = C_{a_e, b_e ; D_e}\left( F_{a_e | D_e}(x_{a_e} | \bx_{D_e}), F_{b_e|D_e}(x_{b_e} | \bx_{D_e}) ; \bx_{D_e}\right).
\]
Furthermore, the one-dimensional margins of $F$ are given by $F_1, \ldots , F_d$.

One often employs the so-called \emph{simplifying assumption}: the copula $C_{a_e, b_e ; D_e}$ does not depend on the specific value of $\bX_{D_e}$. 
Since our theory is not limited to simplified vines, we consider the general, non-simplified case in the theoretical part of the paper.

Frequently, interest lies primarily in the copula density, i.e., the density of the random vector $\bU =  F(\bX) = (F_1(X_1), \ldots, F_d(X_d))$, which is given by
\[
c(\bu) = \prod_{t = 1}^{d-1} \prod_{e \in E_t} c_{a_e, b_e; D_e} \left(u_{a_e | D_e}, u_{b_e|D_e}; \bu_{D_e} \right) ,
\]
where $u_{a_e | D_e} = C_{a_e | D_e}(u_{a_e} | \bu_{D_e})$, and $C_{a_e | D_e}$ denotes the conditional distribution of $U_{a_e}$ given $\bU_{D_e}$.
The aforementioned proximity condition guarantees that these distributions can be computed recursively.

Two common special cases of R-vine structures are illustrated in \cref{fig:c_d_vine}.
In a \emph{drawable} or \emph{D-vine} (\textbf{a}), each tree is a path.
Due to the proximity condition, the entire tree structure is determined by the first tree $T_1$, which, for example, may consist of the edges $(1,2), (2,3), \ldots, (d-1, d)$.
The second special case is the \emph{canonical} or \emph{C-vine} (\textbf{b}), in which each tree has a root node that is connected to all other nodes.

\begin{figure}
    \centering
\renewcommand{\thesubfigure}{\textbf{a}}
    \subfigure[]{
    \begin{tikzpicture}[scale=0.85]
        \tikzstyle{every node}=[ellipse, minimum width=20pt,align=center,scale=0.68]
        \node[shape=ellipse,draw=white] (0) at (0.3, 5.5) {};
        \node[shape=ellipse,draw=black,fill=black!15] (1) at (0.5, 5.5) {$1$};
        \node[shape=ellipse,draw=black,fill=black!15] (2) at (2, 5.5) {$2$};
        \node[shape=ellipse,draw=black,fill=black!15] (3) at (3.5, 5.5) {$3$};
        \node[shape=ellipse,draw=black,fill=black!15] (4) at (5,5.5) {$4$};

        \node[shape=ellipse,draw=black,fill=black!15] (6) at (1.25,4.5) {$1,2$};
        \node[shape=ellipse,draw=black,fill=black!15] (7) at (2.75,4.5) {$2,3$};
        \node[shape=ellipse,draw=black,fill=black!15] (8) at (4.25,4.5) {$3,4$};

        \node[shape=ellipse,draw=black,fill=black!15] (10) at (2,3.5)     {$1,3;2$};
        \node[shape=ellipse,draw=black,fill=black!15] (11) at (3.5,3.5)     {$2,4;3$};

        \path [-] (1)  edge node[above] {$1,2$} (2);
        \path [-] (2)  edge node[above] {$2,3$} (3);
        \path [-] (3)  edge node[above] {$3,4$} (4);

        \path [-] (6)  edge node[above] {$1,3;2$} (7);
        \path [-] (7)  edge node[above] {$2,4;3$} (8);

        \path [-] (10)  edge node[above=0.1cm] {$1,4;2,3$} (11);
    \end{tikzpicture}} 
\renewcommand{\thesubfigure}{\textbf{b}}
    \subfigure[]{
    \begin{tikzpicture}[scale=0.85]
        \tikzstyle{every node}=[ellipse, minimum width=20pt,align=center,scale=0.68]
        \node[shape=ellipse,draw=white] (0) at (0.3, 5.5) {};
        \node[shape=ellipse,draw=black,fill=black!15] (1) at (0.5, 5.5) {$1$};
        \node[shape=ellipse,draw=black,fill=black!15] (2) at (5, 6) {$2$};
        \node[shape=ellipse,draw=black,fill=black!15] (3) at (5, 5.5) {$3$};
        \node[shape=ellipse,draw=black,fill=black!15] (4) at (5,5) {$4$};

        \node[shape=ellipse,draw=black,fill=black!15] (6) at (1.25,4.5) {$1,3$};
        \node[shape=ellipse,draw=black,fill=black!15] (7) at (2.75,4.5) {$1,2$};
        \node[shape=ellipse,draw=black,fill=black!15] (8) at (4.25,4.5) {$1,4$};

        \node[shape=ellipse,draw=black,fill=black!15] (10) at (2,3.5)     {$2,3;1$};
        \node[shape=ellipse,draw=black,fill=black!15] (11) at (3.5,3.5)     {$2,4;1$};

        \path [-] (1)  edge node[at end, above left, pos = 0.8] {$1,2$} (2);
        \path [-] (1)  edge node[at end, above left =-0.05cm, pos = 0.85] {$1,3$} (3);
        \path [-] (1)  edge node[at end, above left, pos = 0.8] {$1,4$} (4);

        \path [-] (6)  edge node[above] {$2,3;1$} (7);
        \path [-] (7)  edge node[above] {$2,4;1$} (8);

        \path [-] (10)  edge node[above=0.1cm] {$3,4;1,2$} (11);
    \end{tikzpicture}}
    \hfill
      \caption{Graphical representations of two R-vine structures in four dimensions: D-vine (\textbf{a}) and C-vine (\textbf{b}).
  Each edge in the three trees denotes a pair copula.
  }
  \label{fig:c_d_vine}
\end{figure}

\section{Theoretical Results on Stepwise Maximum Likelihood Estimation in Vine Copulas}
\label{sec:TheoResults}
\subsection{Known Margins}

\label{sec:EstimKnownM}

We first assume that the margins $F_1, \ldots, F_d$ are known.
This makes it easier to study the assumption on the vine copula whose parameters are estimated.

Given parametric models $c_e(\cdot, \cdot; \btheta_e)$ for all edges $e$ in the vine, let $\btheta = (\btheta_e)_{e \in E_t, t = 1, \ldots, d-1}$ be the stacked parameter vector.
\textcolor{red}{Denote by $E_{<t}$ the set of edges from trees $T_j, j < t$ and define $\btheta_{E_{<t}} = (\btheta_e)_{e \in E_{<t}}$.}
The log-likelihood is then given by
\begin{align*}
\ell(\btheta; \bu)  & =  \sum_{t= 1}^{d-1} \sum_{e \in E_t} \ln c_{a_e, b_e; D_e} \left(C_{a_e | D_e}(u_{a_e} | \bu_{D_e};  \btheta_{E_{<t}}), C_{b_e | D_e}(u_{b_e} | \bu_{D_e};  \btheta_{E_{<t}});  \bu_{D_e},  \btheta_e\right) .
\end{align*}
The vine tree structure ensures that all parameters involved in the recursive computation of the conditional distribution functions belong to previous trees and therefore to a subvector of $ \btheta_{E_{<t}}$.

Given an \textit{iid} sample $\bU_1, \ldots, \bU_n$ from this model, $\btheta$ can in principle be estimated by maximizing the log-likelihood.
However, due to the recursive structure of the conditional copula distributions $C_{a_e | D_e}$, the MLE $\hbtheta^{\text{ML}}$ is computationally infeasible in high dimensions \citep{Haff13}.
The solution is a sequential approach, the stepwise MLE \citep{Aas2009, Haff13}: 
The parameters of each bivariate copula are estimated separately, starting with the first tree $T_1$ and proceeding to tree $T_{d-1}$:
\[
\hbtheta_e =  \arg \max_{\btheta_e} \sumin  \ln c_{a_e, b_e; D_e} \left(C_{a_e | D_e}(U_{i, a_e} | \bU_{i, D_e}; \hbtheta_{E_{<t(e)}}), C_{b_e|D_e}(U_{i, b_e} | \bU_{i, D_e}; \hbtheta_{E_{<t(e)}}); \bU_{D_e}, \btheta_e\right),
\]
where $t(e)$ denotes the tree containing $e$.
\textcolor{red}{If the stepwise MLE with known margins does not lie at the boundary of the parameter space, it solves the system of equations}
\begin{align}
\label{eq:estim1}
\sumin \phi(\bU_i;  \hbtheta) = \bnull, \quad \text{where} \quad \phi(\bu; \btheta) \coloneqq
\begin{pmatrix}
\Big( \nabla_{\btheta_e}  \ln  c_{a_e, b_e} (\bu, \btheta_e) \Big)_{e \in E_{1}}  \\
\Big( \nabla_{\btheta_e}  \ln c_{a_e, b_e; D_e} (\bu, \btheta_{E_{<t(e)}}, \btheta_e) \Big)_{e \in E_{2}}  \\
\vdots \\
\Big( \nabla_{\btheta_e}  \ln  c_{a_e, b_e; D_e} (\bu, \btheta_{E_{<t(e)}}, \btheta_e)  \Big)_{e \in E_{d-1}}  
\end{pmatrix},
\end{align}
where $\ln c_{a_e, b_e; D_e}(\bu, \btheta_{E_{<t(e)}} , \btheta_e) =  \ln c_{a_e, b_e; D_e} (C_{a_e | D_e}(u_{a_e} | \bu_{D_e};  \btheta_{E_{<t(e)}}), C_{b_e|D_e}(u_{b_e} | \bu_{D_e};  \btheta_{E_{<t(e)}});  \bu_{D_e}, \btheta_e )$.
\textcolor{red}{From now on, we consider $\hbtheta_U$ that solves $\sumin \phi(\bU_i;  \hbtheta_U) = \bnull$} and define the pseudo-true value $\btheta^*$ as the solution to $\E[ \phi(\bU; \btheta^*) ] = \bnull$.
If the model is correctly specified, these pseudo-true values coincide with the true parameters.

\subsubsection{Consistency}

\label{sec:EstimKnownMAsymp}

We now allow the number of parameters $p_n$ to diverge as the sample size $n$ tends to $\infty$.
Denote $\btheta = (\theta_1, \ldots, \theta_{p_n}) \in \R^{p_n}$.
If each pair copula has one parameter, $p_n$ coincides with the number of pair copulas, i.e., $p_n = d_n (d_n - 1)/2$, however, our results are formulated in terms of $p_n$ and make no explicit assumption on the relation between $d_n$ and $p_n$.
We assume that the entries of $\btheta$ are ordered as $\btheta = ((\btheta_e)_{e \in E_1}, (\btheta_e)_{e \in E_2}, \ldots, (\btheta_e)_{e \in E_{d_n - 1}})$ and that the $k$-th entry in $\phi$, denoted by $\phi(\bu; \btheta)_k$, is the derivative of the respective copula density with respect to $\theta_k$.
Define the $p_n \times p_n$ matrices
\[
I(\btheta) \coloneqq \text{Cov}[\phi(\bU; \btheta)], \quad J(\btheta)  \coloneqq \nabla_{\btheta} \E[\phi(\bU; \btheta)] = \left( \frac{\partial}{\partial \theta_j } \E[\phi(\bU; \btheta)_k] \right)_{k, j = 1, \ldots, p_n}.
\]
We assume that derivative and expectation can be interchanged, so $J(\btheta)  = \E [\nabla_{\btheta} \phi(\bU; \btheta)]$.
Due to the hierarchical structure of the model, $J(\btheta)$ is a lower block triangular matrix and, more precisely, a lower triangular matrix if each pair copula has a single parameter.

The parameter $\btheta$, functions $\phi(\btheta), I(\btheta), J(\btheta)$ and the support and distribution of the $\bU_i$ all depend on $n$, but we omit the index $n$ to improve readability.
Throughout the paper, $\| \cdot \|$ denotes the Euclidean norm for vectors, and the spectral norm  $\| A \| =\sup_{\| \bx \| = 1} \| A\bx\|$ for matrices.
$\| \cdot \|_{\infty}$ denotes the maximum norm $\| \bx \|_{\infty} = \max \{ |x_1| , \ldots, |x_p | \}$ and $\| \bx \|_0$ denotes the number of non-zero entries in $\bx$.

   \newcounter{assump}
   \renewcommand{\theassump}{(A\arabic{assump})}
   
   \newcounter{assumpstar}
\renewcommand{\theassumpstar}{(A\arabic{assump}*)}
   
\newcommand{\itemA}{\refstepcounter{assump}\item[\theassump]}

\newcommand{\itemStar}{\setcounter{assumpstar}{\value{assump}}\refstepcounter{assumpstar}\item[\theassumpstar]}
Denote $r_n = \sqrt{ \ln d_n / n}$ and let $\alpha_{j, n}, j = 1, \ldots, p_n$ be positive sequences that are bounded away from zero and denote $\alpha_n\coloneqq\max_{1 \le j \le p_n} \alpha_{j,n}$.
\textcolor{red}{Let $\Theta_n \subset \R^{p_n}$ be a sequence of sets such that there is a sequence of real numbers $C_n \to \infty$ such that $\Theta_n \supseteq \{ \btheta: | \theta_j - \theta^*_j | \le r_n C_n \alpha_{j,n} \forall j = 1, \ldots, p_n \}$.} 
Denote $\Theta_n^\Delta \coloneqq \{ \bDelta: |\Delta_j | \le \alpha_{j,n}  \forall j = 1, \ldots, p_n\}$.

\begin{enumerate}[label=(A\arabic*), resume=assump]

\itemA   \label{A:ConvRate_new} It holds that $\max_{k = 1, \ldots, p_n} \E[\phi(\bU; \btheta^*)^2_k] = O(1).$
  \itemA \label{A:phi-moments} 
  It holds that $ \Pr\Big(\|\phi(\bU; \btheta^*)\|_\infty >  \sqrt{\frac{\sigma_n^2 n}{4\ln p_n}} \Big) = o(1/n),$ where $\displaystyle  \sigma_n^2 = \max_{k = 1, \ldots, p_n} \E[\phi(\bU; \btheta^*)^2_k] $.
  \itemA \label{A:curvature}  There exists a $c > 0$ such that $\displaystyle \limsup_{n \to \infty}   \max_{1 \le j \le p_n} \sup_{ \bDelta \in \Theta_n^\Delta , | \Delta_j | = \alpha_{j,n}} (r_n C_n \alpha_{j,n})^{-1} \sign(\Delta_j) \, \E \lf[\phi(\bU; \btheta^* + r_n C_n \bDelta)_j \ri]  \le -c$.
  \itemA \label{A:emp_proc} There are sequences $M_n = o(\sqrt{n/(k_n + \ln p_n)})$ and $D_n = o(n/(k_n + \ln p_n))$ such that
  \begin{align*}
   \max_{1 \le j \le p_n} \sup_{\btheta  \in \Theta_n} \E \lf[ \lf(\sum_{k=1}^{p_n} \lf| \frac{\alpha_{k,n}}{\alpha_{j,n}} \frac{\partial}{\partial \theta_k} \phi(\bU; \btheta)_j \ri|  \ri)^2  \ri] \le M_n^2 \quad \text{and} \quad
   \P\lf(  \max_{1 \le j \le p_n} \sup_{\btheta \in \Theta_n}\sum_{k=1}^{p_n} \lf| \frac{\alpha_{k,n}}{\alpha_{j,n}} \frac{\partial}{\partial \theta_k} \phi(\bU; \btheta)_j \ri|    > D_n\ri) = o(1/n),
  \end{align*}
  where $k_n$ is a sequence such that
  $
\max_{1 \le j \le p_n} \sup_{ \btheta \in \Theta_n, \bu \in \R^{d_n}} \| \nabla_{\btheta} \phi(\bu; \btheta)_j\|_0 = k_n.
$
  \end{enumerate}
  Assumption \ref{A:phi-moments} is a tail condition which, together with \ref{A:ConvRate_new}, guarantees that $\| \frac 1 n \sumin \phi (\bU_i; \btheta^*) \|_{\infty} = O_p(\sqrt{\ln p_n / n})$ and imposes mild restrictions on the rate of growth of $p_n$.
  \textcolor{red}{By Markov's inequality and the union bound, the stronger condition $\max_{1 \le k \le p_n} \E[ \phi (\bU; \btheta^*)_k^4] = O(1)$ implies both \ref{A:ConvRate_new} and \ref{A:phi-moments} provided $p_n (\ln p_n)^2 / n \to 0$.}
  Assumption \ref{A:curvature} ensures identifiability and is discussed in more detail below.
  Since this is a population-level condition, \ref{A:emp_proc} guarantees sufficiently fast convergence of the empirical counterpart, thereby implicitly restricting the rate at which $p_n$ can grow.
  
\begin{theorem}[Consistency with known margins]
  \label{theorem_cons_new2}
  Under assumptions \ref{A:ConvRate_new}--\ref{A:emp_proc}, with probability tending to 1, the sets $\Theta_n$ contain at least one solution $\hbtheta_U$ of the estimating equation \eqref{eq:estim1} that satisfies
  \[
  \max_{1 \le j \le p_n} \frac{| \hat{\theta}_{U,j} - \theta^*_j |}{\alpha_{j,n}} =  O_p\lf( \sqrt{\frac{\ln p_n}{n}} \ri).
  \]
\end{theorem}
The estimator achieves the global rate of convergence $ \lVert \hbtheta_U - \btheta^* \rVert_{\infty} = O_p(\sqrt{\ln p_n / n})$ if $\alpha_n = \max_{1 \le j \le p_n} \alpha_{j,n}$ is bounded.
The sequences $\alpha_{j,n}$ allow for different rates of convergence for individual parameters:
While a single parameter in the first tree can be estimated at the standard $\sqrt{n}$ rate, parameters in higher trees may not achieve this rate, as their estimation depends on the noisy estimates from previous trees.
The strength of this dependence is controlled by \ref{A:curvature}:
For any $j$, it holds that
\begin{align*}
\sign(\Delta_j) \, \E \lf[\phi(\bU; \btheta^* + r_n C_n \bDelta)_j \ri]  
 = \sign(\Delta_j) \, \E \lf[\phi(\bU; \btheta^* + r_n C_n \bDelta)_j -\phi(\bU; \btheta^*)_j  \ri]  
 = \sign(\Delta_j) \, r_n C_n \bDelta^T \E\lf[\nabla_{\btheta} \phi(\bU; \bttheta)_j \ri]
\end{align*}
with some $\bttheta$ on the segment between $\btheta^*$ and $ \btheta^* + r_n C_n \bDelta$.
The sensitivity of the estimation of $\theta_j$ with respect to other parameters is encoded in $\nabla_{\btheta} \phi(\bU; \bttheta)_j $.
If the estimating equations $\phi(\bu; \btheta)$ consisted of independent estimation problems, only the $j$-th entry of $\nabla_{\btheta} \phi(\bU; \bttheta)_j $ would be non-zero and \ref{A:curvature} would simplify to $\E[\frac{\partial}{\partial \theta_j} \phi(\bU; \btheta)_j] \le -c < 0$ for all $j=1,\ldots, p_n$ and $\btheta$ in a neighborhood around $\btheta^*$.
In vine copulas however, estimation errors in earlier trees influence the estimation in subsequent trees.
Assumption \ref{A:curvature} guarantees that this sensitivity sufficiently weak so that each $\theta_j$ remains identifiable.
Pair copulas that are (close to) independence copulas facilitate estimation in later trees, since the influence of the estimation of $\theta_j^*$ on estimates in subsequent trees diminishes as the corresponding pair copula converges to the independence copula.
For simplicity, we assume from now on that $\theta^* = 0$ corresponds to the independence copula, as it is the case for Gaussian copulas. 

While $\alpha_{j,n} = 1$ is a suitable choice if $\theta^*_j \to 0$ sufficiently fast as $d_n \to \infty$,
\ref{A:curvature} together with $\alpha_{j,n} \to \infty$ allows for consistency results at the cost of a slower rate of convergence when the dependence in the vine decays more slowly.
Another special case is that of finite-dimensional vines:
As long as $\max_j \E[\frac{\partial}{\partial \theta_j}\phi(\bU; \btheta)_j] \le -c$ for some $c>0$, one can always choose sequences $\alpha_{j,n}$ such that \ref{A:curvature} is satisfied with $\alpha_n = O(1)$, since each  $\nabla_{\btheta} \phi(\bu; \btheta)_j $ has only finitely many non-zero entries.
The same holds for truncated vines with a fixed truncation level, provided that $\E[\frac{\partial}{\partial \theta_k} \phi(\bu; \btheta)_j]$ is uniformly bounded for all $j, k = 1, \ldots, p_n$, see also \cref{prop:trunc}.

In the supplement, we show results of a numerical validation of \ref{A:curvature}.
We set all true parameters in a given tree $t$ to the same value $\theta_t^*$.
The experiments suggest that, for Gaussian C- and D-vine with decreasing functions for $\theta_t^*$ such as $\theta_t^* = 0.5^t$, $\theta_t^* = 1/(t+1)$ and $\theta_t^* = 0.5 / \sqrt{t+1}$, \ref{A:curvature} is satisfied with $\alpha_{j,n} = 1$.
An exception is the D-vine with $\theta_t^* = 0.5 / \sqrt{t+1}$, where we need $\alpha_{j,n} = t_j$ (where $t_j$ denotes the tree of the $j$-th parameter) for \ref{A:curvature} to hold, resulting in the slower rate of convergence $\| \hbtheta - \btheta^* \|_{\infty} = O_p(\sqrt{d_n^2 \ln d_n / n})$.
\textcolor{red}{The supplement contains a more detailed discussion of \ref{A:curvature} for Gaussian D-vines and verifies (A1)--(A4) when all pair-copulas belong to the FGM family under the conditions $d_n^4 \ln d_n / n \to 0$ and $| \theta_t^* | \le 1/(t+1)^k, k > 1$.}

The main restrictions on the growth of $p_n$ are imposed by \ref{A:emp_proc}.
Similar to \ref{A:curvature}, the sequences $M_n$ and $D_n$ quantify the strength of dependence between the estimation of parameters from different trees.
Numerical experiments (see the supplement) suggest that $M_n^2 = O(p_n)$ and $D_n = O(p_n)$ for $\btheta^*$ small enough, which implies that $p_n^2/n \to 0$ is sufficient for consistency.

The sequence $k_n$ in \ref{A:emp_proc} can be interpreted as the maximum number of parameters affecting estimates in subsequent steps.
For an untruncated Vine, we typically have $k_n =O (p_n)$. 
For a vine truncated after tree $t_n$, there are $t_n d_n - t_n (t_n - 1)/2$ parameters to estimate, however, the estimation of any parameter $\theta_j$ depends on at most $O(t_n^2)$ parameters, thus $k_n = O(t_n^2)$.
 In particular, if the truncation level is fixed, we obtain $k_n = O(1)$.
 
If $\alpha_n = O(1)$, \cref{theorem_cons_new2} yields $\| \hbtheta_U - \btheta^* \|_2 = O_p( \sqrt{p_n \ln p_n / n})$, which is slightly worse than the optimal rate $\sqrt{p_n/ n}$.
Convergence at the optimal rate would follow from \citet[Theorem 1]{Gauss24}, however, this requires that the largest eigenvalue of $J(\btheta^*) + J(\btheta^*)^\top$ is negative and bounded away from zero.
This condition is not satisfied for many truncated C-vines (see Proposition 1 in the supplement).
We therefore establish convergence in $\| \cdot \|_\infty$ norm using the \emph{Poincar\'e-Miranda theorem}.
Notably, while the condition on $J(\btheta^*) + J(\btheta^*)^\top$ ignores the sequential estimation procedure, i.e., that $J(\btheta)$ is (block) triangular, this structure is exploited in \ref{A:curvature}.
\cref{theorem_cons_new2} is not specific to vines but can also be applied to other stepwise estimators as well as general estimating problems that can be expressed in the form $\sumin \phi (\bX_i; \hbtheta) = \0$ for some function $\phi$, e.g., a gradient.

\subsubsection{Asymptotic Normality}

Since the dimension of $\btheta$ increases with the sample size, asymptotic normality of $\hbtheta_U$ is formulated in terms of finite-dimensional projections.
Let $A_n \in \R^{q \times p_n}$ be some matrix and define $\tilde r_n =  \alpha_n\sqrt{p_n \ln p_n / n}$.

\begin{enumerate}[label=(A\arabic*), resume=assump]
  \itemA \label{A:Asymp} There exists a sequence $D_n = o(\sqrt{n}/ (\tilde r_n p_n))$ such that
        \begin{align*}
          \sup_{\| \bDelta \|, \| \bDelta' \| \le  \tilde r_n C_n}  \frac{\E\lf[ \| A_n [\phi(\bU; \btheta^* +  \bDelta) - \phi(\bU;\btheta^* +  \bDelta')]\|^2\ri]}{\| \bDelta - \bDelta'\|^2}                        & = o\lf( \frac{1}{\tilde r_n^2 p_n} \ri),     \\
          \Pr\lf( \sup_{\| \bDelta \|, \| \bDelta' \| \le \tilde r_n C_n} \frac{ \|A_n[\phi(\bU;\btheta^* +  \bDelta) - \phi(\bU;\btheta^* +  \bDelta')] \|}{\| \bDelta - \bDelta'\|}  >  D_n  \ri) & = o\lf(\frac 1 n \ri),                                 \\
          \sup_{ \bDelta \in \Theta_n^{\Delta}} \| A_n[J(\btheta^* +  r_n C_n  \bDelta) - J(\btheta^*)] \|    & = o\lf(\frac{1}{\sqrt{n} \tilde r_n}\ri).
        \end{align*}
  \itemA  \label{A:Asymp2} It holds that $\displaystyle \E \left[ \| A_n \phi(\bU; \btheta^*)\|^4 \right] =o(n).$
\end{enumerate}

Assumption \ref{A:Asymp} is a stochastic smoothness condition required to control fluctuations of the estimating equation and implicitely restricts the rate of growth of $p_n$.
\textcolor{red}{For example, if the left-hand side of the first equation is bounded, \ref{A:Asymp} requires $\tilde r_n^2 p_n \to 0$, i.e., $\alpha_n^2 p_n^2 \ln p_n / n \to 0$.}
The moment condition \ref{A:Asymp2} typically requires that $p_n^2/n \to 0$, e.g., if $\| A_n \| = O(1)$ and $\max_k \E[\phi(\bU; \btheta^*)_k^4] = O(1)$.
\begin{theorem}[Asymptotic normality with known margins]
  \label{theorem2}
  If conditions \ref{A:ConvRate_new}--\ref{A:Asymp2} hold for some matrix $A_n \in \R^{q \times p_n}$ for which $\Sigma = \lim_{n \to \infty} A_n I(\btheta^*)A_n^\top$ exists, $\hbtheta_U$ in \cref{theorem_cons_new2} satisfies
  \[
    \sqrt{n} A_n  J(\btheta^*)(\hbtheta_U - \btheta^*) \to_d \Ncal(\0, \Sigma).
  \]
\end{theorem}
If $\btheta$ is finite-dimensional, choosing $A_n = J(\btheta^*)^{-1}$ yields $\sqrt{n}(\hbtheta_U - \btheta^*) \to_p \Ncal(\0, J(\btheta^*)^{-1} I(\btheta^*)J(\btheta^*)^{-\top})$.
In \cref{sec:Sim}, we present simulations of $\sqrt{n / p_n} \sum_{j=1}^{p_n} (\hat{\theta}_j - \theta^*_j)$ for different vine copula models.
This corresponds to the choice
\[
A_n = \frac{1}{\sqrt{p_n}} \begin{pmatrix} 1 & 1 & \ldots & 1 \end{pmatrix} J(\btheta^*)^{-1} \in \R^{1 \times p_n},
\text{ so} \quad
\Sigma = \lim_{n \to \infty} \frac{1}{p_n} \sum_{j=1}^{p_n} \sum_{k=1}^{p_n} \lf(  J(\btheta^*)^{-1} I(\btheta^*)J(\btheta^*)^{-\top}\ri)_{jk}.
\]

\textcolor{red}{For an FGM copula with $| \theta_t^* | \le 1/(t+1)^k, k > 1$, the conditions of \cref{theorem2} with hold if $d_n^6 (\ln d_n)^2 / n \to 0$ and $\| A_n \|_{\infty} = O(1)$, see Proposition 3 in the supplement.
The following proposition gives sufficient conditions for \cref{theorem_cons_new2} and \cref{theorem2} to hold for a 2-truncated vine.
Its proof, given in the appendix, illustrates how to derive moment conditions from the above stated assumptions.}
\textcolor{red}{\begin{proposition}\label{prop:trunc}
    Consider a $d_n$-dimensional 2-truncated vine\footnote{I.e., all pair copulas from the third tree onwards are set to independence copulas.} with $p_n$ parameters, arbitrary R-vine structure and families. Suppose $\max_{1 \le k \le p_n} \E[\phi(\bU; \btheta^*)_k^4] = O(1)$. Then, conditions (A1)--(A4) hold with sequences $\alpha_{j,n}$ with $\max_{1 \le j \le n} \alpha_{j,n} = O(1)$ if $p_n (\ln p_n)^2/n \to 0$ and
    \[
    \sup_{\btheta \in \Theta_n} \max_{1 \le k \le p_n} \E\left[ \frac{\partial}{\partial \theta_k} \phi(\bu; \btheta)_k \right] \le -c < 0, \qquad \max_{1 \le j, k \le p_n} \E \left[ \sup_{\btheta \in \Theta_n} \left( \frac{\partial}{\partial \theta_j}  \phi(\bU; \btheta)_k \right)^2 \right] = O(1).
    \]
    Conditions (A5)--(A6) hold if, additionally, $\| A_n \|_\infty = O(1)$, $p_n^2 (\ln p_n)^2 / n \to 0$, the number of non-zero terms in $\sum_{k=1}^{p_n} \frac{\partial}{\partial \theta_j} \phi(\bu; \btheta)_k$ is uniformly bouned, and
    \[
    \max_{1 \le j, k \le p_n} \E\left[ \sup_{\btheta \in \Theta_n} \left( \frac{\partial}{\partial \theta_j} \phi(\bU; \btheta)_k \right)^4 \right] =O(1), \qquad 
    \sup_{\btheta \in \Theta_n} \max_{1 \le j, k, l \le p_n} \E\left[ \frac{\partial^2}{\partial \theta_l \partial \theta_j} \phi(\bU; \btheta)_k \right] = O(1).
    \] 
\end{proposition}}
\textcolor{red}{The assumption on the number of non-zero terms in $\sum_{k=1}^{p_n} \frac{\partial}{\partial \theta_j} \phi(\bu; \btheta)_k$ is trivial for many truncated vine structures such as a D-vine, but excludes a C-vine. 
The proposition can easily be extended to $t$-truncated vines with finite $t$.}

\subsection{Parametric Estimation of Margins}

Assuming parametric models $F_l(x_l ;\bdeta_l)$, $l=1, \ldots, d$ for the marginals and denote $\bdeta = (\bdeta_1, \ldots, \bdeta_d)$, the log-likelihood of the model is given by
\begin{align*}
\ell(\bdeta, \btheta; \bx)  = & \sum_{l = 1}^d \ln  f_l(x_l; \bdeta_l)  + \sum_{t = 1}^{d-1} \sum_{e \in E_t} \ln c_{a_e, b_e; D_e} \left(F_{a_e | D_e}(x_{a_e} | \bx_{D_e}; \bdeta, \btheta_{S_a(e)}), F_{b_e|D_e}(x_{b_e} | \bx_{D_e}; \bdeta, \btheta_{S_b(e)}); \bx_{D_e},  \btheta_e\right) \\
 = & \ell_M(\bdeta; \bx)  + \ell_C(\bdeta, \btheta; \bx) .
\end{align*}
Both the joint MLE and the two-step \emph{inference for margins} approach by \citet{Joe05}, which maximizes $\ell_M$ and $\ell_C$ sequentially, are computationally too demanding for practical purposes.
One therefore again resorts to a sequential approach, slightly modifying the stepwise MLE presented in \cref{sec:EstimKnownM}:
first, $\bdeta$ is estimated by maximizing $\sumin \ell_M(\hbdeta; \bX_i)$.
Then, the parameters of each bivariate copula are estimated separately, plugging in the estimate $\hbdeta$.
Similar to the previous section, stepwise MLE $(\hbdeta, \hbtheta_\eta)$ solves the solution to the system of equations 
\begin{align}
\label{eq:estimPar}
\sumin \phi(\bX_i; \hbdeta, \hbtheta_\eta) = \bnull, \quad \text{where} \quad
 \phi(\bx; \bdeta, \btheta) \coloneqq
\begin{pmatrix}
 \nabla_{\bdeta}   \sum_{j = 1}^d \ln  f_j(x_j; \bdeta_j)  \\
\Big( \nabla_{\btheta_e}  \ln  c_{a_e, b_e} (\bx, \bdeta, \btheta_e) \Big)_{e \in E_{1}}  \\
\Big( \nabla_{\btheta_e}  \ln c_{a_e, b_e; D_e} (\bx, \bdeta,\btheta_{S_{a,b}(e)}, \btheta_e) \Big)_{e \in E_{2}}  \\
\vdots \\
\Big( \nabla_{\btheta_e}  \ln  c_{a_e, b_e; D_e} (\bx, \bdeta,\btheta_{S_{a,b}(e)}, \btheta_e)  \Big)_{e \in E_{d-1}}  
\end{pmatrix}.
\end{align}
Consistency and asymptotic normality of $(\hbdeta, \hbtheta_\eta)$ directly follow from \cref{theorem_cons_new2} and \cref{theorem2} by replacing $\btheta$ by $(\btheta, \bdeta)$, $\phi(\bu; \btheta)$ by $\phi(\bx; \bdeta, \btheta)$ and $\Theta_n$ by a sequence of sets $\tilde \Theta_n \subset \R^{p_n}$ satisfying $\tilde \Theta_n \supset \{ (\bdeta, \btheta): | (\bdeta, \btheta)_j - (\bdeta^*, \btheta^*)_j | \le r_n C_n \alpha_{j,n} \forall j = 1, \ldots, p_n \}$ for a sequence $C_n \to \infty$.

 \subsection{Nonparametric Estimation of Margins}

Parametric estimation of the marginals is often not the primary focus, or one may prefer an estimator of the copula parameters that is robust to misspecification of the margins.
In the semiparametric approach, first proposed by \citet{Genest95}, the marginals $F_l(x_l; \bdeta_l)$ are replaced by the empirical distribution functions $\hat{u}_l = F_{n l}(x_l) = (n+1)^{-1} \sumin \ind (X_{i, l} \le x_l)$.
Define $ \hbu = F_n(\bx) = (F_{n1}(x_1), \ldots, F_{n d}(x_d))$.
In the semiparametric stepwise MLE introduced by \citet{Aas2009}, $\bU_i$ in \eqref{eq:estim1} is replaced by $\hbU_i$, i.e., we define the estimator $\hbtheta_X$ as the solution to 
\begin{equation}
\label{eq:semiPar}
\sumin \phi(\hbU_i; \hbtheta_X) = \bnull
\end{equation}
with $\phi(\bu, \btheta)$ as defined in \eqref{eq:estim1}.
The pseudo-true value $\btheta^*$ remains the solution to $\E[\phi(\bU;  \btheta^*) ] = \bnull$.
\citet{Haff13} showed that, in finite dimensions, this estimator is consistent and asymptotically normal.

We again consider a diverging number of parameters $p_n$.
The notation is the same as in \cref{sec:EstimKnownMAsymp}, i.e., the expectations in the definitions of $I(\btheta)$ and $J(\btheta)$ and in \ref{A:ConvRate_new}--\ref{A:emp_proc} are still taken with respect to $\bU = F(\bX)$.
Denote
  \begin{equation}
\label{eq:def_Y}
\textcolor{red}{h^a(\bx, \tilde \bx) \coloneqq  \underbrace{\nabla_{\bu} \phi(F(\bx); \btheta^*) }_{\in \R^{p_n \times d_n}}
 \begin{pmatrix}
 \ind (\tilde x_1 \le x_1) - F_1(x_1) \\
 \vdots \\
  \ind (\tilde x_{d_n} \le x_{d_n}) - F_{d_n}(x_{d_n})
 \end{pmatrix} ,  \quad h(\bx, \tilde \bx) \coloneqq\frac 1 2 h^a(\bx, \tilde \bx) + \frac 1 2 h^a(\tilde \bx, \bx), \quad  h_1(\bx) \coloneqq \E[h(\bx , \bX) ].}
  \end{equation}
We require some additional assumptions to account for the estimation of $F(\bX)$:

\begin{enumerate}[label=(A\arabic*), resume=assump]
\itemA \label{A:Var_h1} Denote $\sigma_n^2 = \max_{1 \le k \le p_n} \E[h_1(\bX)_k^2 ]$.
It holds that $\sigma_n^2 = O(1)$ and
$
\P \lf( \|h_1(\bX) \|_{\infty} > \sqrt{\frac{\sigma_n^2 n}{4 \ln p_n}}\ri) = o(1/n).
$
\itemA \label{A:AsympSemi_2} It holds that \textcolor{red}{$\max_{1 \le k \le p_n; 1 \le i, i' \le n} \E[ h(\bX_i, \bX_{i'})^2_k ] = O(1)$.} \itemA \label{A:SemiPar_I2} Define
    \begin{align*}
\Gcal_n \coloneqq \Big\{ & G \colon \R^{d_n} \to [0,1]^{d_n}, G(\bx) = \lf(G_1(x_1), \ldots, G_{d_n}(x_{d_n}) \ri) \mid   \text{$G_m, m = 1, \ldots, d_n$ is a continuous distribution function} \\ &\text{and}
  \sup_{x \in \R, 1\le m \le d_n} \frac{ |G_m(x) - F_m(x) |}{w(F_m(x))} \le C d_n^b \sqrt{\ln d_n / n} \quad \text{for all $C < \infty$ and $n$ large enough} \Big \}
\end{align*}
with $w(s) = s^\gamma (1-s)^\gamma$ for some $\gamma \in (0,1)$ and $b > \gamma$.
Denote $\partial_{ml} \phi(\bu; \btheta)_k = \partial^2/(\partial u_m \partial u_l) \ \phi(\bu; \btheta)_k$.
It holds that
\textcolor{red}{
\[
\E \lf[ \sup_{1 \le k \le p_n, G \in \Gcal_n } \left( \sum_{m=1}^{d_n} \sum_{l=1}^{d_n}  |   \partial_{ml} \phi(G(\bX); \btheta^*)_k w(F_m(X_m)) w(F_l(X_l))  |^2 \right)^{1/2} \ri] = O( \sqrt{d_n \ln p_n}).
\]}
\end{enumerate}
\textcolor{red}{The set $\Gcal_n$ of distribution functions characterizes the convergence of the empirical margins to the true ones (see also \cref{lem:emp_cdfs}), the convergence is faster for $x$ with $F(x)$ close to 0 or 1.
Since derivatives of $\phi(\bu; \btheta)$ w.r.t.~$u_m$ tend to explode for $u_m \to 0$ or $u_m \to 1$, we exploit the faster convergence of $F_{nm}$ at the boundaries by multiplying the derivatives $ \partial_{ml} \phi(G(\bX); \btheta^*)_k$ with $w(s) = s^\gamma (1-s)^\gamma, \gamma \in (0,1)$.
Choosing $\gamma$ close to 1 relaxes the assumption on the derivatives, but leads to stronger conditions on $d_n$, as \cref{theorem_semiparam_new} requires $d_n^{(3 + 4b)}(\ln d_n)^2/n =O(1)$ with $b > \gamma$.
}
\begin{enumerate}[label=(A\arabic*), resume=assump]
\itemA \label{A:SemiPar_r} For all $m = 1, \ldots, d_n$, there exists a function $\psi_m \colon (0,1)^{d_n} \to \R^+_0$ such that for all $\bu \in  (0,1)^{d_n}, k  =1, \ldots, p_n$ and $\bDelta \in \R^{p_n}$, it holds that
\[
\lf|  \frac{\partial}{\partial u_m} \lf[\phi(\bu; \btheta^* + \bDelta)_k - \phi(\bu; \btheta^*)_k \ri] \ri| \le \| \bDelta \|_{\infty} \psi_m(\bu).
\]
With $\Gcal_n$ as defined in \ref{A:SemiPar_I2}, it holds that \textcolor{red}{
$
\E\left[\sup_{G \in \Gcal_n} \lf\| \left( w(F_m(X_{m})) \psi_m(G(\bX))\right)_{m=1,\ldots, d_n} \ri\|_2 \right] = O(\sqrt{d_n}) .
$}
\itemA \label{A:alpha_max} Denote $\tilde I_{n,j}$ the set of indices $k$ such that $\partial \phi(\bx; \bu)_j / \partial \theta_k = 0$ for all $\btheta$ and $\bu$ and denote $I_{n,j} = \{ 1, \ldots, p_n \} \setminus \tilde I_{n,j}$. It holds that \textcolor{red}{\[
\max_{1 \le j \le p_n} \frac{\max_{k \in I_{n,j} } \alpha_{k,n}}{ \alpha_{j,n}} = O(1).
\]}
\end{enumerate}
In a vine where each pair copula has a single parameter, we have $I_{n,j} \subset \{ 1, \ldots, j\}$ in \ref{A:alpha_max}.
\begin{theorem}[Consistency with nonparametric estimation of margins]
  \label{theorem_semiparam_new}
  Suppose assumptions \ref{A:ConvRate_new}--\ref{A:alpha_max} hold and \textcolor{red}{$d_n^{(3 + 4b)}(\ln d_n)^2/n =O(1)$, where $b$ is defined in \ref{A:SemiPar_I2}.} Then, with probability tending to 1, the sets $\Theta_n$ contain at least one solution of the estimating equation \eqref{eq:semiPar} that satisfies
          \[
          \max_{1 \le j \le p_n } \frac{|  \hat{\theta}_{X,j} - \theta^*_j|}{\alpha_{j,n}}= O_p\lf( \sqrt{\frac{\ln p_n}{n}}\ri).
          \]
\end{theorem}

\begin{enumerate}[label=(A\arabic*), resume=assump]
    \itemA \label{A:SemiPar2} With $\Gcal_n$ and $\partial_{ml} \phi (\bu; \btheta)_k$ as defined in \ref{A:SemiPar_I2}, holds that \textcolor{red}{
    \[
    \E \lf[ \sup_{G \in \Gcal_n } \left( \sum_{m=1}^{d_n} \sum_{l=1}^{d_n} \sum_{k=1}^{p_n}  | \partial_{ml} \phi(G(\bX); \btheta^*)_k w(F_m(X_m)) w(F_l(X_l))  |^2 \right)^{1/2} \ri] = O(\sqrt{p_n}).
    \]}
  \itemA  \label{A:Asymp2_margins} With $h_1(\bx)$ as defined in \eqref{eq:def_Y}, it holds that $\displaystyle \E \left[ \| A_n h_1(\bX)\|^4 \right] =o(n).$
\end{enumerate}

\noindent Let $C$ be the true copula distribution of $\bX$, i.e., the distribution of $\bU = F(\bX)$.

\begin{theorem}[Asymptotic normality with nonparametric estimation of margins]
  \label{theorem2_semiparam_new}
  Suppose assumptions \ref{A:ConvRate_new}--\ref{A:Asymp2_margins} hold with some matrix $A_n \in \R^{q \times p_n}$ with $\| A_n \| = O(1)$ and for which $\Sigma_X = \lim_{n \to \infty} A_n \widetilde \Sigma_n A_n^\top$ exists, where
   \[
   \widetilde \Sigma_n = \cov \lf( \phi(\bxi; \btheta^*) +  \int \sum_{l=1}^{d_n}  \frac{\partial}{\partial u_l} \phi(\bu ; \btheta^*) ( \ind ( \xi_l \le u_l) - u_l) d C(\bu) \ri),
   \]
    and $\bxi$ is a random variable with distribution $C$. 
    Suppose that \textcolor{red}{$\alpha_n^2 p_n d_n^{(2 + 2b)} \ln p_n \ln d_n /n \to 0$ and $d_n^{(2 + 4b)} p_n(\ln d_n)^2 /n \to 0$.}
    Then, $\hbtheta_X$ in \cref{theorem_semiparam_new} satisfies
  \[
    \sqrt{n} A_n  J(\btheta^*)(\hbtheta_X - \btheta^*) \to_d \Ncal(\0, \Sigma_X).
  \]
\end{theorem}
For finite dimensions, choosing $A_n = J^{-1}(\btheta^*)$, \cref{theorem2_semiparam_new} recovers the results from \citet{Tsuka} and \citet{Haff13}.
Compared to the estimation with known margins in \cref{theorem2}, an additional term appears in the covariance matrix stemming from the increased uncertainty due to the estimation of margins.
\textcolor{red}{We furthermore require additional assumptions on the sensitivity of $\phi(\bu; \btheta)$ with respect to $\bu$ to ensure convergence of expressions involving $\phi(\widehat{\bU}_i ; \btheta)$ to their population counterparts $\phi(\bU_i ; \btheta)$.
This is mainly achieved by bounding sums over first or second partial derivatives of $\phi(\bu; \btheta^*)$ with respect to $\bu$, either implicitly (\ref{A:Var_h1}, \ref{A:AsympSemi_2}, \ref{A:SemiPar_r}) or explicitly (\ref{A:SemiPar_I2}, \ref{A:SemiPar2}).
As in \citet{Tsuka} and \citet{Haff13}, the derivatives are attenuated by the `U-shaped' weight function $w(s) = s^\gamma (1-s)^\gamma$ to exploit that empirical distributions have smaller errors near the boundaries, which counteracts potentially exploding derivatives.
For the independence case of a Gaussian vine, we have $U_{a_e | D_e} = U_{a_e}$ for each edge $e = (a_e, b_e ; D_e)$ and the estimation of the corresponding pair copula parameter $\theta_e$ only depends on $U_{a_e}$ and $U_{b_e}$, leading to at most four non-zero entries for each $k$ in the sums in \ref{A:SemiPar_I2} and \ref{A:SemiPar2} ($m = l = a; m = l = b; m = a, l = b$ and vice versa).
In a vine whose pair copulas in later trees are sufficiently close to independence (which ensures identifiability; see the discussion of \ref{A:curvature} in \cref{sec:EstimKnownMAsymp}), it is therefore reasonable to expect that sums of partial derivatives grow sufficiently slowly.
Furthermore, while the set of variables $U_j$ that affect the estimation of a parameter in the $t$-th tree increases with $t$ (since a pair copula in the $t$-th tree is affected by $t+1$ variables), the influence of the original $U_j$ on $U_{a_e | D_e}$ typically diminishes for most $j$ as $t$ grows.
The discussed conditions can be relaxed at the cost of stricter assumptions on $d_n$.
See the supplement for additional details on the derivatives $\partial \phi(\bu; \btheta)_k / \partial u_j$.}

\section{Simulation Study}\label{sec:Sim}

To evaluate the performance of the stepwise maximum likelihood estimator, we compare the estimates $\hbtheta_U$ and $\hbtheta_X$ to the true values $\btheta^*$ for several (simplified) vine copula models.
The structure is either a D- or a C-vine.
All bivariate copulas within a given vine are chosen from the same family, either Gaussian, Gumbel, or Student's t.
Gaussian vines provide a well-behaved benchmark with light tails, while Gumbel vines exhibit heavy tails and asymmetric dependence.
With Student's t, we include a copula family with a second parameter, in which $\rho = 0$ does not correspond to independence.
\textcolor{red}{As a second model with two parameters per pair copula, we include a 50--50 mixture of two rotated Gumbel copulas.}

For Gumbel, we chose the following parameterization: 
For $\theta \ge 0$, we use the Gumbel copula with parameter $\theta + 1$.
For $\theta < 0$, we use the 90-degree rotated Gumbel copula with parameter $- \theta + 1$.
This ensures that $\theta = 0$ (independence copula) does not lie on the boundary of the parameter space.
\textcolor{red}{For the mixture of two Gumbel pair copulas, we define the first component as $\text{Gumbel}(\theta_1 + 1)$ for $\theta_1 \ge 0$ and $\text{Gumbel}_{90}(-\theta_1+1)$ for $\theta_1 <0$, and the second component as $\text{Gumbel}_{180}(\theta_2 + 1)$ for $\theta_2 \ge 0$ and $\text{Gumbel}_{270}(-\theta_2+1)$ for $\theta_2 <0$. For simplicity, we take $\theta_1^* = \theta_2^*$.}
For the Student's t vines, the degrees of freedom parameter $\nu$ is set to $\nu^* = 4$ for all pair copulas.
Let $\theta^*_t$ denote the true parameter in the $t$-th tree.

The dimension $d$ ranges from 10 to 200, and the sample size $n$ ranges from 100 to 5000. 
For the true parameters, we choose the functions $\theta_t^* = 0.5^t$, $\theta_t^* = 1/(t+1)$ and $\theta_t^* = 0.5 / \sqrt{t+1}$.
For each configuration, $N = 100$ data sets ($N = 50$ for Student's t) of size $n$ are generated by sampling copula data $\bU_i$ or the untransformed data $\bX_i, i = 1, \ldots, n$.
We estimate $\hbtheta_U$ and $\hbtheta_X$ via stepwise ML and compute the normalized maximum norm of the error $\sqrt{n/\log(d_n)} \| \hbtheta - \btheta^* \|_{\infty}$, which should be bounded in probability if \cref{theorem_cons_new2} and \cref{theorem_semiparam_new} hold, as well as the normalized sum of errors $\sqrt{n}/d_n \sum_{k= 1}^{p}( \htheta_k - \theta_k^*)$.
If the assumptions of \cref{theorem2} and \cref{theorem2_semiparam_new} are satisfied and $n$ is sufficiently large, this sum should be approximately normally distributed with mean zero and variance independent of both $d$ and $n$.

When assessing consistency, the following question arises: Since the theoretical results concern the regime in which both $d$ and $n$ diverge, which pairs $(n,d)$ should be compared to each other with respect to $\sqrt{n/\log(d_n)} \| \hbtheta - \btheta^* \|_{\infty}$?
Our approach is as follows: 
We fix a set of values for $d$ and three different growth regimes for $d$ relative to $n$, namely $d \sim n$, $d \sim n^2$ and $d \sim n^3$.\footnote{The explicit functions are $n = 25 d$, $n = 0.125 d^2$ and $n = 0.003 (d-50)^3$ (for $d > 50$) for Gaussian and Gumbel vines, and $n = 0.005 d^3$ for Student's t.}
For each $d$ and each regime, we compute the corresponding sample size $n^{\text{new}}$.
Since we usually do not have $d$-dimensional estimates $\hbtheta^{(d)}_{n^{\text{new}}}$ based on samples of size $n^{\text{new}}$, we estimate the mean of the error $\| \hbtheta^{(d)}_{n^{\text{new}}} - \btheta^{*,(d)} \|_{\infty}$ using linear interpolation or, when extrapolation is required, a linear model fitted to the estimates using the two $n$ values that are closest to $n^{\text{new}}$.
Both interpolation and extrapolation are performed on the log scale, as for fixed $d$, we expect $ \| \hbtheta_n - \btheta^* \|_{\infty} \approx c \, n^a$, i.e., $\log  \| \hbtheta_n - \btheta^* \|_{\infty} = \log c + a \log n$.
In \cref{fig:par_gauss_gumbel_norm}, we show the estimated normalized errors for the three growth regimes. 
A stabilization of the values indicates that, for the given relation between $n$ and $d$, the parameter estimates are consistent in $\| \cdot\|_{\infty}$ norm.
 Note that the results for $n \sim d^3$ are only computed for $d \ge 75$ (for Gaussian and Gumbel vines): 
To avoid extreme extrapolation for $d = 200$, the constant in $n \sim d^3$ is so small that for small $d$, we obtain $n^{\text{new}}$ that are too small to yield reasonable estimates.

\begin{figure}
  \centering
  \includegraphics[width = \textwidth]{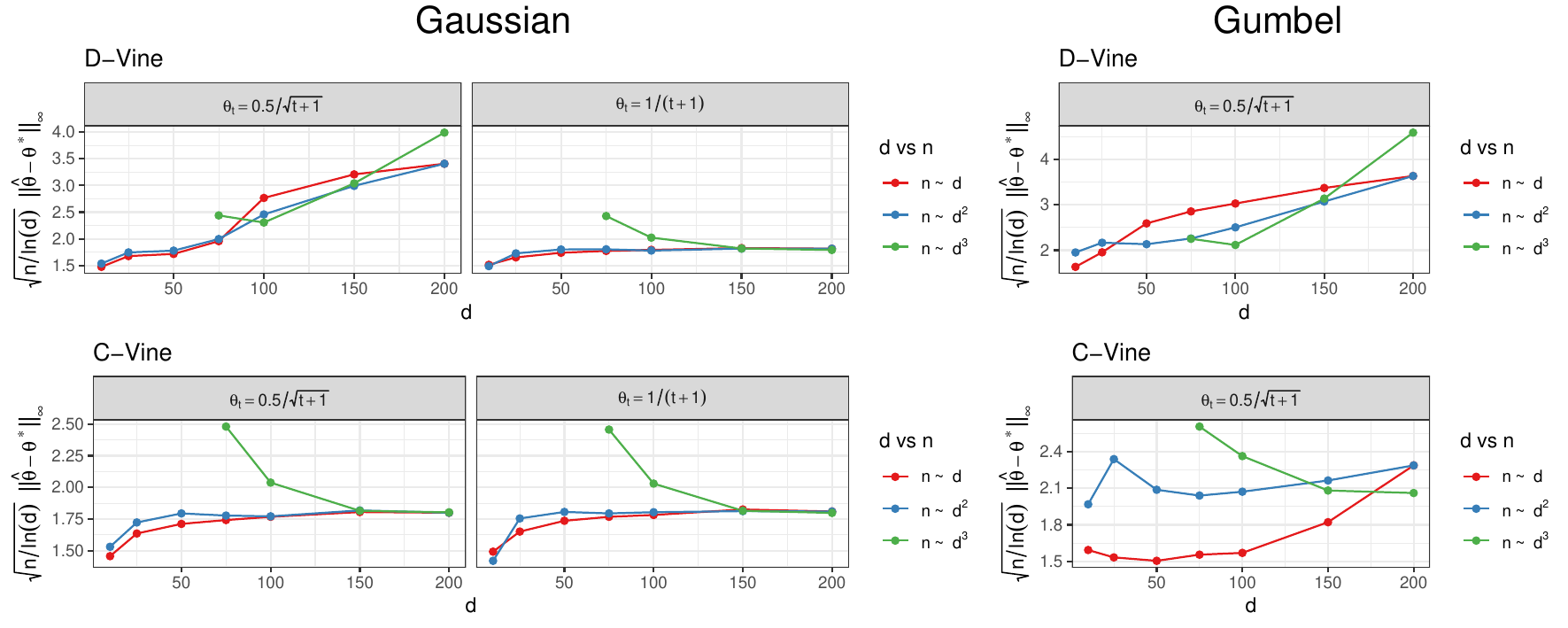}
  \caption{Parameter estimation for Gaussian and Gumbel vines, mean maximum norm of estimation error for different proportions of $d$ and $n$. Parameterization: $\theta = \rho$ for Gaussian and $\text{Gumbel}(\theta + 1)$ for $\theta \ge 0$, $\text{Gumbel}_{90}(-\theta+1)$ for $\theta <0$. }
  \label{fig:par_gauss_gumbel_norm}
\end{figure}

We first omit the estimation of margins, i.e., the true copula data $\bU_i = F(\bX_i)$ are used.
 
\begin{figure}
  \centering
  \includegraphics[width = 0.9\textwidth]{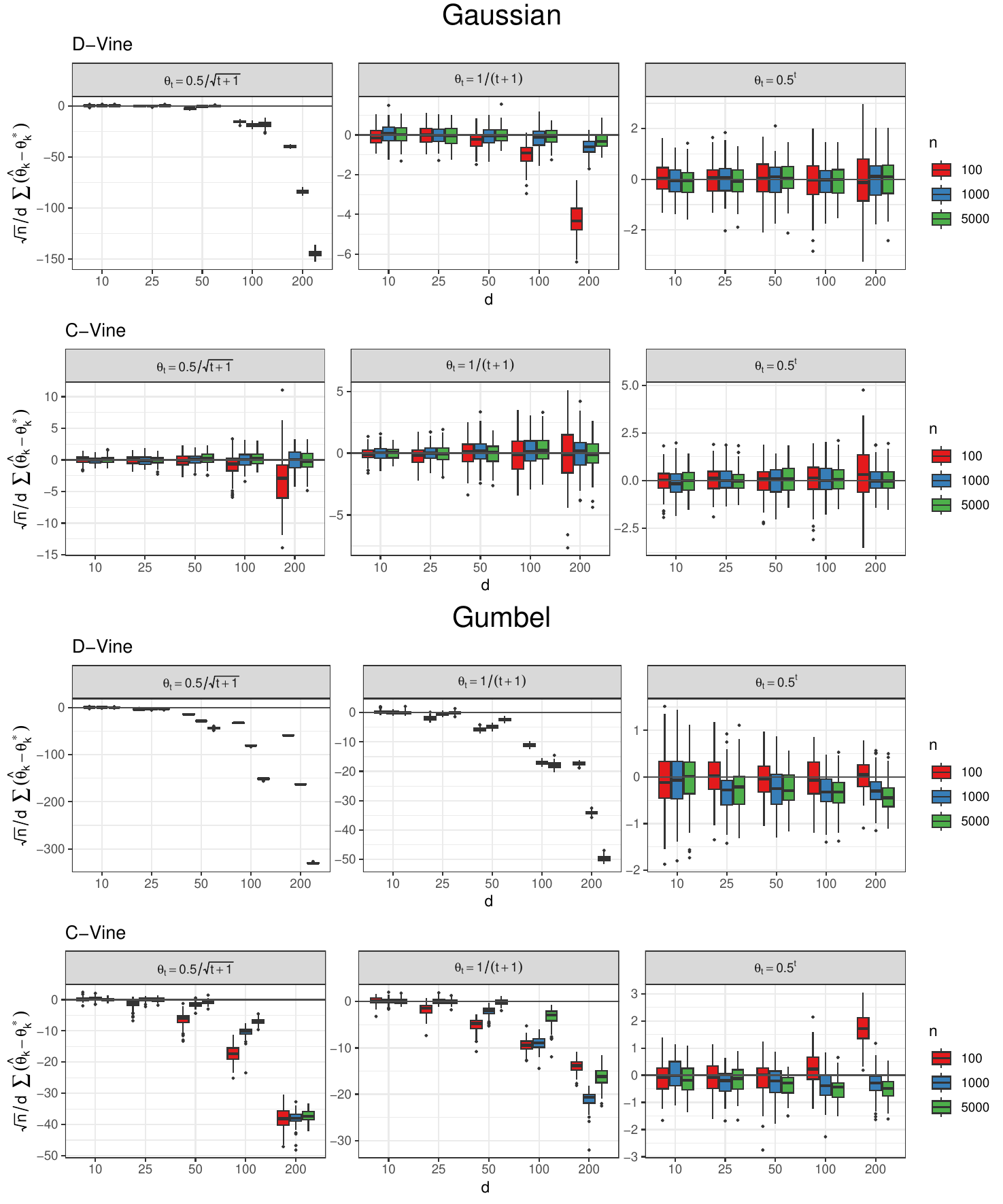}
  \caption{Parameter estimation for Gaussian and Gumbel vines, sum of estimation errors. Parameterization: $\theta = \rho$ for Gaussian and $\text{Gumbel}(\theta + 1)$ for $\theta \ge 0$, $\text{Gumbel}_{90}(-\theta+1)$ for $\theta <0$. Each boxplot represents $N=100$ replications.}
  \label{fig:par_gauss_gumbel}
\end{figure}

\myparagraph{Gaussian and Gumbel Vines}

 \cref{fig:par_gauss_gumbel_norm} shows the error norm for Gaussian and Gumbel vines for a selection of the models. 
 The results for $\theta^*_t = 1/(t+1)$ (Gumbel) and $\theta^*_t = 0.5^t$ (Gumbel and Gaussian) are similar to the Gaussian vines with $\theta^*_t = 1/(t+1)$ and therefore only shown in the supplement.
 
 For Gaussian vines, $\hbtheta_U$ appears to be consistent for all models even when $n \sim d$, except for the D-vine with $\theta_t = 0.5/\sqrt{t+1}$, which is in line with the numerical experiments on the assumptions in the supplement.
 Estimation in Gumbel vines mostly exhibits similar behavior, except for the C-vine with $\theta_t= 0.5/\sqrt{t+1}$, which seems to require $n \sim d^2$ or even $n \sim d^3$ for consistency. 
 
 Since \cref{theorem_cons_new2} also allows for convergence rates slower than $\sqrt{\log(d_n)/n}$ (in maximum norm), we computed $\| \hbtheta_U - \btheta^*\|_\infty$ with different normalizations for the Gaussian D-vine with $\theta_t = 0.5/\sqrt{t+1}$.
 The results are shown in the supplement.
 They suggest that, while the numerical experiments in the supplement ($\alpha(t) = t$, i.e., $\alpha_n = d_n$) imply the rate $\sqrt{d_n^2 \log(d_n)/n}$, already with rate $\sqrt{d_n \log(d_n)/n}$, the mean maximum norm is bounded.

The normalized sum of errors is shown in \cref{fig:par_gauss_gumbel} to assess the asymptotic normality of $\hbtheta_U$.
Recall from \cref{sec:EstimKnownMAsymp} that this requires stronger assumptions on $\phi$ and $d_n$ than consistency.
For Gaussian vines, $\hbtheta_U$ is asymptotically normal with mean zero in all models except D-vine, $\theta_t = 0.5/\sqrt{t+1}$.
For Gumbel vines however, $\hbtheta_U$ is heavily biased for both C- and D-vines with $\theta_t = 0.5/\sqrt{t+1}$ and $\theta_t= 1/(t+1)$, with larger biases for D-vines.
Even for $\theta_t = 0.5^t$, the results indicate a negative bias.
While $\hbtheta_U$ is a consistent estimator, a closer analysis shows that, while the magnitudes of positive and negative biases are approximately the same, negative biases (i.e., $\hat{\theta}_k < \theta_k^*$) occur far more often than positive ones.

\begin{figure}
  \centering
  \includegraphics[width = 0.85\textwidth]{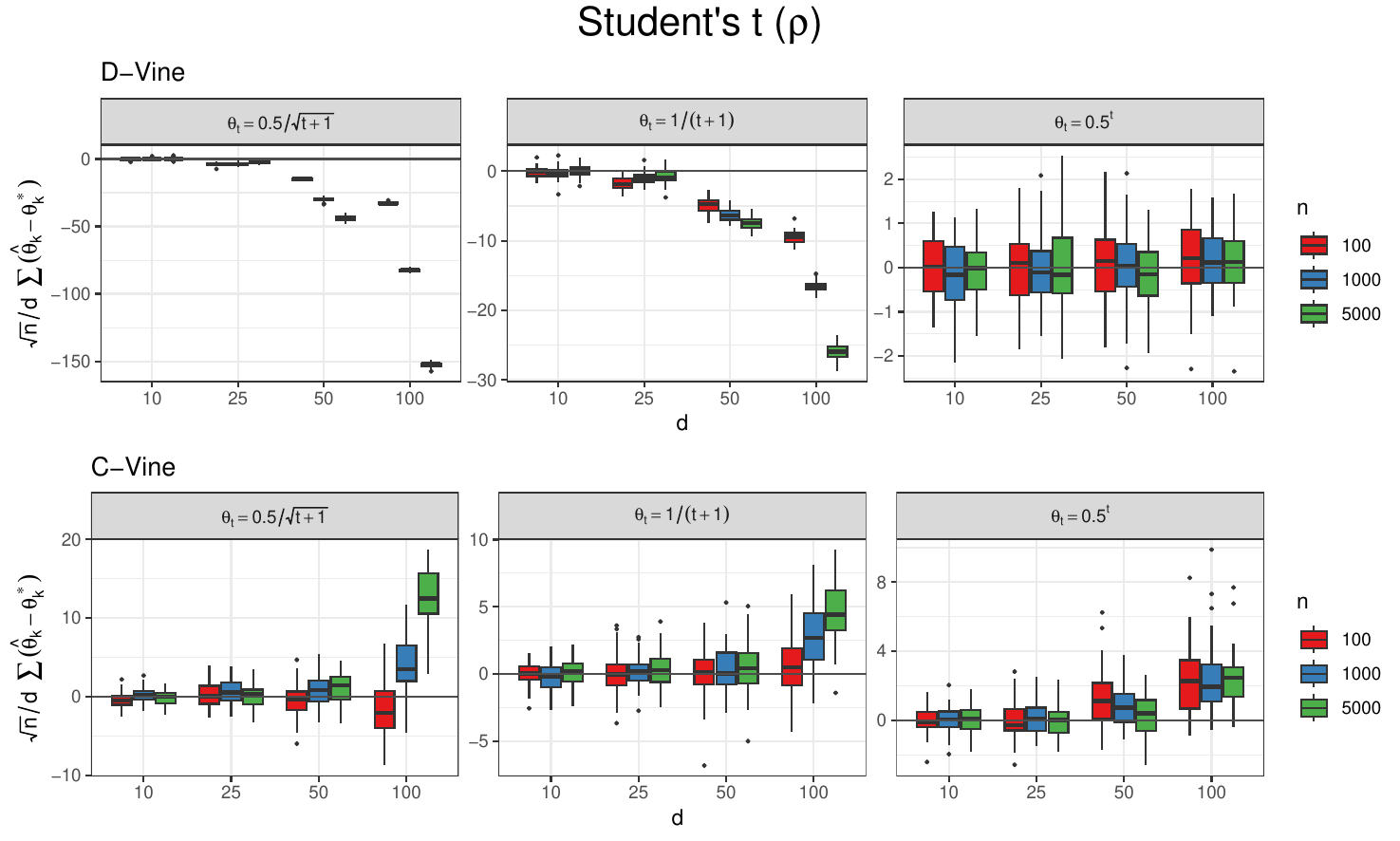}
  \caption{Parameter estimation for Student's t vines, sum of estimation errors. }
  \label{fig:par_student}
\end{figure}

\myparagraph{Student's t Vines}

We now turn to the estimation of Student's t vines. The results on the assessment of consistency are shown in the supplement. They indicate that consistent estimation of $\brho$ is possible in certain settings (D-vine, $\rho_t = 1/(t+1)$ and $\rho_t = 0.5^t$), whereas the degrees of freedom parameter $\bnu$ cannot be estimated consistently.
\cref{fig:par_student} shows that $\bm{\hat{\rho}}$ is biased in almost all settings, with even larger biases for $\bm{\hat{\nu}}$ (see the supplement).
Unlike for Gumbel and Gaussian vines, the bivariate copulas do not converge to independence copulas in this setting, which probably causes the observed biases.
Nevertheless, the upper right panel (D-vine, $\rho_t = 0.5^t$) suggests that there are settings in which the estimation of the correlation parameter $\rho$ can still be asymptotically normal with mean zero.

\myparagraph{Mixture of Gumbel pair copulas}
\textcolor{red}{
  In contrast to the Student’s t vines, the two-parametric pair copulas in the simulation from a mixture of Gumbel copulas converge to independence, as both $\theta_1^* \to 0, \theta_2^* \to 0$ as $d_n \to \infty$.
  The results are shown in the supplement.
  Consistent estimation of $(\theta_1^*, \theta_2^*)$ with $d \sim n$ is possible in all investigated settings.
  The sum of errors indicates that, similar to the results for the Gumbel vine, the estimation is biased except for $\theta^*_t = 0.5^t$.
}

\myparagraph{Number of Parameters and Truncated Vines}

In many of the investigated models, we encounter settings in which the number of parameters $p$ exceeds the sample size $n$, yet consistent estimation of $\btheta^*$ remains possible.
This is in line with \cref{theorem_cons_new2}, which does not explicitly require $p_n / n \to 0$.
Instead, restrictions on the rate of growth of $p_n$ are implicitly imposed, mainly by \ref{A:emp_proc}.
While the numerical experiments in the supplement indicate that $p_n^2/n \to 0$ is sufficient for \ref{A:emp_proc}, the simulation study suggests that this restriction can be relaxed in practice.

Another interesting application of the theoretical results arises in truncated vines with copula dimension $d_n \to \infty$ and a fixed truncation level.
In this setting, \cref{theorem_cons_new2} only requires $\ln d_n / n \to 0$ under standard conditions.
Simulations from a Gaussian C-vine with $\theta_t = 1/(t+1)$ ($t=1,2$) that is truncated after the second tree indicate that even for $d = 2000$, estimation is unbiased for $n \ge 100$ (see the supplement for an additional figure).

\myparagraph{Nonparametric Estimation of Margins}
\label{sec:ResParEstimNonPar}

We now compare $\hbtheta_U$, the estimator based on known margins, to $\hbtheta_X$, where the copula data is estimated from the empirical distribution functions $F_n(\bX_i)$, which induces additional uncertainty.
All univariate margins are standard normal, and all remaining parameters are the same as above.

\begin{figure}
  \centering
  \includegraphics[width = 0.95\textwidth]{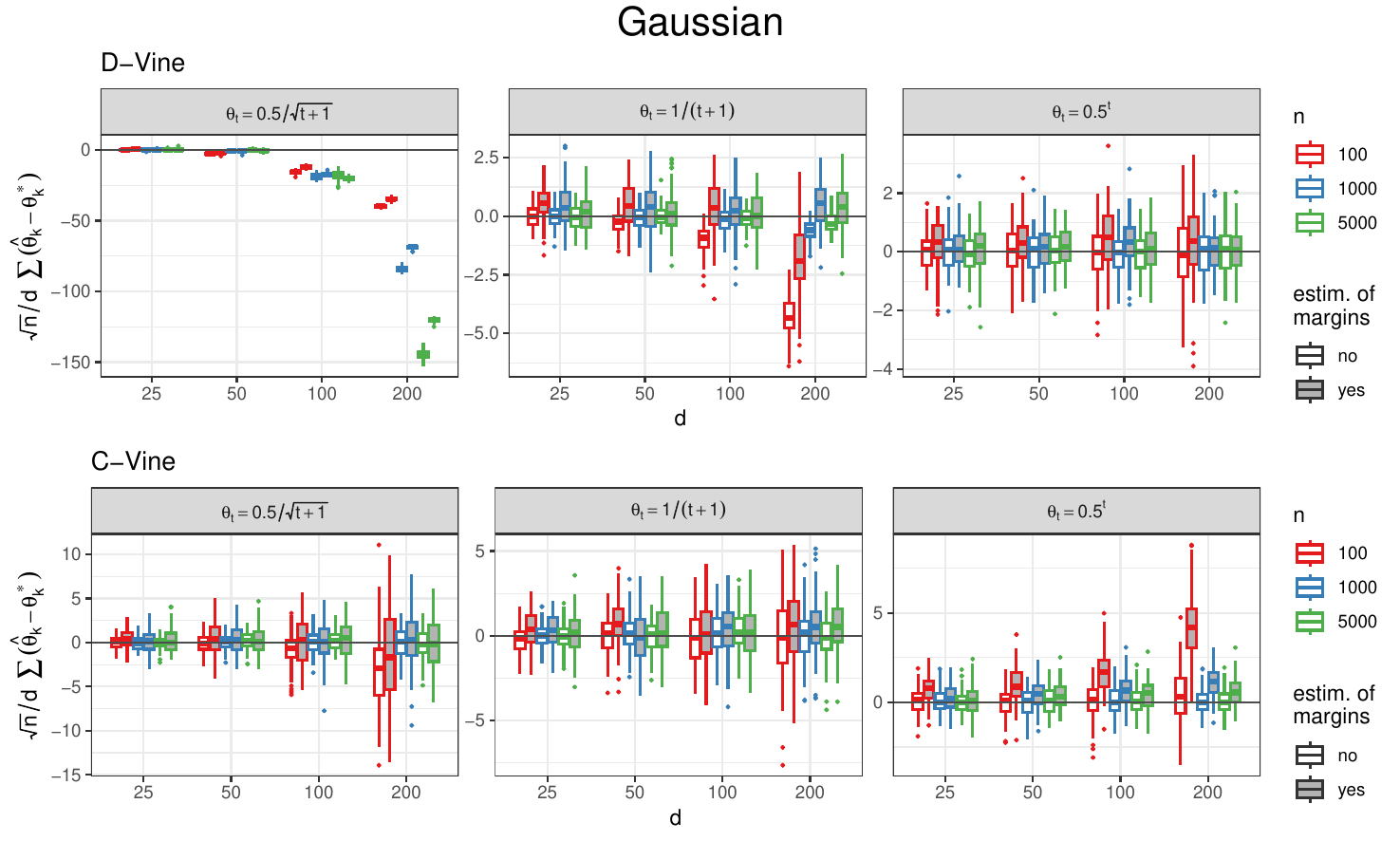}
  \caption{Parameter estimation for Gaussian vines with nonparametric estimation of margins, sum of estimation errors. }
  \label{fig:par_gauss_margins}
\end{figure}

\cref{fig:par_gauss_margins} shows the sum of errors for Gaussian vines for both $\hbtheta_U$ and $\hbtheta_X$.
For D-vines, we only observe a small increase in variance, whereas the estimation based on nonparametric margins in high-dimensional C-vines is slightly biased even for large $n$.
The corresponding results for Gumbel and Student's t vines exhibit similar patterns and are therefore only shown in the supplement.

\section{Discussion}
\label{sec:conclusion}
In this work, we provide asymptotic results for stepwise parameter estimation in high-dimensional vine copulas.
The theory covers both parametric and nonparametric estimation of margins as well as truncated vines, and is not restricted to specific R-vine structures or parametric families.
Moreover, it is applicable to any parametric estimation problem whose solution can be expressed as the root of a system of estimating equations.

\textcolor{red}{Some of the presented assumptions are rather abstract and difficult to verify analytically for generic vines involving a mix of popular copula families. However, most assumptions  are trivial or straightforward to establish for finite-dimensional or truncated vines (see \cref{prop:trunc}), whereas  analytical verification for sufficiently smooth pair-copula models is feasible, but tedious.
At the same time, the conditions show that the estimation of parameters of a vine with diverging dimension and too strong dependence becomes challenging due to the propagation of estimation errors from earlier trees.
Using an independence vine with $d_n \to \infty$ as a starting point, we can allow for a certain amount of dependence, potentially at the cost of stricter assumptions on $d_n$ or a slower rate of convergence.}

A numerical validation of the assumptions suggests that the derived conditions hold if the pair copulas in higher trees converge sufficiently fast to independence copulas and $p_n = o(\sqrt{n})$, which may be too restrictive for many practical applications.
The simulation study indicates that, in practice, inference on vine copula parameters is often feasible even when the dimension and therefore the number of parameters is relatively high compared to the sample size, e.g., $p_n \sim n$.

In applications, estimation of vine copula parameters often comes with the challenge of selecting an appropriate vine tree structure and pair copula families.
Choosing between different copula families is typically done by minimizing AIC or BIC, with a modified version tailored for high-dimensional vines developed by \citet{Nagler2019}.
As the number of possible vine structures grows superexponentially \citep{MoralesN}, maximizing such criteria to also determine the vine structure is infeasible.
The popular algorithm by \citet{Dissmann2013} greedily builds the tree structure by fitting the strongest dependencies first, typically resulting in (almost) independence in higher trees.
The validity of this approach is supported by our derived conditions for consistent estimation, which require decreasing dependence in higher trees.

While we only cover parametric estimation of vine copulas, our results are also applicable to nonparametric approaches such as B-splines \citep{Kauermann2014} or Bernstein copulas (\citet{Scheffer17}, see \citet{Nagler17} for an overview and comparison of nonparametric methods).
In such cases, it is natural to increase the number of basis functions as $n$ increases.
Bias-variance tradeoff considerations typically yield rules of the form $n^\delta$ parameters per pair copula for some small $\delta$.
If $M_n$ in \ref{A:emp_proc} is sufficiently small, our theory ensures consistent estimation of $O(d^2 n^\delta)$ parameters.
The smaller $M_n$, the more we can relax the assumptions on $d_n$.
Another use case with $p \to \infty$ as $n \to \infty$, even when $d = O(1)$, arises in non-simplified vines.
Here, our results can be extended to approaches such as the nonparametric method by \citet{Schellhase18}, as well as models in which the parameters $\theta_{a_e, b_e ; D_e}(\bu_{D_e})$ are estimated using GLMs \citep{Han17} or splines \citep{Vatter2018}.
Working out the assumptions in detail in these special cases is beyond the scope of this paper.

Many existing methods for high-dimensional vines aim to induce sparsity, such as truncated or thresholded vines (see \cref{sec:intro}), which could also be achieved by adding a suitable penalty term to the estimating function.
To the best of our knowledge, penalized estimation of parametric vines has not yet been studied, and extending the results on high-dimensional penalized estimation in \citet{Gauss24} remains an open problem for future research.

\vspace{1cm}

\paragraph{Funding sources}
This research did not receive any specific grant from funding agencies in the public, commercial, or not-for-profit sectors.

\paragraph{Supplement}
The supplement contains a numerical validation of the assumptions\textcolor{red}{, a theoretical discussion of the identifiability assumption \ref{A:curvature}, an analytical verification of the assumptions for FGM vines}, additional figures for the simulation study as well as a proposition on $J(\btheta^*)$ for Gaussian C-vines.

\paragraph{Code availability} 
The code and data to reproduce the results can be found on GitHub:
\url{https://github.com/JanaGauss/high\_dimensional\_vines}

 \bibliographystyle{elsarticle-harv} 
\bibliography{bibliography_JMVA}

 \appendix

\section{Proofs of the Main Results}

\label{sec:proofs} 

To simplify notation, we use the following abbreviations: $\P_n f = \frac 1 n \sumin f(\bX_i)$, $P f = \E[f(\bX)]$ and
\[
\P_n \phi_X (\btheta) = \frac 1 n \sumin \phi(F_n(\bX_i); \btheta) \quad \text{and} \quad \P_n \phi_U (\btheta) = \frac 1 n \sumin \phi(F(\bX_i); \btheta),
\]
where
\[
F(\bX_i) = (F_1(X_{i,1}), \ldots, F_{d_n}(X_{i,d_n})), \quad 
F_n(\bX_i) = (F_{n1}(X_{i,1}), \ldots, F_{n d_n}(X_{i,d_n})) 
\]
and $ F_{n k}(x) = (n+1)^{-1} \sumin \ind (X_{i, k} \le x)$.

$\| \cdot \|$ denotes the Euclidean norm for vectors and the spectral norm $\| A \| = \sup_{\| \bx \| = 1} \| A \bx \|$ for matrices.
Define $r_n \coloneqq\sqrt{\ln p_n / n}$.

\subsection{Proof of \autoref{theorem_cons_new2}}

Denote $\Theta_n^\Delta = \{ \bDelta: |\Delta_j | \le \alpha_{j,n}  \forall j = 1, \ldots, p_n\} \subset \R^{p_n}$.
The \emph{Poincar\'e-Miranda theorem} \citep[see for example][]{Kulpa} implies that if 
\[
\max_{1 \le j \le p_n} \sup_{\bDelta \in \Theta_n^\Delta , | \Delta_j | = \alpha_{j,n}}  \sign(\Delta_j) \P_n \phi_U(\btheta^* + r_n C \bDelta)_j \le 0
\]
holds, there is a solution $\hbtheta_U$ of $\P_n \phi_U(\btheta ) = \0$ with $| \hat{\theta}_{j,U} - \theta_j^* | \le r_n C \alpha_{j,n}$ for all $j = 1, \ldots, p_n$.
\textcolor{red}{We show that for any $\eps > 0$, we can choose a constant $C$ such that
\begin{equation}
\label{eq:PoincareM}
\liminf_{n \to \infty} \P\left( \max_{1 \le j \le p_n} \sup_{\bDelta \in \Theta_n^\Delta , | \Delta_j | = \alpha_{j,n}}  (r_n C \alpha_{j,n})^{-1} \sign(\Delta_j) \P_n \phi_U(\btheta^* + r_n C \bDelta)_j  < 0 \right) \ge 1 - \eps.
\end{equation}
This means that, for all $\eps > 0$, we can choose $C_\eps$ and $n_{C_{\eps}}$ such that for all $n \ge n_{C_{\eps}}$, we have 
\[
\P\left( \exists \hbtheta_U \text{ s.t. }\max_{1 \le j \le p_n } | \hat{\theta}_{U,j} - \theta^*_j| \le r_n C_\eps \alpha_{j,n} \right) \ge 1 - \eps.
\]
}
Therefore, with probability tending to 1, the sets $\Theta_n$ contain a solution $\hbtheta_U$ that satisfies $\max_{1 \le j \le p_n } | \hat{\theta}_{U,j} - \theta^*_j| /\alpha_{j,n} = O_p(r_n)$. 
\textcolor{red}{See \citet[][proof of Theorem 1]{Fan} for a somewhat similar proof strategy.}
We have
\begin{align}
\label{eq:proof1}
&  \max_{1 \le j \le p_n} \sup_{\bDelta \in \Theta_n^\Delta , | \Delta_j | = \alpha_{j,n}} (r_n C\alpha_{j,n})^{-1} \sign(\Delta_j) \P_n \phi_U(\btheta^* + r_n C \bDelta)_j  \nonumber \\
& \le   \max_{1 \le j \le p_n} (r_n C\alpha_{j,n})^{-1}  | \P_n \phi_U(\btheta^*)_j |   + \max_{1 \le j \le p_n}  \sup_{\bDelta \in \Theta_n^\Delta , | \Delta_j | = \alpha_{j,n}} (r_n C\alpha_{j,n})^{-1} \sign(\Delta_j)  \P_n \lf[\phi_U(\btheta^* + r_n C \bDelta)_j - \phi_U(\btheta^* )_j  \ri].
\end{align}
\textcolor{red}{Note that $\max_{1 \le j \le p_n} \alpha_{j,n}^{-1} = O(1)$, since $\min_{1 \le j \le p_n} \alpha_{j,n}$ is bounded away from 0 by the assumptions on $\alpha_{j,n}$.
The first term in \eqref{eq:proof1} is of order $C^{-1} O_p(1)$ since $\| \P_n \phi_U(\btheta^*) \|_{\infty} = O_p(\sqrt{\ln p_n / n})$ by \ref{A:ConvRate_new}, \ref{A:phi-moments} and \cref{lem:eta_n}.}
For the second term, we have, with some $c > 0$ and $k_n$ as defined in \ref{A:emp_proc},
\begin{align*}
& \quad \max_{1 \le j \le p_n} \sup_{\bDelta \in \Theta_n^\Delta , | \Delta_j | = \alpha_{j,n}} (r_n C \alpha_{j,n})^{-1} \sign(\Delta_j)  \P_n \lf[\phi_U(\btheta^* + r_n C \bDelta)_j - \phi_U(\btheta^* )_j  \ri] \\
& \le  \max_{1 \le j \le p_n} \sup_{\bDelta \in \Theta_n^\Delta , | \Delta_j | = \alpha_{j,n}}  (r_n C \alpha_{j,n})^{-1} \sign(\Delta_j)  P \lf[\phi_U(\btheta^* + r_n C \bDelta)_j - \phi_U(\btheta^* )_j  \ri]  + \\
& \quad \max_{1 \le j \le p_n} \sup_{ \bDelta \in \Theta_n^{\Delta}}  (r_n C \alpha_{j,n})^{-1} \left| (\P_n - P) \lf[\phi_U(\btheta^* + r_n C \bDelta)_j - \phi_U(\btheta^* )_j  \ri] \right| \\
& \le -c + O_p\lf( \frac{M_n  \sqrt{k_n + \ln p_n}}{\sqrt{n}} + \frac{D_n  (k_n + \ln p_n) }{n} \ri) = -c + o_p(1)
\end{align*}
by \ref{A:curvature}, \ref{A:emp_proc}  and \cref{lem:emp_proc}.
\textcolor{red}{Given any $\eps > 0$, we can choose $C$ large enough such that the first term in \eqref{eq:proof1} is smaller than $c$ with probability $\ge 1 - \eps$.}
The second term in \eqref{eq:proof1} therefore dominates the first one. As this term is negative for $n$ large, this implies the claim.
 
\subsection{Proof of \autoref{theorem2}}
\textcolor{red}{
Define $\tilde r_n =  \alpha_n\sqrt{p_n \ln p_n / n}$.
We have
\begin{align*}
\0 & = \P_n \phi(\hbtheta_U)
= \P_n \phi(\btheta^*) + \P_n [ \phi(\hbtheta_U) - \phi(\btheta^*)] =  \P_n \phi(\btheta^*) + P [ \phi(\hbtheta_U) - \phi(\btheta^*)] + (\P_n - P) [ \phi(\hbtheta_U) - \phi(\btheta^*)]\\
& =  \P_n \phi(\btheta^*) + \nabla_{\btheta} P \phi(\bttheta)(\hbtheta_U - \btheta^*)   + (\P_n - P) [ \phi(\hbtheta_U) - \phi(\btheta^*)], 
\end{align*}
with some $\bttheta$ on the segment between $\hbtheta_U$ and $\btheta^*$.
We have $\nabla_{\btheta} P \phi(\bttheta) = J(\bttheta) $, so
\begin{align*}
- J(\btheta^*)(\hbtheta_U - \btheta^*)  & = \P_n \phi(\btheta^*) + [J(\bttheta) - J(\btheta^*)] (\hbtheta_U - \btheta^*)   + (\P_n - P) [ \phi(\hbtheta_U) - \phi(\btheta^*)]
\quad \text{and} \\
- \sqrt{n} A_n J(\btheta^*)(\hbtheta_U - \btheta^*) 
= & \sqrt{n} A_n \P_n \phi(\btheta^*)   + \sqrt{n} A_n  [J(\bttheta) - J(\btheta^*)] (\hbtheta_U - \btheta^*)   + \sqrt{n} (\P_n - P) A_n [\phi(\hbtheta_U ) - \phi(\btheta^*)].
\end{align*}
The second and the third term are negligible, since
$\sqrt{n} A_n [J(\bttheta) - J(\btheta^*)] (\hbtheta_U - \btheta^*) 
= o_p(\sqrt{n} (\sqrt{n} \tilde r_n)^{-1} \tilde r_n ) = o_p(1)
$
by the third condition in \ref{A:Asymp}, and \cref{lem:lemma10} and \ref{A:Asymp} yield 
$
 \sqrt{n} (\P_n - P) A_n [\phi(\hbtheta_U ) - \phi(\btheta^*)] =o_p(1).
$
It remains to prove a central limit theorem for $\sqrt{n} A_n \P_n \phi(\btheta^*)  = \sumin  \frac{1}{\sqrt{n}} A_n \phi_{i}(\btheta^*) \eqqcolon\sumin \bY_{i}.$
Since
\begin{align*}
\sumin \E \left[ \lVert \bY_{i} \rVert^2 \mathbbm{1}\{\lVert \bY_{i} \rVert > \eps \}  \right] 
& \le \sumin \E \left[ \lVert \bY_{i} \rVert^2 \mathbbm{1}\{\lVert \bY_{i} \rVert > \eps \} \| \bY_{i} \|^2/\eps^2  \right] \leq \sumin \E\left[ \lVert \bY_{i} \rVert^4 \right]/\eps^2,
\end{align*}
and $ \E[ \lVert \bY_{i} \rVert^4] = n^{-2}\E[ \| A_n \phi_{i}(\btheta^*) \|^4 ] = o(n^{-1})$ for all $i = 1, \ldots, n $, by \ref{A:Asymp2}, we have $\sumin \mathbb{E} \left[ \lVert \bY_{i} \rVert^2 \mathbbm{1}\{\lVert \bY_{i} \rVert > \eps \}  \right] \to 0 \text{ for every } \eps > 0.$
Since $\mathbb{E}[\bY_{i}]=\boldsymbol{0}$ for all $i=1, \ldots, n$ and 
\begin{align*}
\sum_{i=1}^n \mathrm{Cov}(\bY_{i}) 
 = \frac 1 n \sumin  \mathrm{Cov}[A_n \phi_{i}(\btheta^*)] 
 =  \frac 1 n A_n\sumin  \mathrm{Cov}[ \phi_{i}(\btheta^*)]  A_n^\top 
=  A_n  I(\btheta^*)  A_n^\top
\to \Sigma,
\end{align*}
the conditions of the Lindeberg-Feller central limit theorem \citep[Section 2.8]{vdV2} are satisfied, and we obtain $\sqrt{n} A_n J(\btheta^*)  (\hbtheta_U - \btheta^*) \to_d \Ncal(\0, \Sigma).$
}
 
\subsection{Proof of \autoref{theorem_semiparam_new}}

Similar to the proof of \cref{theorem_cons_new2}, we show that, by choosing $C$ large enough, the probability of
\begin{equation}
\label{eq:PoincareM}
\max_{1 \le j \le p_n} \sup_{\bDelta \in \Theta_n^\Delta , | \Delta_j | = \alpha_{j,n}}  (r_n C \alpha_{j,n})^{-1} \sign(\Delta_j) \P_n \phi_X(\btheta^* + r_n C \bDelta)_j < 0
\end{equation}
gets arbitrarily close to 1.
This implies that, with probability tending to 1, the sets $\Theta_n$ contain a solution $\hbtheta_X$ that satisfies $\max_{1 \le j \le p_n } | \hat{\theta}_{X,j} - \theta^*_j| /\alpha_{j,n} = O_p( r_n)$.

Denote $\underline{\alpha}_n = \min_{1 \le j \le p_n} \alpha_{j,n}$, is bounded away from 0 by the assumptions on $\alpha_{j,n}$.
It holds that
\begin{align}
\label{eq:margins_nonparam}
&  \max_{1 \le j \le p_n} \sup_{\bDelta \in \Theta_n^\Delta , | \Delta_j | = \alpha_{j,n}} (r_n C\alpha_{j,n})^{-1} \sign(\Delta_j) \P_n \phi_X(\btheta^* + r_n C \bDelta)_j \nonumber \\
& \le   \max_{1 \le j \le p_n} (r_n C \underline{\alpha}_n)^{-1}  | \P_n \phi_X(\btheta^*)_j |  + \max_{1 \le j \le p_n}  \sup_{\bDelta \in \Theta_n^\Delta , | \Delta_j | = \alpha_{j,n}} (r_n C\alpha_{j,n})^{-1} \sign(\Delta_j)  \P_n \lf[\phi_X(\btheta^* + r_n C \bDelta)_j - \phi_X(\btheta^* )_j  \ri]. 
\end{align}
For the first term, we have
\[
(r_n C\underline{\alpha}_n)^{-1}  \| \P_n \phi_X(\btheta^*) \|_{\infty} \le (r_n C\underline{\alpha}_n)^{-1}  \| \P_n \phi_U(\btheta^*) \|_{\infty} + (r_n C\underline{\alpha}_n)^{-1}  \| \P_n \phi_X(\btheta^*)  - \P_n \phi_U(\btheta^*) \|_{\infty}.
\]
The first term is of order \textcolor{red}{$C^{-1}O_p(1)$ by \cref{lem:eta_n}}.
The second term is also of order \textcolor{red}{$C^{-1}O_p(1)$} since $  \| \P_n \phi_X(\btheta^*)  - \P_n \phi_U(\btheta^*) \|_{\infty} = O_p(\sqrt{\ln p_n / n})$ by \cref{lem:margins1}, \ref{A:Var_h1}, \ref{A:AsympSemi_2}, \ref{A:SemiPar_I2} and \textcolor{red}{$d_n^{(3 + 4b)}(\ln d_n)^2/n =O(1)$.}

For the second term in \eqref{eq:margins_nonparam}, we have
\begin{align*}
&  \max_{1 \le j \le p_n}  \sup_{\bDelta \in \Theta_n^\Delta , | \Delta_j | = \alpha_{j,n}} (r_n C\alpha_{j,n})^{-1} \sign(\Delta_j)  \P_n \lf[\phi_X(\btheta^* + r_n C \bDelta)_j - \phi_X(\btheta^* )_j  \ri] \\
& \le  \max_{1 \le j \le p_n}  \sup_{\bDelta \in \Theta_n^\Delta , | \Delta_j | = \alpha_{j,n}} (r_n C\alpha_{j,n})^{-1} \sign(\Delta_j)  \P_n \lf[\phi_U(\btheta^* + r_n C \bDelta)_j - \phi_U(\btheta^* )_j  \ri] + \max_{1 \le j \le p_n}  \sup_{ \bDelta \in \Theta_n^\Delta , | \Delta_j | = \alpha_{j,n}} (r_n C\alpha_{j,n})^{-1}|  \P_n f_j(r_n C \bDelta) | ,
\end{align*}
where
\begin{equation}
\label{eq:def_f}
f_j(\bDelta) \coloneqq \phi_X(\btheta^* +  \bDelta)_j - \phi_X(\btheta^* )_j  - (\phi_U(\btheta^* +  \bDelta)_j - \phi_U(\btheta^* )_j ).
\end{equation}
The first term remains below some $-c < 0$ with probability tending to 1 by \ref{A:curvature} and \ref{A:emp_proc}, see the proof of \cref{theorem_cons_new2}.
The second term is $o_p(1)$ by \cref{lem:margins_cons_remainder}, \ref{A:SemiPar_r}, \ref{A:alpha_max} and \textcolor{red}{$d_n^{(2+2b)} \ln d_n/n \to 0$}.
This implies the claim.
 
\subsection{Proof of \autoref{theorem2_semiparam_new}}
It holds that
\begin{align*}
\0 & = \P_n \phi_X(\hbtheta_X) =  \P_n \phi_X(\btheta^*) +  \P_n [  \phi_X(\hbtheta_X) -  \phi_X(\btheta^*)] \\
& =  \P_n \phi_X(\btheta^*)  +  \P_n [  \phi_U(\hbtheta_X) -  \phi_U(\btheta^*)] +  \P_n [  \phi_X(\hbtheta_X) -  \phi_X(\btheta^*) - ( \phi_U(\hbtheta_X) -  \phi_U(\btheta^*))] .
\end{align*}
Similar to the proof of \cref{theorem2}, we obtain
\begin{align*}
 - \sqrt{n} A_n J(\btheta^*) (\hbtheta_X - \btheta^*)  & =   \sqrt{n} A_n \P_n \phi_X (\btheta^*)  + \sqrt{n} A_n [J(\bttheta)- J(\btheta^*)] (\hbtheta_X - \btheta^*) \\
&\quad + \sqrt{n} (\P_n - P) A_n [\phi_U(\hbtheta_X) - \phi_U(\btheta^*) ]  +  \sqrt{n} A_n  \P_n [  \phi_X(\hbtheta_X) -  \phi_X(\btheta^*) - ( \phi_U(\hbtheta_X) -  \phi_U(\btheta^*))] ,
\end{align*}
with some $\bttheta$ on the segment between $\btheta^*, \hbtheta_X$.
The second and third term are $o_p(1)$ by \ref{A:Asymp}, see the proof of \cref{theorem2}.
For the last term, \cref{lem:margins_cons_remainder}, $\| \hbtheta_X - \btheta^* \|_{\infty} = O_p(\alpha_n \sqrt{\ln p_n / n})$ and  $\| \bx \|_2 \le \sqrt{p_n} \| \bx \|_{\infty}$ for $\bx \in \R^{p_n}$ (with $\bx = \P_n [  \phi_X(\hbtheta_X) -  \phi_X(\btheta^*) - ( \phi_U(\hbtheta_X) -  \phi_U(\btheta^*))]$) yield\textcolor{red}{
\[
\sqrt{n} A_n  \P_n [  \phi_X(\hbtheta_X) -  \phi_X(\btheta^*) - ( \phi_U(\hbtheta_X) -  \phi_U(\btheta^*))]
  = O_p\lf(\alpha_n \sqrt{\frac{p_n d_n^{(2 + 2b)} \ln p_n \ln d_n}{n}}\ri) = o_p(1)
\]
since $\alpha_n^2 p_n d_n^{(2 + 2b)} \ln p_n \ln d_n /n \to 0$.}
For $ \sqrt{n} A_n \P_n \phi_X (\btheta^*) $, it holds that
\begin{align*}
\P_n \phi_X (\btheta^*)  = \P_n \phi_U (\btheta^*) + [ \P_n \phi_X (\btheta^*) - \P_n \phi_U (\btheta^*)]  =  \P_n \phi_U (\btheta^*) +  \P_n \nabla_{\bu} \phi(F(\cdot ); \btheta^*) (F_n(\cdot ) - F(\cdot )) +  I_2  =\tilde I_1 +  I_2
\end{align*}
with $I_2$ as defined in \cref{lem:margins1}.
\cref{lem2_new} and \ref{A:SemiPar2} give \textcolor{red}{$\| I_2 \| =  O_p( \sqrt{ d_n^{(2 + 4b)} p_n (\ln d_n)^2 /n^2 })$, so $\sqrt{n} A_n I_2 = o_p(1)$ since $\| A_n \| = O(1)$ and $d_n^{(2 + 4b)} p_n(\ln d_n)^2 /n \to 0$.}
$\tilde I_1$ can be treated as a U-statistics \citep[see for example][]{vdV2}.
Define
\[
\tilde h(\bX_i, \bX_{i'}) \coloneqq\phi(F(\bX_i); \btheta^*) + \sum_{l=1}^{d_n} \frac{\partial }{\partial u_l} \phi(F(\bX_{i'}); \btheta^*) ( \ind( X_{il} \le X_{i'l}) - F_l(X_{i'l})) \in \R^{p_n},
\]
where $ \frac{\partial }{\partial u_l} \phi(F(\bX_{i'}); \btheta^*) $ denotes the $p_n$-dimensional vector with entries $ \frac{\partial }{\partial u_l} \phi(F(\bX_{i'}); \btheta^*)_k$.
Note that the second term is the same as $h^a(\bX_i, \bX_{i'})$ as defined in \eqref{eq:def_Y}.
Then
\begin{align*}
\frac{n+1}{n-1} \, \tilde I_1 & = \frac{1}{n (n-1)} \sumin \sum_{i'=1, i' \neq i }^n \tilde h(\bX_i,\bX_{i'}) +  \frac{1}{n (n-1)} \sumin  h^a(\bX_i , \bX_i).
\end{align*}
The second term is of order $O_p(\sqrt{p_n}/n)$ by \ref{A:AsympSemi_2}, see \cref{lem:margins1}.
Multiplying this term with $\sqrt{n} A_n$ yields an $o_p(1)$ term since $p_n/n \to 0$ and $\| A_n \| = O(1)$.
Now define the symmetric kernel 
\begin{equation}
\label{eq:hUstat1}
\bar h (\bX_i, \bX_{i'}) = \frac 1 2 \tilde h(\bX_i, \bX_{i'}) + \frac 1 2  \tilde h(\bX_{i'}, \bX_i) 
\end{equation}
and
\begin{equation}
\label{eq:hUstat}\bar h_1(\bx) = \E[\bar h(\bx, \bX)] = \frac 1 2 \phi(F(\bx); \btheta^*) + \frac 1 2 \E\lf[ \sum_{l=1}^{d_n} \frac{\partial }{\partial u_l} \phi(F(\bX); \btheta^*) ( \ind( x_l \le X_l) - F_l(X_l))  \ri],
\end{equation}
since $\E[\tilde h(\bX, \bx)] = \bnull$, as $\E[ \ind (X_{i l} \le X_{i'l}) - F_1(X_{i'l})  | X_{i'l}] = 0$ for all $l = 1, \ldots, d_n,$ $i,i' = 1, \ldots, n$.
Note that also $\E[\bar h_1(\bX)] = \bnull$. 
Then, \emph{Hoeffding's decomposition} \citep{Hoeff} yields
\begin{align*}
 \frac{1}{n (n-1)} \sumin \sum_{i'=1, i' \neq i }^n \tilde h(\bX_i,\bX_{i'})
  = \frac{2}{n} \sumin  \bar h_1(\bX_i)  +  \frac{1}{n (n-1)}  \sumin \sum_{i' = 1, i' \neq i}  \bar h(\bX_i,\bX_{i'}) - \bar  h_1(\bX_i) -  \bar h_1(\bX_{i'})  = I + II.
\end{align*}
\cref{lem:hoeffd_remainder}, \ref{A:Var_h1} and \ref{A:AsympSemi_2} yield $\sqrt{n} A_n II = o_p(1)$ since $p_n/n \to 0$ and $\| A_n \| = O(1)$.
We can apply \cref{lem:hoeffd_remainder} here despite the difference between $\bar h(\bX_i, \bX_{i'})$ and $h(\bX_i, \bX_{i'})$ as defined in \eqref{eq:def_Y}, since the terms $\phi(F(\bX_i); \btheta^*), \phi(F(\bX_{i'}); \btheta^*)$ in $\bar h_1(\bX_i), \bar h_1(\bX_{i'})$ cancel with the same terms in $\bar h(\bX_i, \bX_{i'})$.

It remains to show a central limit theorem for $\sumin 2 / \sqrt{n} \; A_n   \bar h_1(\bX_i) \eqqcolon \sumin \bY_i$.
We write
\begin{align*}
\bar h_1(\bx) = \frac 1 2 \phi(F(\bx); \btheta^*) + \frac 1 2 \int \sum_{l=1}^{d_n} \frac{\partial }{\partial u_l} \phi(\bu; \btheta^*) (\ind(F(x_l) \le u_l) - u_l) d C(\bu) ,
\end{align*}
where $C$ is the copula distribution of $\bX$, i.e., the distribution of $\bU = F(\bX)$.
It holds that
\[
\cov(2 \bar h_1(\bX))  = \cov \lf( \phi(\bxi; \btheta^*) +  \int \sum_{l=1}^{d_n}  \frac{\partial}{\partial u_l} \phi(\bu ; \btheta^*) ( \ind ( \xi_l \le u_l) - u_l) d C(\bu) \ri),
\]
where $\bxi$ is a random variable with distribution $C$.
Therefore 
\[
\lim_{n \to \infty} \cov\lf( \sumin \bY_i \ri)  = \lim_{n \to \infty} A_n \cov(2 \bar h_1(\bX)) A_n^\top = \Sigma_X
\]
with $\Sigma_X$ as defined in \cref{theorem2_semiparam_new}.
Since $\E[\bY_i] = \bnull$ as $\E[\bar h_1(\bX_i)] = \bnull$ and 
\begin{align*}
\sumin \E \left[ \lVert \bY_{i} \rVert^2 \mathbbm{1}\{\lVert \bY_{i} \rVert > \eps \}  \right] 
& \le \sumin \E \left[ \lVert \bY_{i} \rVert^2 \mathbbm{1}\{\lVert \bY_{i} \rVert > \eps \} \| \bY_{i} \|^2/\eps^2  \right] \leq \sumin \E\left[ \lVert \bY_{i} \rVert^4 \right]/\eps^2 = o(1)
\end{align*}
for every $\eps > 0$ by \ref{A:Asymp2} and \ref{A:Asymp2_margins}, the conditions of the Lindeberg-Feller central limit theorem \citep[Section 2.8]{vdV2} are satisfied, and we obtain 
$
\sqrt{n} A_n J(\btheta^*)  (\hbtheta_X - \btheta^*) \to_d \Ncal(\0, \Sigma_X).
$
 
\subsection{Proof of \autoref{prop:trunc}}

\textcolor{red}{
\ref{A:ConvRate_new} immediatly follows from the assumptions.
For \ref{A:phi-moments}, using Markov's inequality, it holds that
\[
 \Pr \left( \| \phi(\bU; \btheta^*) \|_\infty > \sqrt{\frac{\sigma_n^2 n}{4 \ln p_n} } \right) \le p_n \max_{1 \le k \le p_n} \Pr \left(  | \phi(\bU; \btheta^*)_k | > \sqrt{\frac{\sigma_n^2 n}{4 \ln p_n} } \right) \le p_n \frac{\max_{1 \le k \le p_n} \E[(\phi(\bU; \btheta^*))^4]}{\left(\frac{\sigma_n^2 n}{4 \ln p_n } \right)^2} = O\left( \frac{p_n (\ln p_n)^2}{n^2} \right),
\]
which is $o(1/n)$ since $p_n (\ln p_n)^2/n \to 0$.
For \ref{A:curvature}, we have, with some $\bttheta$ on the segment between $\btheta^*$ and $\btheta^* + r_n C_n \bDelta$, 
\[
(r_n C_n \alpha_{j,n})^{-1} \sign(\Delta_j) \, \E \lf[\phi(\bU; \btheta^* + r_n C_n \bDelta)_j \ri] 
 = \E\left[ \frac{\partial}{\partial \theta_j} \phi(\bU; \bttheta)_j \right] + \frac{\sign(\Delta_j)}{\alpha_{j,n}} \sum_{k=1}^{p_n} \Delta_k \E\left[ \frac{\partial}{\partial \theta_k} \phi(\bU; \bttheta)_j \right] \le -c + \frac{C}{ \alpha_{j,n}} \sum_{k \in I_j}\alpha_{k,n},
\]
where $I_j$ contains the indices with $\frac{\partial}{\partial \theta_k} \phi(\bu; \btheta)_j \neq 0$. Due to the truncation, $|I_j| \le 2$ for all $j$, so we can always choose sequences $\alpha_{j,n}, j  =1, \ldots, p_n$ such that \ref{A:curvature} is satisfied and $\max_{1 \le j \le p_n} \alpha_{j,n} = O(1)$, leading to the rate of convergence $\sqrt{\ln p_n / n}$.
For \ref{A:emp_proc}, the truncation and the assumptions imply that the left-hand side of the first expression is bounded, so we may choose $M_n = C$ with some large enough constant $C$. 
Since $k_n = O(1)$, this is valid since $\ln p_n / n \to 0$.
Set $D_n = \sqrt{n p_n} a_n$ with some $a_n \to \infty$ arbitrarily slowly, which is $o(n/\ln p_n)$ since $p_n (\ln p_n)^2/n \to 0$, and observe that
\[
\P\lf(  \max_{1 \le j \le p_n} \sup_{\btheta \in \Theta_n}\sum_{k=1}^{p_n} \lf| \frac{\alpha_{k,n}}{\alpha_{j,n}} \frac{\partial}{\partial \theta_k} \phi(\bU; \btheta)_j \ri|    > D_n\ri) \le p_n \frac{\max_{1 \le j \le p_n} \E\left[ \sup_{\btheta \in \Theta_n} \left( \sum_{k=1}^{p_n} \lf| \frac{\alpha_{k,n}}{\alpha_{j,n}} \frac{\partial}{\partial \theta_k} \phi(\bU; \btheta)_j \ri|\right)^2 \right] }{n p_n a_n^2 } = O\left( \frac{1}{n a_n^2} \right) = o\left( \frac 1 n  \right).
\]
For \ref{A:Asymp}, note that $\| A_n [\phi(\bu; \btheta^* +  \bDelta) - \phi(\bu;\btheta^* +  \bDelta')]\|/ \| \bDelta - \bDelta'\| \lesssim \| \nabla_{\btheta} \phi(\bu; \bttheta) \|$ with some $\bttheta$ on the segment between $\btheta^* + \bDelta, \btheta^* + \bDelta'$.
The assumptions imply that both $\sup_{\btheta \in \Theta_n} \E[\| \nabla_{\btheta} \phi(\bU; \btheta) \|_1^2]$ and $\sup_{\btheta \in \Theta_n} \E[\| \nabla_{\btheta} \phi(\bU; \btheta) \|_\infty^2]$ are $O(1)$, since each row and each column of $\nabla_{\btheta} \phi(\bu; \btheta)$ contains only finitely many non-zero entries that are uniformly bounded.
Since $\| B \|_2^2 \le \| B \|_1 \| B \|_\infty$, this implies that the left-hand side of the first condition in \ref{A:Asymp} bounded and the condition is satisfied since $p_n^2 \ln p_n / n \to 0$.
For the second condition, set $D_n = n^{1/4}$ with $a_n \to \infty$ arbitrarily slowly, which is $o(\sqrt{n} /(\tilde r_n p_n))$ since $p_n^2 (\ln p_n)^2/n \to 0$.
We can then bound the probability in the second expression in \ref{A:Asymp} by 
\[
\frac{\| A_n \| \, \E\left[ \sup_{\btheta \in \Theta_n} \|  \nabla_{\btheta} \phi(\bU; \btheta)\|^4 \right]}{n a_n^{4}} \le \frac{\| A_n \| \sqrt{ \E\left[ \sup_{\btheta \in \Theta_n} \|  \nabla_{\btheta} \phi(\bU; \btheta)\|^4_1 \right] \E\left[ \sup_{\btheta \in \Theta_n} \|  \nabla_{\btheta} \phi(\bU; \btheta)\|^4_\infty \right] }}{n a_n^{4}} = O\left( \frac{1}{n a_n^4} \right) = o\left( \frac 1 n  \right).
\]
Regarding the third condition, note that the absolute value of each entry of $J(\btheta^* + r_n C_n \bDelta) - J(\btheta^*)$ is bounded by $r_n C_n \| \bDelta \| \sup_{\btheta \in \Theta_n}  \| \nabla_{\btheta} ( \frac{\partial}{ \partial \theta_j}  \E[\phi(\bU; \btheta)_k]) \|$.
Each of $\sup_{\btheta \in \Theta_n}  \| \nabla_{\btheta} ( \frac{\partial}{ \partial \theta_j}  \E[\phi(\bU; \btheta)_k]) \|$ is uniformly bounded by the assumptions, so $\sup_{\bDelta} \| J(\btheta^* + r_n C_n \bDelta) - J(\btheta^*)\| = O(r_n C_n p_n) = o((\sqrt{n} \tilde r_n)^{-1})$ since $p_n^2 \ln p_n / n \to 0$.
\ref{A:Asymp2} follows since $\E[\| \phi(\bU; \btheta^*) \|^4 ] = O(p_n^2)$ by the assumptions and $p_n^2/n \to 0$.}

\section{Lemmas}

{\color{red}
\begin{lemma} \label[lemma]{lem:eta_n}
  Under assumption \ref{A:phi-moments} and $p_n \to \infty$, it holds that
$ \left\| \P_n \phi_U(\btheta^*) \right\|_{\infty} =O_p(\sigma_n \sqrt{\ln p_n / n}).$
 \end{lemma}
 \begin{proof}
  Denote $\eta_n = 2\sigma_n \sqrt{\ln p_n/n}$ and $B_n = \sigma_n \sqrt{n/(4\ln p_n)}$.
  Using \cref{lem:truncation} and \ref{A:phi-moments}, we obtain $\left\| \P_n \phi_U(\btheta^*) \right\|_{\infty} \le \left\| (\P_n - P)\phi_U(\btheta^*) \ind_{\|\phi_U(\btheta^*)\|_\infty \le B_n}\right\|_{\infty} + o_p(\eta_n)$.
  Further, the union bound and Bernstein's inequality give
  \begin{align*}
     \Pr\left(\left\| (\P_n - P) \phi_U(\btheta^*) \ind_{\|\phi_U(\btheta^*)\|_\infty \le B_n}  \right\|_{\infty} > \eta_n\right)
   &   \le  2\, p_n \max_{1 \le k \le p_n} \exp\left(-\frac{\frac 1 2  \eta_n^2 }{\frac 1 n \sigma_n^2 + \frac 1 3 \eta_n B_n/n}\right)  \le  2 \exp\left(\ln p_n -\frac{\eta_n^2 n}{2 \sigma_n^2 + \eta_n B_n}\right) \\
    & =  2 \exp\left(\ln p_n -\frac{\eta_n^2 n}{3 \sigma_n^2 }\right) =  2 \exp\left(\ln p_n - \frac 4 3 \ln p_n\right)  = o(1)
  \end{align*}
  and therefore $ \left\| (\P_n - P)\phi_U(\btheta^*) \ind_{\|\phi_U(\btheta^*)\|_\infty \le B_n}\right\|_{\infty} = O_p(\eta_n)$, which implies the claim.
 \end{proof}
}
\begin{lemma}
\label{lem:emp_proc}
Under assumption \ref{A:emp_proc}, it holds that
\[
\sup_{f \in \Fcal_n} | (\P_n - P) f| = O_p\lf( \frac{M_n c_n \sqrt{k_n + \ln p_n}}{\sqrt{n}} + \frac{D_n c_n (k_n + \ln p_n) }{n} \ri),
\]
where
$
\Fcal_n = \{ f_{\bDelta,j}(\bu) =\alpha_{j,n}^{-1}  [ \phi(\bu ; \btheta^* + \bDelta)_j - \phi(\bu; \btheta^*)_j ]: \forall j' = 1, \ldots, p_n : | \Delta_{j'} | \le c_n \alpha_{j',n},   j  =1, \ldots, p_n \}
$
with $c_n = r_n C$ for some $C < \infty$ and $D_n, M_n$ and $k_n$ as defined in \ref{A:emp_proc}.
\end{lemma}

\begin{proof}
Define
\[
F_n(\bu) \coloneqq\max_{1 \le j \le p_n} \sup_{\btheta \in \Theta_n} \sum_{k=1}^{p_n} \lf| \frac{\alpha_{k,n}}{\alpha_{j,n}} \frac{\partial}{\partial \theta_k} \phi(\bu; \btheta)_j  \ri|.
\]
$F_n$ is an envelope for the functions in $c_n^{-1} \Fcal_n$, i.e., $\sup_{f \in \Fcal_n} c_n^{-1} | f(\bu) | \le  F_n(\bu)$, since a Taylor expansion and Hölder's inequality with $p = 1, q = \infty$ give, with some $\btheta$ on the segment between $\btheta^*$ and $\btheta^* + \bDelta$,
\[
| f_{\bDelta,j}(\bu) | =  \alpha_{j,n}^{-1}\lf| \bDelta^T \nabla_{\btheta} \phi(\bu; \btheta)_j \ri| \le c_n \sum_{k=1}^{p_n} \lf| \frac{\alpha_{k,n}}{\alpha_{j,n}} \frac{\partial}{\partial \theta_k} \phi(\bu; \btheta)_j  \ri|.
\] 
\cref{lem:truncation} gives
$
\sup_{f \in \Fcal_n} | (\P_n - P) f|  \le \sup_{f \in \Fcal_n} | (\P_n - P) f \ind_{F_n \le D_n}| + o_p (\sqrt{\sup_{f \in \Fcal_n} P f^2/n })
$
with $D_n$ as defined in \ref{A:emp_proc}.
The second term is $o_p(M_n c_n / \sqrt{n})$, since the definition of $M_n$ yields
\begin{align*}
 \sup_{f \in \Fcal_n} P f^2 
 \le \max_{1 \le j \le p_n} \sup_{\btheta \in \Theta_n} c_n^2 \ \E\lf[ \lf(\sum_{k=1}^{p_n}\lf| \frac{\alpha_{k,n}}{\alpha_{j,n}} \frac{\partial}{\partial \theta_k} \phi(\bU; \btheta)_j  \ri| \ri)^2 \ri]  \le c_n^2 M_n^2.
\end{align*}

For the first term, we proceed as following:
Fix $j$ and consider $\Fcal_n^{(j)} = \{ f_{\bDelta, j} \ind_{F_n \le D_n}: \forall j' = 1, \ldots, p_n : | \Delta_{j'} | \le c_n \alpha_{j',n} \}$.
We need to bound the covering numbers $N(\eps, \Fcal_n^{(j)}, L_2(P))$ and $N(\eps, \Fcal_n^{(j)}, \| \cdot \|_{\infty})$ \citep[see][Chapter 2.1.1 for a definition]{vdV}, where $\| f \|_{L_2(P)}^2 = P f^2$ and $\| f \|_{\infty} = \sup_{\bu} | f(\bu) |$.

Denote $\tilde I_{n,j} \subseteq \{1, \ldots, p_n \}$ the set of indices $k$ such that $\partial  \phi(\bu; \btheta)_j / \partial \theta_k = 0$ for all $\btheta$ and $\bu$ and denote $I_{n, j} = \{ 1, \ldots, p_n \} \setminus \tilde I_{n,j}$.
By the above definition of $k_n$, it holds $| I_{n,j} | \le k_n$ for each $j$, i.e., $\nabla_{\btheta} \phi(\bu; \btheta)_j $ has at most $k_n$ non-zero entries.
It holds that
\begin{align*}
 & \| f_{\bDelta ,j} - f_{\bDelta', j} \|^2_{L_2(P)} \le \sup_{\btheta \in \Theta_n} \E \lf[ ( \alpha_{j,n}^{-1} (\bDelta - \bDelta')^T \nabla_{\btheta} \phi(\bU; \btheta)_j)^2 \ri] \le  M_n^2 \max_{k \in I_{n,j}} \lf| \frac{\Delta_k}{\alpha_{k,n}} -  \frac{\Delta'_k}{\alpha_{k,n}}  \ri|^2 \quad \text{and} \\
  & \| ( f_{\bDelta ,j} - f_{\bDelta', j}) \ind_{F_n \le D_n} \|_{\infty} 
   = \sup_{\bu: F_n (\bu) \le D_n} | f_{\bDelta ,j}(\bu) - f_{\bDelta', j} (\bu) | \\
&  \le \max_{k \in I_{n,j}} \lf| \frac{\Delta_k}{\alpha_{k,n}} -  \frac{\Delta'_k}{\alpha_{k,n}}  \ri|  \sup_{\bu: F_n (\bu) \le D_n, \btheta \in \Theta_n} \sum_{k=1}^{p_n} \lf| \frac{\alpha_{k,n}}{\alpha_{j,n}} \frac{\partial}{\partial \theta_k} \phi(\bu; \btheta)_j  \ri|  \le  \max_{k \in I_{n,j}} \lf| \frac{\Delta_k}{\alpha_{k,n}} -  \frac{\Delta'_k}{\alpha_{k,n}}  \ri| D_n.
\end{align*}
Note that $| \Delta_k / \alpha_{k,n} | \le c_n$ for all $k$ by the definition of $\bDelta$ and therefore $| \Delta_k / \alpha_{k,n} \, - \, \Delta'_k / \alpha_{k,n} | \le 2 c_n$.
Now let $ \bDelta_1, \ldots,  \bDelta_N$ be the centers of an $\eta$-covering of $\{ \bDelta \in \R^{p_n}: \| \bDelta \|_{\infty} \le 2 c_n, \Delta_{j'} = 0 \forall j' \notin I_{j,n}\}$ w.r.t.~$\| \cdot \|_{\infty}$, which we can find with $N = (2 c_n/\eta)^{k_n}$.
Then, $f_{\bDelta_1, j}, \ldots, f_{\bDelta_N,j}$ are the centers of an $M_n \eta$-covering of $\Fcal_n^{(j)}$ w.r.t.~$L_2(P)$ and a $D_n \eta$-covering of $\Fcal_n^{(j)}$ w.r.t.~$\| \cdot \|_{\infty}$.
Choosing $\eta = \eps / M_n$ and $\eta = \eps / D_n$, respectively, gives $N(\eps, \Fcal^{(j)}_n, L_2(P)) \le ( 2 M_n c_n /\eps)^{k_n}$ and $N(\eps, \Fcal^{(j)}_n, \| \cdot \|_{\infty}) \le ( 2  D_n c_n/\eps)^{k_n}$.
Taking the union over the coverings of $\Fcal_n^{(j)}, j =1, \ldots, p_n$, we obtain a covering of $\Fcal_n$ and
\[
\ln N(\eps, \Fcal_n,  L_2(P)) \le \ln p_n + k_n \ln ( 2  M_nc_n/\eps), \quad \ln N(\eps, \Fcal_n,  \| \cdot \|_{\infty}) \le \ln p_n + k_n \ln ( 2   D_nc_n/\eps).
\]
Theorem 2.14.21 in \citet{vdV} yields
\[
\E \lf[\sup_{f \in \Fcal_n} | (\P_n - P) f|  \ri] \lesssim  \frac{\int_0^{M_n c_n } \sqrt{1 + \ln N(\eps, \Fcal_n,  L_2(P))} d \eps}{\sqrt{n}} + \frac{\int_0^{D_n c_n} [1 +\ln N(\eps, \Fcal_n,  \| \cdot \|_{\infty}) ] d \eps}{n},
\]
where $\lesssim$ means ``bounded up to a universal constant''.
The change of variables $t = \eps/(M_n c_n)$ and $t = \eps/(D_n c_n)$ gives
\[
\E \lf[\sup_{f \in \Fcal_n} | (\P_n - P) f|  \ri] = O\left( \frac{M_n c_n \sqrt{k_n + \ln p_n}}{\sqrt{n}} + \frac{D_n c_n (k_n + \ln p_n) }{n} \right).
\]
The claim now follows from Markov's inequality.
\end{proof}
 
\begin{lemma}
\label{lem:margins1}
Under \ref{A:Var_h1}, \ref{A:AsympSemi_2}, \ref{A:SemiPar_I2} and \textcolor{red}{$d_n^{3 + 4b}(\ln d_n)^2/n =O(1)$}, it holds that
$
 \| \P_n \phi_X(\btheta^*)  - \P_n \phi_U(\btheta^*) \|_{\infty} = O_p(\sqrt{\ln p_n / n}).
$
\end{lemma}

\begin{proof}
A Taylor expansion gives
\begin{align*}
\P_n [\phi_X (\btheta^*) - \phi_U (\btheta^*) ] 
& =  \frac 1 n \sumin \phi ( F_n(\bX_i) ; \btheta^*) - \phi ( F(\bX_i) ; \btheta^*) =  \frac 1 n \sumin \underbrace{\nabla_{\bu} \phi(F(\bX_i); \btheta^*)}_{\in \R^{p_n \times d_n}} (F_n(\bX_i ) - F(\bX_i )) + I_2 \\
& =
\frac{1}{n ( n+1)} \sumin \sum_{i'=1}^n  \nabla_{\bu} \phi(F(\bX_i); \btheta^*) 
 \begin{pmatrix}
 \ind (X_{i'1} \le X_{i1}) - F_1(X_{i1}) \\
 \vdots \\
  \ind (X_{i'd_n} \le X_{id_n}) - F_{d_n}(X_{id_n})
 \end{pmatrix}+ I_2   = I_1 + I_2,
\end{align*}
where
\begin{equation}
\label{eq:lem1_I2}
I_2  = \frac 1 n  \sumin I_{2,i} = \frac 1 n \sumin \begin{pmatrix}
 (F_n(\bX_i ) - F(\bX_i ))^\top  \nabla^2_{\bu} \phi(\btu_i ; \btheta^*)_1  (F_n(\bX_i ) - F(\bX_i )) \\
 \vdots \\
   (F_n(\bX_i ) - F(\bX_i ))^\top\nabla^2_{\bu} \phi(\btu_i; \btheta^*)_{p_n}  (F_n(\bX_i ) - F(\bX_i )) 
 \end{pmatrix} 
\end{equation}
with some $\btu_i$ on the segment between $F_n(\bX_i), F(\bX_i)$, and $\phi(\bu; \btheta)_k$ denotes the $k$-th entry of $\phi(\bu; \btheta)$.
Recall the definitions in \eqref{eq:def_Y}.
By the definition of $h^a$, we have $I_1 = (n (n+1))^{-1} \sumin \sum_{i'=1}^n h^a(\bX_i, \bX_{i'})$.
 \emph{Hoeffding's decomposition} yields
 \begin{align*}
  \frac{1}{n ( n+1)}  \sumin \sum_{i'=1}^n   h^a(\bX_i, \bX_{i'}) & =  \frac{1}{n ( n+1)}  \sumin \sum_{i'=1}^n   h(\bX_i, \bX_{i'}) =  \frac{1}{n ( n+1)}  \sumin  h(\bX_i, \bX_i)  +  \frac{1}{n ( n+1)}  \sumin \sum_{i'=1, i' \neq i}^n   h(\bX_i, \bX_{i'}) \\
  & =  \frac{1}{n ( n+1)}  \sumin  h(\bX_i, \bX_i)  +  \frac 2 n \sumin  h_1(\bX_i)   +    \frac{1}{n ( n+1)}  \sumin \sum_{i'=1, i' \neq i}^n h(\bX_i, \bX_{i'}) -   h_1(\bX_i) -   h_1(\bX_{i'}).
 \end{align*}
 Since $| h(\bX_i, \bX_i)_k | = O_p(1)$ uniformly for all $k = 1, \ldots, p_n$ by \ref{A:AsympSemi_2}, we have
 \[
 \lf\|   \frac{1}{n ( n+1)}  \sumin  h(\bX_i, \bX_i) \ri\|_{\infty} \le  \lf\|   \frac{1}{n ( n+1)}  \sumin  h(\bX_i, \bX_i) \ri\|_2 = O_p(\sqrt{p_n} / n) = o_p(\sqrt{\ln p_n / n}).
 \]
\cref{lem:eta_n} and \ref{A:Var_h1} yield $\| \frac 2 n \sumin  h_1(\bX_i) \|_{\infty} = O_p(\sqrt{\ln p_n / n})$.
\cref{lem:hoeffd_remainder}, \ref{A:Var_h1} and \ref{A:AsympSemi_2} imply 
\[
 \lf\|   \frac{1}{n ( n+1)}  \sumin \sum_{i'=1, i' \neq i}^n h(\bX_i, \bX_{i'}) -   h_1(\bX_i) -   h_1(\bX_{i'}) \ri\|_{\infty} = o_p(\sqrt{\ln p_n / n}),
 \]
 so $\| I_1 \|_{\infty} = O_p(\sqrt{\ln p_n / n})$.
\cref{lem:I2_margins_cons}, \ref{A:SemiPar_I2} and $d_n^{3 + 4b}(\ln d_n)^2/n =O(1)$ yield $\| I_2 \|_{\infty} = O_p(\sqrt{\ln p_n / n})$.
Since
 \[
  \| \P_n \phi_X(\btheta^*)  - \P_n \phi_U(\btheta^*) \|_{\infty}  \le \| I_1 \|_{\infty} +  \| I_2 \|_{\infty},
 \]
 this implies the claim.
\end{proof}

\begin{lemma}
\label{lem:hoeffd_remainder}
Under assumption \ref{A:Var_h1} and \ref{A:AsympSemi_2}, with $h(\bX_i, \bX_{i'})$ and $h_1(\bx)$ as defined in \eqref{eq:def_Y}, it holds that
\[
\| II \|_2 =  \lf\| \frac{1}{n (n-1)} \sumin \sum_{i'  =1, i' \neq i}^n h(\bX_i,\bX_{i'}) -   h_1(\bX_i) -   h_1(\bX_{i'})  \ri\|_2 = O_p(\sqrt{p_n}/n).
\]
\end{lemma}

\begin{proof}
Define $\eta (\bX_i, \bX_{i'}) \coloneqq h(\bX_i,\bX_{i'}) -   h_1(\bX_i) -   h_1(\bX_{i'}).$
It holds that
\begin{align}
\E\lf[  \|II \|_2^2 \ri] & = \frac{1}{n^2(n-1)^2} \sum_{k=1}^{p_n} \E\lf[ \lf( \sumin \sum_{i' = 1, i' \neq i}^n \eta(\bX_i, \bX_{i'})_k  \ri)^2 \ri] \nonumber \\
& = \frac{1}{n^2(n-1)^2} \sum_{k=1}^{p_n} \sum_{i_1 = 1}^n \sum_{i'_1 = 1 ,  i'_1 \neq i_1}^n  \sum_{i_2 = 1}^n \sum_{i'_2 = 1, i'_2 \neq i_2}^n\E\lf[ \eta(\bX_{i_1}, \bX_{i'_1})_k  \eta(\bX_{i_2}, \bX_{i'_2})_k \ri] .  \label{eq:sum_eta}
\end{align}
In this sum, all terms with $\{ i_1, i'_1 \} \neq \{ i_2, i'_2\}$ are zero.
To see this, first note that
\begin{align*}
\E[ \eta(\bX_i, \bX_{i'})_k   | \bX_i] & = \E[ h (\bX_i, \bX_{i'})_k | \bX_i] - \E[ h_1(\bX_i)_k | \bX_i] -  \E[ h_1(\bX_{i'})_k | \bX_i] \\
&  =  h_1(\bX_i)_k -  h_1(\bX_i)_k -   \E[ h_1(\bX_{i'})_k ] =  h_1(\bX_i)_k -  h_1(\bX_i)_k - 0  = 0,
\end{align*}
where $  \E[ h_1(\bX_{i'})_k ] = 0$ follows from $\E[ \ind (X_{i' l} \le X_{il}) - F_1(X_{il})  | X_{il}] = 0$ for all $l = 1, \ldots, d_n,$ $i,i' = 1, \ldots, n$.
Now, with $i_1 = i_2$ but $i'_1 \neq i'_2$ (the case $i_1 \neq i_2$, $i'_1 = i'_2$ works the same), we have
\begin{align*}
\E\lf[\eta(\bX_{i_1}, \bX_{i'_1})_k  \eta(\bX_{i_2}, \bX_{i'_2})_k \ri] & = 
\E\lf[\E[ \eta(\bX_{i_1}, \bX_{i'_1})_k  \eta(\bX_{i_1}, \bX_{i'_2})_k  | \bX_{i_1}, \bX_{i'_2}]\ri] = \E\lf[\E[ \eta(\bX_{i_1}, \bX_{i'_1})_k   | \bX_{i_1}, \bX_{i'_2}] \eta(\bX_{i_1}, \bX_{i'_2})_k \ri] \\
& =  \E\lf[\E[ \eta(\bX_{i_1}, \bX_{i'_1})_k   | \bX_{i_1}] \eta(\bX_{i_1}, \bX_{i'_2})_k \ri] = 0.
\end{align*}

Each of the remaining $O(p_n n^2)$ terms in \eqref{eq:sum_eta} is $O(1)$ since \ref{A:Var_h1} and\ref{A:AsympSemi_2} imply $\E[\eta (\bX_i, \bX_{i'})_k^2] = O(1)$ uniformly for all $k, i, i'$.
Therefore $\E[  \|II \|_2^2] = O(p_n/n^2)$, so $\| II \|_2 = O_p(\sqrt{p_n}/n)$.
\end{proof}

\begin{lemma}
\label{lem:I2_margins_cons}
Under assumptions \ref{A:SemiPar_I2}, with $I_2$ as defined in \eqref{eq:lem1_I2}, it holds that \textcolor{red}{
$
\| I_2 \|_{\infty} = O_p\lf(\sqrt{ d_n^{(3 + 4b)} (\ln d_n)^2 \ln p_n /n^2 } \ri).
$}
\end{lemma}
\begin{proof}
We use $| \langle A, B \rangle_F | \le  \| A \|_F \| B \|_F$ for matrices $A,B$ with $\langle A, B \rangle_F = \sum_{ml} a_{ml} b_{ml}$ and $\| A \|_F = \sqrt{\langle A, A \rangle_F}$. 
The inequality follows directly from Cauchy-Schwarz inequality when treating a matrix $A \in \R^{l \times m }$ as a vector in $\R^{lm}$.
$I_2$ depends on $\btu_i, i = 1, \ldots, n$ on segments between $F_n(\bX_i)$ and $F(\bX_i)$.
\cref{lem:emp_cdfs} implies that we can bound terms involving these $\btu_i$ by taking the supremum over $G \in \Gcal_n$ as defined in \ref{A:SemiPar_I2}.
\textcolor{red}{Let $w(s) = s^\gamma (1-s)^\gamma$ for some $\gamma \in (0,1)$.
Using $\| f \|_{\infty} = \sup_{x \in \R} | f(x) |$ for a function $f \colon \R \to \R$ and the definition of $\partial_{ml} \phi(\bu; \btheta)_k$ in \ref{A:SemiPar_I2}, we obtain
\begin{align*}
\| I_2 \|_{\infty} & \le \frac{1}{n} \sumin  \sup_{1 \le k \le p_n, G \in \Gcal_n}  \sum_{m = 1}^{d_n}  \sum_{l = 1}^{d_n}  |\partial_{ml} \phi(G(\bX_i); \btheta^*)_k  (F_{nm}(X_{im}) - F_m(X_{im})) (F_{nl}(X_{il}) - F_l(X_{il}))  | \\
& \le \frac{1}{n} \sumin \sup_{1 \le k \le p_n,  G \in \Gcal_n} \sum_{m = 1}^{d_n}  \sum_{l = 1}^{d_n}  \underbrace{  |   \partial_{ml} \phi(G(\bX_i); \btheta^*)_k  w(F_m(X_{im})) w(F_l(X_{il})) | }_{a_{ml, k}(G, \bX_i)}  \underbrace{ \left\| \frac{F_{nm} - F_m }{w(F_m)}\right\|_{\infty}  \left\| \frac{F_{nl} - F_l }{w(F_l)}\right\|_{\infty} }_{b_{ml}}\\
& \le  \| B \|_F \,\frac{1}{n} \sumin \sup_{1 \le k \le p_n,  G \in \Gcal_n} \| A_{G, k}(\bX_i) \|_F 
\end{align*}
with matrices $A_{G, k}, B \in \R^{d_n \times d_n}$.
We have, with $b$ as defined in \ref{A:SemiPar_I2},
\[
\| B \|_F^2 = \sum_{m = 1}^{d_n}  \sum_{l = 1}^{d_n} \left\| \frac{F_{nm} - F_m }{w(F_m)}\right\|_{\infty}^2  \left\| \frac{F_{nl} - F_l }{w(F_l)}\right\|_{\infty}^2 \le d_n^2 \sup_{1 \le m \le d_n}\left\| \frac{F_{nm} - F_m }{w(F_m)}\right\|_{\infty}^4 = O_p\left(\frac{d_n^{(2 + 4b)} (\ln d_n)^2}{n^2}\right)
\]
by \cref{lem:emp_cdfs}.
The claim now follows from \ref{A:SemiPar_I2}, since
\[
\frac{1}{n} \sumin \sup_{1 \le k \le p_n,  G \in \Gcal_n} \| A_{G, k}(\bX_i) \|_F 
= O_p\left(\E\left[\sup_{1 \le k \le p_n,  G \in \Gcal_n} \| A_{G, k}(\bX) \|_F \right]  \right) = O_p(\sqrt{d_n \ln p_n}).
\]
}
\end{proof}
 \begin{lemma}
\label{lem:margins_cons_remainder}
Under assumptions \ref{A:SemiPar_r} and \ref{A:alpha_max}, with $f_j(\bDelta)$ as defined in \eqref{eq:def_f}, it holds that \textcolor{red}{
\[
\max_{1 \le j \le p_n}  \sup_{ \bDelta \in \Theta_n^\Delta } (r_n C\alpha_{j,n})^{-1}|  \P_n f_j( r_n C \bDelta) | = O_p\left(\sqrt{\frac{d_n^{(2+2b)} \ln d_n}{n}} \right) .
\]}
\end{lemma}

\begin{proof}
With some $\btU_i$ on the segment between $F_n(\bX_i)$ and $F(\bX_i)$, it holds that
\textcolor{red}{
\begin{align*}
 f_j( r_n C \bDelta ) & =  \nabla_{\bu} \lf[\phi(\btU_i; \btheta^* + r_n C \bDelta)_j - \phi(\btU_i; \btheta^*)_j \ri] (F_n(\bX_i) - F(\bX_i) )\\
 & \le  \lf\| \left( w(F_m(X_{im})) \frac{\partial}{\partial u_m} \lf[\phi(\btU_i; \btheta^* + r_n C \bDelta)_j - \phi(\btU_i; \btheta^*)_j \ri] \right)_{m=1,\ldots, d_n} \ri\|_2 \cdot \left\| \left( \frac{F_{nm}(X_{im}) - F_m(X_{im})}{w(F_m(X_{im})) } \right)_{m=1, \ldots, d_n}\right\|_2.
\end{align*}
By \ref{A:SemiPar_r} and \ref{A:alpha_max}, there exist non-negative, real-valued functions $\psi_m(\bu)$, $m = 1,\ldots, d_n$, such that for all $j = 1, \ldots, p_n$,
\[
 (r_n C\alpha_{j,n})^{-1}  \lf|  \frac{\partial}{\partial u_m} \lf[\phi(\btU_i; \btheta^* + r_n C \bDelta)_j - \phi(\btU_i; \btheta^*)_j \ri] \ri|
\le \frac{\max_{k \in I_{n,j} } \alpha_{k,n}}{\alpha_{j,n}} \psi_m(\btU_i) 
\]
and $\max_{1 \le j \le p_n} (\max_{k \in I_{n,j}  } \alpha_{k,n}/ \alpha_{j,n}) = O(1)$ by \ref{A:alpha_max}.
We therefore have
\begin{align*}
&  \max_{1 \le j \le p_n}  \sup_{ \bDelta \in \Theta_n^\Delta , | \Delta_j | = \alpha_{j,n}} (r_n C\alpha_{j,n})^{-1}    |  \P_n f_j( r_n C \bDelta) | \\
& \le \sqrt{d_n} \sup_{x \in \R, 1 \le m \le d_n  } \frac{| F_{nm}(x) - F_m(x)| }{w(F_m(x)) } \ \frac{1}{n} \sumin \lf\| \left( w(F_m(X_{im})) \psi_m(\btU_i)\right)_{m=1,\ldots, d_n} \ri\|_2 O(1) \\
& = O_p\left(\sqrt{\frac{d_n^{(1+2b)} \ln d_n}{n}} \right) O_p\left(\E\left[\sup_{G \in \Gcal_n} \lf\| \left( w(F_m(X_{m})) \psi_m(G(\bX))\right)_{m=1,\ldots, d_n} \ri\|_2 \right] \right) = O_p\left(\sqrt{\frac{d_n^{(2+2b)} \ln d_n}{n}} \right) 
\end{align*}
by \ref{A:SemiPar_r} and \cref{lem:emp_cdfs}, with $b$ as defined in \ref{A:SemiPar_I2}..
}
\end{proof}

\begin{lemma}
\label{lem2_new}
Under assumption \ref{A:SemiPar2}, for $I_2$ as defined in \eqref{eq:lem1_I2}, it holds that
$
\| I_2 \| =  O_p\lf( \sqrt{ d_n^{(2 + 4b)} p_n (\ln d_n)^2 /n^2 }\ri).
$
\end{lemma}

\begin{proof}
We slightly adapt \cref{lem:I2_margins_cons} to obtain tight bounds in $\| \cdot \|_2$ norm:
\textcolor{red}{
Using $\| \bx \| = \sup_{\| \bDelta \| = 1} \langle \bx, \bDelta \rangle$, we have, with $\Gcal_n$ as defined in \ref{A:SemiPar_I2} and matrices $A_G, B \in \R^{d_n \times d_n}$,
\begin{align*}
 \| I_2 \| & \le  \frac{1}{n} \sumin \sup_{\| \bDelta \| = 1, G \in \Gcal_n} \sum_{k = 1}^{p_n} \sum_{m = 1}^{d_n}  \sum_{l = 1}^{d_n}  \partial_{ml} \phi(G(\bX_i); \btheta^*)_k  (F_{nm}(X_{im}) - F_m(X_{im})) (F_{nl}(X_{il}) - F_l(X_{il})) \Delta_k \\
 & \le  \frac{1}{n} \sumin \sup_{\| \bDelta \| = 1, G \in \Gcal_n} \sum_{m = 1}^{d_n}  \sum_{l = 1}^{d_n}  \left\| \frac{F_{nm} - F_m }{w(F_m)}\right\|_{\infty}  \left\| \frac{F_{nl} - F_l }{w(F_l)}\right\|_{\infty} \sum_{k = 1}^{p_n}  |   \partial_{ml} \phi(G(\bX_i); \btheta^*)_k  w(F_m(X_{im})) w(F_l(X_{il})) |  \Delta_k \\
 & \le  \frac{1}{n} \sumin \sup_{\| \bDelta \| = 1, G \in \Gcal_n} \sum_{m = 1}^{d_n}  \sum_{l = 1}^{d_n}  \underbrace{\left\| \frac{F_{nm} - F_m }{w(F_m)}\right\|_{\infty}  \left\| \frac{F_{nl} - F_l }{w(F_l)}\right\|_{\infty} }_{b_{ml}} \underbrace{ \| \ba_{ml}(G, \bX_i) \|}_{a_{ml}(G, \bX_i)}  \| \Delta \| \\
 & \le \| B \|_F \, \frac{1}{n} \sumin \sup_{ G \in \Gcal_n} \| A_G(\bX_i) \| _F
\end{align*}
where $\ba_{ml}(G, \bx) = \lf( |   \partial_{ml} \phi(G(\bx); \btheta^*)_k  w(F_m(x_m)) w(F_l(x_l) |  \ri)_{k=1, \ldots, p_n} \in \R^{p_n}$.
Now \ref{A:SemiPar2} ($\E[\sup_{ G \in \Gcal_n} \| A_G(\bX) \| _F] = \sqrt{p_n}$) and \cref{lem:I2_margins_cons} ($\| B \|_F  = O_p(\sqrt{ d_n^{(2 + 4b)} (\ln d_n)^2 /n^2 })$) imply the claim.
}
\end{proof}
 \begin{lemma}
\label{lem:emp_cdfs}
For any random variable $\bX \in \R^{d_n}$ with continuous c.d.f.s $F_m(x)$ and empirical c.d.f.s $F_{nm}(x), m = 1, \ldots, d_n$, $\gamma \in (0,1)$ and $b > \gamma$, with $w(s) = s^\gamma (1-s)^\gamma$, it holds that
\[
\sup_{x \in \R, 1 \le m \le d_n} \frac{ |F_{nm}(x) - F_m(x) |}{w(F_m(x))} = O_p \lf( d_n^b \sqrt{\frac{\ln d_n}{ n}} \ri).
\]
\end{lemma}

\begin{proof}
Define the function class
$
\Gcal = \lf\{ f_{m, t} (\bx) = \ind ( x_m \le  t) / w(F_m ( x_m)) : t \in \R, m = 1, \ldots, d_n\ri\} .
$
Now
\[
\sup_{x \in \R, 1 \le m \le d_n} \frac{ |F_{nm}(x) - F_m(x) |}{w(F_m(x))} = \sup_{f \in \Gcal} | (\P_n - P) f |.
\]
We need to derive the bracketing number $\Ncal_{[ \, ]}(\eps, \Gcal, L_2(P))$ \citep[see][Chapter 2.1.1 for a definition]{vdV} of this function class.
Denote $\Gcal_m = \{ f_t (\bx) = \ind ( x_m \le  t) / w(F_m ( x_m)) : t \in \R \} $.
It holds $\Ncal_{[ \, ]}(\eps, \Gcal_m, L_2(P)) = O(\eps^{-k})$ for some $k < \infty$.
This follows from \citet[][proof of Lemma A.1]{Nagler22}, since $\| \cdot \|_\beta = \| \cdot \|_{L_2(P)}$ for independent random variables.
Since $\Gcal = \bigcup_m \Gcal_m$, a union of $\eps$-bracketings of $\Gcal_m, m = 1, \ldots, d_n$, is an $\eps$-bracketing of $\Fcal$ and therefore $\Ncal_{[ \, ]}(\eps, \Gcal, L_2(P)) = O(d_n \eps^{-k})$ for some $0 < k < \infty$.

Now we derive $\| G \|_{L_2(P)} = O(d_n^{b})$ with some $b > \gamma$ for the envelope $G (\bx ) = \max_{1 \le m \le d_n} 1/w(F_m(x_m))$ of $\Gcal$.
With $U \sim U[0,1]$ and any $a$ such that $2a < 1/\gamma$, it holds that
\[
\Pr \lf(G(\bX)^2 > t \ri) \le d_n \Pr \lf(\frac{1}{w(U)^2} > t \ri) \le d_n \frac{\E [w(U)^{-2a}]}{t^a} \lesssim d_n t^{-a},
\]
since $F_m(X_m) \sim U[0,1]$ and $\E [(U^{-c} (1-U)^{-c}] < \infty$ for $c < 1$.
It holds that
\[
\E\lf[ G(\bX)^2 \ri]  = \eta + \int_\eta^\infty \Pr \lf( G(\bX)^2  > t \ri) d t \lesssim \eta + \int_\eta^\infty d_n t^{-a} d t \lesssim  \eta + d_n \eta^{1 - a} . 
\]
With $\eta = d_n^{1/a}$, we obtain
\[
 \E\lf[ G(\bX)^2 \ri]  \lesssim  d_n^{1/a }+ d_n \cdot d_n^{(1-a)/a} =  d_n^{1/a } +  d_n^{1/a} = O\lf(d_n^{1/a} \ri)
 \]
 and therefore $\| G \|_{L_2(P)} = O(d_n^b)$ for any with $b = 1/(2a) > \gamma$.
For the bracketing integral \citep[][Chapter 2.14.2]{vdV}, we obtain with $\| G \|_{L_2(P)} \ge 1$
\[
J_{[ \, ]}(1, \Gcal \mid G, L_2(P)) \le \int_0^1 \sqrt{1 +  \ln \Ncal_{[ \, ]}(\eps ,  \Gcal, L_2(P))  } d \eps = O(\sqrt{\ln d_n}).
\]
Now, Theorem 2.14.16 in \citet[][Chapter 2.14.2]{vdV} yields 
$
\sup_{f \in \Gcal} | (\P_n - P) f | = O_p(\sqrt{ \ln d_n} \, d_n^b/\sqrt{n}) .
$
\end{proof}
 {\color{red}
\begin{lemma} \label[lemma]{lem:truncation}
  Let $\Fcal_n$ be classes of real-valued functions from $\Xcal$ to $\R$, $F_n$ be any function with $\sup_{f \in \Fcal_n}|f| \le F_n$, and $B_n$ be any sequence with $ P\ind_{F_n > B_n} = o(1/n)$. It holds that $$\sup\nolimits_{f \in \Fcal_n}|(\P_n - P)f| \le  \sup\nolimits_{f \in \Fcal_n}|(\P_n - P)f \ind_{F_n \le B_n}| + o_p\left( \sqrt{n^{-1} \sup\nolimits_{f \in \Fcal_n} P f^2}\right).$$
\end{lemma}
\begin{proof}
  It holds that
    \begin{align*}
      \sup_{f \in \Fcal_n}|(\P_n - P)f| 
      &\le \sup_{f \in \Fcal_n}|(\P_n - P)f \ind_{F_n \le B_n}| + \sup_{f \in \Fcal_n}|(\P_n - P) f\ind_{F_n > B_n}|  \le \sup_{f \in \Fcal_n}|(\P_n - P)f \ind_{F_n \le B_n}|  +|  \P_n F_n\ind_{F_n > B_n}|  +  \sup_{f \in \Fcal_n}|P f\ind_{F_n > B_n}|.
  \end{align*}
  For the last term, the Cauchy Schwarz inequality gives
    \begin{align*}
    \sup\nolimits_{f \in \Fcal_n}|P f\ind_{F_n > B_n}| \le \sqrt{\sup\nolimits_{ f \in \Fcal_n} P f^2} \sqrt{P \ind_{F_n > B_n}} = o\lf( \sqrt{n^{-1} \sup\nolimits_{f \in \Fcal_n} P f^2}\ri).
  \end{align*}
  Since $ P\ind_{F_n > B_n} = o(1/n)$, it holds that $\Pr(\exists i\colon F_n(\bX_i) > B_n) \le n P\ind_{F_n > B_n} = o(1),$ so $\P_n F_n\ind_{F_n > B_n} = 0$ with probability tending to 1, which implies the claim.
\end{proof}}

{\color{red}
\begin{lemma}\label[lemma]{lem:lemma10}
Under assumption \ref{A:Asymp}, it holds that $
 \sqrt{n} (\P_n - P) A_n [\phi(\hbtheta_U ) - \phi(\btheta^*)] =o_p(1).
$
\end{lemma}
\begin{proof}
Denote $\tilde r_n =  \alpha_n\sqrt{p_n \ln p_n / n}$.
From \cref{theorem_cons_new2}, it follows that $ \| \hbtheta_U - \btheta^* \| = O_p( r_n )$.
  We show that for each row $\ba_n$ from $A_n \in \R^{q \times p_n} $ and $C < \infty$, it holds that
  $
    \sup_{\| \bDelta \|\le \tilde r_n C} |(\P_n - P) \ba_n^\top [\phi(\btheta^* +  \bDelta) - \phi(\btheta^*)]| = o_p(1 / \sqrt{n}).
  $
  Since $A_n [\phi(\btheta^* +  \bDelta) - \phi(\btheta^*)]$ is a finite dimensional vector, this implies the claim.
  Let $\ba_n$ be some row of $A_n$.
  We have
  $$
    \sup_{\| \bDelta \|\le \tilde  r_n C} | (\P_n - P) \ba_n^\top[\phi(\btheta^* +  \bDelta) - \phi(\btheta^*)] |
    = \sup_{f_{\bDelta} \in \Fcal_n}  | (\P_n - P) f_{\bDelta}|
  $$
  with $\Fcal_n \coloneqq \{f_{\bDelta}(\bu) = \ba_n^\top [\phi(\bu; \btheta^* +  \bDelta) - \phi(\bu; \btheta^*)] : \| \bDelta \| \le \tilde r_n C \}$.
  Define $F_n(\bu) = \sup_{\| \bDelta \|, \| \bDelta' \| \le \tilde r_n C} | f_{\bDelta}(\bu) - f_{\bDelta'}(\bu) | / \| \bDelta - \bDelta' \|$, which is an envelope for the functions in $(\tilde r_n C)^{-1} \Fcal_n$, i.e., $\sup_{f_{\bDelta} \in \Fcal_n}  (\tilde r_n C)^{-1} | f_{\bDelta} (\bu) | \le F_n(\bu)$, since $f_{\bnull} = 0$.
  The second conditon in \ref{A:Asymp} implies that we can apply \cref{lem:truncation}, which gives
  \[
 (\tilde r_n C)^{-1}  \sup_{f_{\bDelta} \in \Fcal_n}  | (\P_n - P) f_{\bDelta}|
 \le (\tilde r_n C)^{-1}  \sup_{f_{\bDelta} \in \Fcal_n}  | (\P_n - P) f_{\bDelta} \ind_{F_n \le D_n}| + o_p(M_n/\sqrt{n})
  \]
  with $D_n$ as defined in \ref{A:Asymp} and $M^2_n \coloneqq  \sup_{\| \bDelta \|, \| \bDelta' \| } \E[| f_{\bDelta}(\bU) - f_{\bDelta'}(\bU)|^2]/  \| \bDelta - \bDelta' \|^2$, so $ \sup_{f_{\bDelta} \in \Fcal_n}  | (\P_n - P) f_{\bDelta}|  \le  \sup_{f_{\bDelta} \in \Fcal_n}  | (\P_n - P) f_{\bDelta} \ind_{F_n \le D_n}| + o_p(M_n \tilde r_n  / \sqrt{n})$.
   
  To bound the first term, we need covering numbers of $\Fcal_n$. It holds that $\ln N(\eps, \Fcal_n, L_2(P)) \le p_n \ln (3 M_n \tilde r_n C / \eps)$ and $N(\eps, \Fcal_n, \| \cdot \|_{\infty}) \le  p_n \ln (3 D_n \tilde r_n C / \eps)$: 
  Let $\bDelta_1, \ldots, \bDelta_N$ be the centers of an $\eta$-covering of $\{ \bDelta \in \R^{p_n}: \| \bDelta \| \le \tilde r_n C \}$, which we can find with $N = (3 \tilde r_n C / \eta)^{p_n}$. 
  Then, by the definitions of $M_n$ and $D_n$, the functions $f_{\bDelta_1}, \ldots, f_{\bDelta_N}$ are the centers of an $M_n \eta$-covering of $\Fcal_n$ in $L_2(P)$ and of a $D_n \eta$-covering in $\| \cdot \|_{\infty}$, respectively. 
  Choosing $\eta = \eps/M_n$ and $\eta = \eps/D_n$ gives the above bounds.
  
  Now, similar to the proof of \cref{lem:emp_proc}, \citep[][Theorem 2.14.21]{vdV} yields
\[
\E \lf[ \sup_{f_{\bDelta} \in \Fcal_n}  | (\P_n - P) f_{\bDelta} \ind_{F_n \le D_n}| \ri] = O\left( \frac{M_n \tilde r_n C \sqrt{p_n }}{\sqrt{n}} + \frac{D_n \tilde r_n C p_n }{n}\right) = o(1/\sqrt{n}),
\]
  since $M_n = o(1/( \tilde r_n \sqrt{ p_n}))$ and $D_n = o(\sqrt{n}/(\tilde r_n \, p_n))$ by \ref{A:Asymp}.
  This implies the claim.
\end{proof}}
 
\newpage

\begin{center}
    {\Large Supplement to ``Properties of stepwise parameter estimation in high-dimensional vine copulas''} 
\end{center}

\section{Numerical Validation of Assumptions}
\label{sec:ValAssump}

In this section, we numerically investigate some of the assumptions underlying the theorems presented above for Gaussian vine copulas.
While Gaussian vines are of limited use in practice, their simplicity allows computations in high dimensions, as many required quantities can easily be computed analytically.
Despite this limitation, the findings provide valuable insights into the conditions under which the theoretical results from the paper are valid.

\subsection{Validation of Assumption (A3)}
\label{sec:ValAssumpA3}

The key assumption to ensure identifiability is assumption (A3).
A central question is how to choose the sequences $\alpha_{j,n}$ such that the expression in (A3) remains negative and bounded away from 0.
If $\max_{1 \le j \le p_n} \alpha_{j,n} = O(1)$, the optimal rate of convergence in $\| \cdot \|_{\infty}$ norm, $\sqrt{\ln p_n / n}$, is attained.
As a starting point, we set all $\alpha_{j,n} = 1$. 
If the numerical approximation produces positive values, we try increasing functions for $\alpha_{j,n}$.

\textcolor{red}{For a given $\bDelta$, it is possible to calculate $\E[\phi(\bU; \btheta^* + \bDelta)]$ using recursively computed expectations $\E_{\btheta^*}[X_{a|D}(\btheta^* + \bDelta) \, X_{b|D}(\btheta^* + \bDelta)]$, where $x_{a|D} = \phi^{-1}(u_{a|D})$.
To approximate the supremum in (A3), we simulate several vectors $\bDelta_k$ and compute
\begin{equation}
\label{eq:A3Emp}
\max_{1 \le j \le p}\ \max_{k =1, \ldots K} \frac{\sign(\Delta_j)}{| \Delta_j |} \E[ \phi(\bU; \btheta^* + \bDelta_k)_j ]
\end{equation}
for Gaussian D- and C-vines with various models for the true parameter.}

For simplicity, all true parameters (the partial correlations in a Gaussian vine) in a given tree are set to the same value.
Let $\theta_t^*$ denote the true parameter in the $t$-th tree.
As discussed in the article, estimation is facilitated, or perhaps only possible,  if the pair copulas in higher trees converge to independence copulas.
Accordingly, we select functions for $\theta_t^*$ that converge to 0 at different rates.

The dimension $d$ ranges from $5$ to $50$. 
To approximate the supremum in (A3), we draw \textcolor{red}{$K=1000$} vectors $\bDelta_k$ with entries $\Delta_{k,j} = \pm \eps \alpha_ j$, where $\alpha_j = \alpha(t(j))$, and $t(j)$ denotes the tree of the $j$-th parameter.
The sign of $\Delta_{k,j}$ is drawn uniformly from $\{ -1, 1\}$.
The constant $\eps$ is $0.005$ and $\alpha(t) = 1$ for all models except for the D-vine with $\theta_t^* = 0.5/\sqrt{t+1}$.

The results are shown in \cref{fig:estimA3}.
For $\theta_t^* = 0$, $\theta_t^* = 0.5^t$ and $\theta_t^* = 1/(t+1)$, and the C-vine with $\theta_t^* = 0.5/(t+1)$, we obtain negative values when all $\alpha_{j,n}$ are set to 1.
Note that for some settings, the obtained values do not depend on $d$, as the maximum in \eqref{eq:A3Emp} is attained for parameters in an early tree.
In contrast, for the D-vine with $\theta_t^* = 0.5/\sqrt{t+1}$, the same setup produces positive values that are too large to show here.
While the increasing but bounded sequence $\alpha (t) =  \sum_{s=1}^t s^{-1.1}$ also yields positive values (not shown here), negative estimates are obtained using $\alpha(t) = t$ with $\eps = 10^{-7}$.
Using a diverging sequence for $\alpha_{j,n}$ results in a slower rate of convergence in Theorem 1, here $\| \hbtheta - \btheta^* \|_{\infty} = O_p(\sqrt{d_n^2 \ln d_n / n})$.

Although these results are obtained in finite dimensions and for a finite number of sampled $\btheta_k$ values and should therefore be interpreted with caution, they provide a useful impression of the relationship between (A3) and the rate of convergence in Theorem 1.

\begin{figure}
  \centering
  \includegraphics[width = 0.85\textwidth]{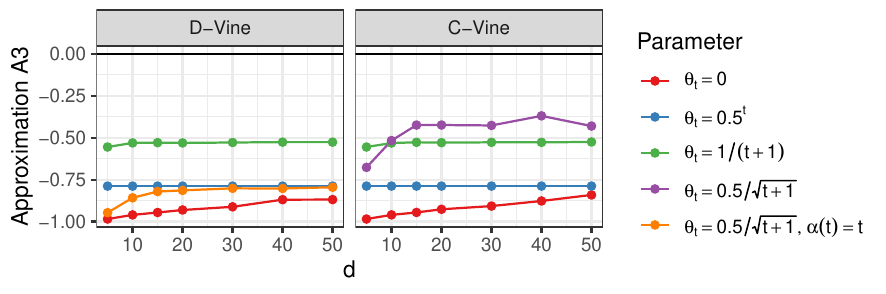}
  \caption{Approximated values for the validation of (A3) for Gaussian vines. The assumption is satisfied if the values are negative and bounded away from 0.
  The supremum over $\btheta$ is approximated by taking the maximum over  $K=1000$ values of $\btheta$, the expectation $\E[\phi(\bU; \btheta)]$ is computed analytically.  
$\alpha_{j,n} = 1$ unless otherwise stated.
  }\label{fig:estimA3}
  \vspace{0.5cm}
\includegraphics[width = 0.85\textwidth]{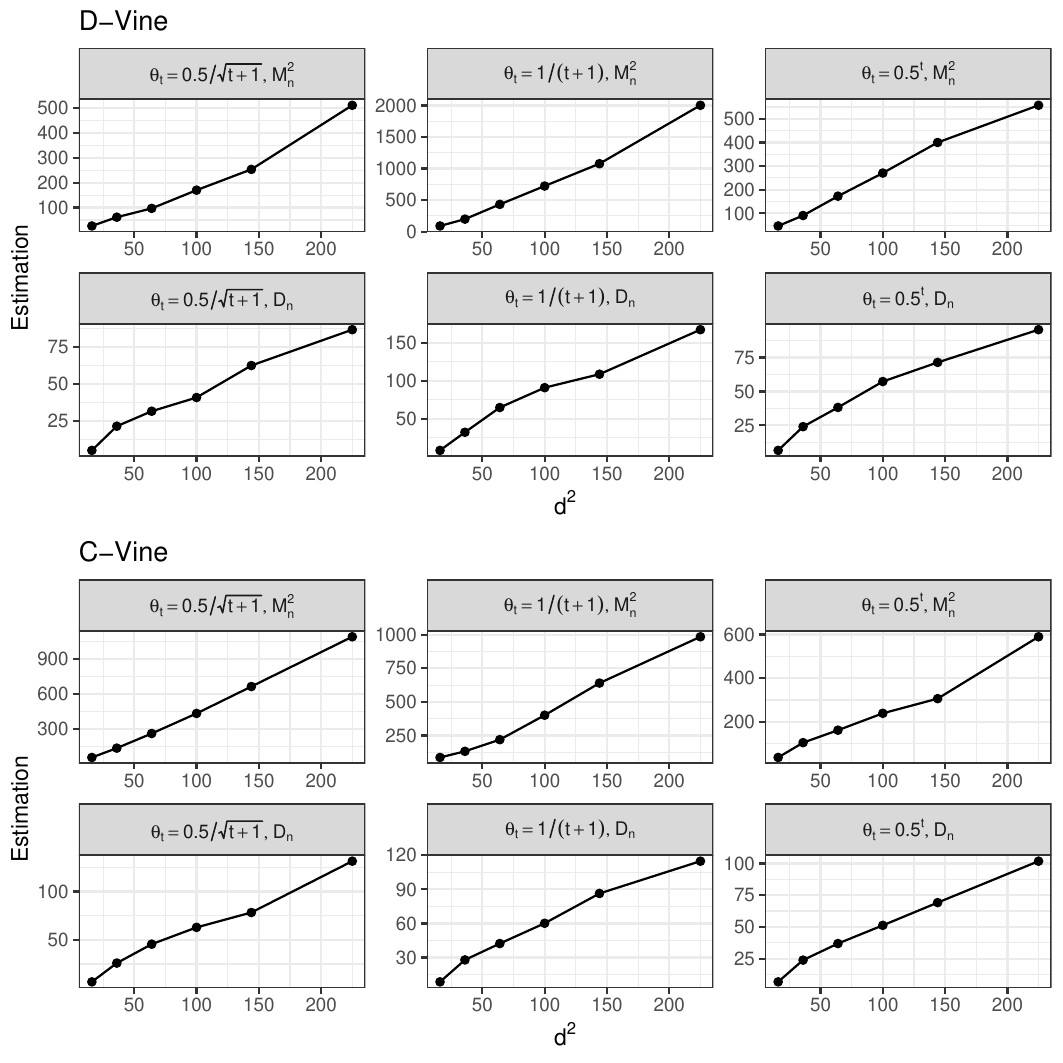}
  \caption{Estimated values for $M_n^2$ and $D_n$ in (A4).
  The supremum over $\btheta$ is approximated by taking the maximum over  $K=30$ values of $\btheta$.  
  }\label{fig:estimM}
\end{figure}

\subsection{Estimation of $M_n$ and $D_n$ from Assumption (A4)}

By estimating the quantities $M_n$ and $D_n$, we can obtain a rough assessment of the restrictions on $p_n$ imposed by assumption (A4).
For Gaussian vines, it is relatively easy to analytically compute derivatives of $\phi(\bu; \btheta)$ w.r.t.~$\btheta$.

The supremum is again approximated by drawing $K=30$ vectors $\bDelta_k$ as described in the previous section.
The constant $\eps$ is again set to $0.005$ and $\alpha(t) = 1$, with the exceptions for the D-vine with $\theta^*_t = 0.5/\sqrt{t+1}$, where $\eps = 10^{-7}$ and $\alpha(t)  = t$.
Due to the increased computational complexity compared to the previous section, the maximal dimension is limited to $d = 15$. 
To approximate the expectation in the definition of $M_n^2$, we draw $200 \log(d)$ samples.
For the estimation of $D_n$, the $1/n$ term in A4 must be replaced by some value depending on the number of parameters $p$.
Based on the results for $M_n$, we compute the $(1-15 / p^2)$-quantile since $p_n^2 \sim n$, using the same $200 \log(d)$ samples.
The results, shown in \cref{fig:estimM}, suggest that $M_n = O(d_n)$ and $D_n = O(d^2_n)$ for all considered models, indicating that the condition $p_n^2/n \to 0$ should be sufficient for (A4) to hold.

\section{$J(\theta)$ for Gaussian C-vines}

The following proposition provides a condition necessary to ensure that all eigenvalues of $J(\btheta^*) + J(\btheta^*)^\top$ are negative for a Gaussian C-vine.
For Gaussian vines, the parameters coincide with the partial correlations $\rho_{a,b;D}$ \citep{Kurowicka2003}, so we will use partial correlations instead of $\theta_{a,b;D}$ for convenience.

Both our theoretical results and the simulation study indicate that negative definiteness of $J(\btheta^*) + J(\btheta^*)^\top$ is not required for consistent estimation of $\btheta^*$.
This proposition thus shows that the results in \citet{Gauss24} are not suitable for vine copulas, as they rely on the eigenvalues of $J(\btheta^*) + J(\btheta^*)^\top$ being negative and bounded away from 0.

\begin{proposition}
\label{prop:CVine}
Consider a $d_n$-dimensional Gaussian C-vine with root nodes $1$ and $(1,2)$ in the first and second tree.
Assume that the true parameters in the second tree satisfy $|\rho^*_{23;1}| = |\rho^*_{24;1}| = \ldots = |\rho^*_{2d;1}| \neq 0$.
For the condition
\[
\limsup_{d \to \infty} \lambda_{\max} \lf(J(\brho^*) + J(\brho^*)^\top \ri) \le -c < 0
\]
to hold, the true parameters $\rho^*_{12}$ or $\rho^*_{2i;1}$ must decay with $d_n$.
Specifically, it is necessary that
\[
 \frac{\rho^*_{12} \rho^*_{2i;1}}{(1 - \rho^{*2}_{2i;1})(1 - \rho^{*2}_{12})} = O\lf(\frac{1}{ \sqrt{d_n} }\ri).
\]
\end{proposition}
For instance, this condition is satisfied if $\rho^*_{12} = O(d_n^{-c})$ and $\rho^*_{2i;1} = O(d_n^{-c})$ with $c \ge 0.25$, or if $\rho^*_{12} = O(1)$ and $\rho^*_{2i;1} = O(d_n^{-c})$ with $c \ge 0.5$.

\textcolor{red}{\section{Assumption (A3) from a Theoretical Perspective}} \label{sec:A3theory}

The key assumption for consistency of the estimator $\hbtheta_U$ is (A3), which states that there exists a $c > 0$ such that 
\begin{equation}\label{eq:A3}
\limsup_{n \to \infty}   \max_{1 \le j \le p_n} \sup_{ \bDelta \in \Theta_n^\Delta , | \Delta_j | = \alpha_{j,n}} (r_n C_n \alpha_{j,n})^{-1} \sign(\Delta_j) \, \E \lf[\phi(\bU; \btheta^* + r_n C_n \bDelta)_j \ri]  \le -c  ,
\end{equation}
where $\alpha_{j, n}, j = 1, \ldots, p_n$ are positive sequences that are bounded away from zero and $\Theta_n^\Delta \coloneqq \{ \bDelta: |\Delta_j | \le \alpha_{j,n}  \forall j = 1, \ldots, p_n\}$.
This section discusses this assumption from an analytical point of view.

Since $\E[\phi( \bU; \btheta^*)] = \bnull$, a Taylor expansion gives, for any $j = 1, \ldots, p_n$,
\[
\E \lf[\phi(\bU; \btheta^* + r_n C_n \bDelta)_j \ri]  
 = r_n C_n \bDelta^T \E\lf[\nabla_{\btheta} \phi(\bU; \bttheta)_j \ri]
\]
with some $\bttheta$ on the segment between $\btheta^*$ and $ \btheta^* + r_n C_n \bDelta$.
We can therefore write
\[
 (r_n C_n \alpha_{j,n})^{-1} \sign(\Delta_j) \, \E \lf[\phi(\bU; \btheta^* + r_n C_n \bDelta)_j \ri] 
 = \frac{\sign(\Delta_j)}{\alpha_{j,n}} \sum_{k=1}^{p_n} \Delta_k \E\left[ \frac{\partial}{\partial \theta_k} \phi(\bU; \bttheta)_j \right].
\]
Note that $\sign(\Delta_j) \Delta_j = | \Delta_j|$ and $ | \Delta_j| =\alpha_{j,n}$.

\vspace{5pt}

\myparagraph{Independent estimation problems}

If the estimating equation $\phi(\bU; \btheta)$ consisted of independent estimation problems, only the expectation $\E[ \frac{\partial}{\partial \theta_j} \phi(\bU; \bttheta)_j ]$ would be non-zero.
Then, if the entries of $\phi(\bu; \btheta)$ are the derivatives of functions to be maximized, \eqref{eq:A3} reduces to a standard second-order condition of non-zero, negative curvature around $\btheta^*$ to ensure identifiability.

\myparagraph{Vine Copulas}
In a vine where each pair copula has a single parameter, all partial derivatives with $k > j$ are 0 due to the sequential estimation.
The above sum simplifies to 
\begin{equation}
\label{eq:A3_sum}
\frac{\sign(\Delta_j)}{\alpha_{j,n}} \sum_{k=1}^{j} \Delta_k \E\left[ \frac{\partial}{\partial \theta_k} \phi(\bU; \bttheta)_j \right] = \E\left[ \frac{\partial}{\partial \theta_j} \phi(\bU; \bttheta)_j \right] + \frac{\sign(\Delta_j)}{\alpha_{j,n}} \sum_{k=1}^{j - 1} \Delta_k \E\left[ \frac{\partial}{\partial \theta_k} \phi(\bU; \bttheta)_j \right] . 
\end{equation}
The first term is the second derivative of the log-likelihood of the $j$-th pair copula w.r.t.~$\theta_j$.
In a small neighborhood around the true $\btheta^*$, it is reasonable to assume that this expected curvature is negative and bounded away from zero. 
For example, in a Gaussian vine evaluated at the true $\brho^*$, we have $ \E[ \frac{\partial}{\partial \rho_j} \phi(\bU; \brho^*)_j ] = -(1+(\rho_j^*)^2)/(1-(\rho_j^*)^2)^2$ (see the proof of \cref{lem:CVineGauss}),\footnote{Note that for computing this expectation, it is essential to evaluate at the true parameter, i.e., to compute the $u_{a|D}$ recursively using the true parameters. This ensures that $\E[X_{a|D}^2] = 1$ and $\E[X_{a|D} X_{b|D}] = \rho^*$, where $x_{a|D} = \Phi^{-1}(u_{a|D})$.} which converges to $-1$ for $\rho_j^* \to 0$.
Each of the remaining $j-1$ terms, however, could be positive and therefore requires a deeper analysis.
A bound for the worst-case scenario is
\begin{equation}
\label{eq:A3_bound}
\frac{\sign(\Delta_j)}{\alpha_{j,n}} \sum_{k=1}^{j - 1} \Delta_k \E\left[ \frac{\partial}{\partial \theta_k} \phi(\bU; \bttheta)_j \right] \le \sum_{k=1}^{j - 1} \frac{\alpha_{k,n}}{\alpha_{j,n}} \left| \E\left[ \frac{\partial}{\partial \theta_k} \phi(\bU; \bttheta)_j \right] \right|.
\end{equation}

\myparagraph{Finite-dimensional Vines}
In a finite-dimensional vine, as long as $\E[ \frac{\partial}{\partial \theta_j} \phi(\bU; \bttheta)_j ]$ is negative for each $j = 1, \ldots, p$, we can always choose an increasing sequence $\alpha_j, j = 1, \ldots p$ such that the fractions $\alpha_k/\alpha_j$ are small enough to ensure the first term in \eqref{eq:A3_sum} dominates the remaining $j-1$ terms.
Since there are only finitely many $\alpha_j$, we have $\alpha_j = O(1)$, leading to the best possible rate of convergence in Theorem 1.

\myparagraph{Truncated Vines}
In a truncated vine with $d_n \to \infty$ and a fixed truncation level $T$, the sum in \eqref{eq:A3_bound} contains at most $O(T^2)$ (i.e., finitely many) non-zero terms for each $j = 1, \ldots, p_n$.
Again, by choosing increasing sequences $\alpha_{j,n}$, we can ensure that the first term dominates in \eqref{eq:A3_sum}.
Since each $\alpha_{j,n}$ only has to ``mitigate'' the contribution of finitely many $ | \E[ \frac{\partial}{\partial \theta_k} \phi(\bU; \bttheta)_j ] |$, we again obtain $\alpha_{j,n} = O(1)$ ( as long as $\max_{j,k}  | \E[ \frac{\partial}{\partial \theta_k} \phi(\bU; \bttheta)_j ] | = O(1)$), see also Proposition 1 in the manuscript.

\myparagraph{General Vines}
While the above discussed special cases (independent estimation problems, finite-dimensional and truncated vines) provide sanity checks for the plausibility of (A3), they do not provide enough insight into the untruncated vines with $d_n \to \infty$.
We have already seen that the sequences $\alpha_{j,n}$ play a pivotal role in (A3):
Choosing sequences that increase in $j$, it is easier to satisfy (A3).
However, this might come at the cost of a slower rate of convergence in Theorem 1, if $\alpha_n = \max_{1 \le j \le p_n} \alpha_{j,n}$ diverges.
$\frac{\partial}{\partial \theta_k} \phi(\bu; \bttheta)_j $ captures the dependence of the estimation of $\theta_j$ on $\theta_k$. 
If this effect is too strong, estimating $\theta_j$ becomes more difficult and choosing $\alpha_{j,n}$ such that $\alpha_n \to \infty$ might be necessary to satisfy (A3). 
For which models is this the case and in which settings can we choose bounded sequences $\alpha_{j,n}$ or even $\alpha_{j,n} = 1$ for all $j$ and $n$?
To obtain more insights, we need to take a closer look into the derivatives $\frac{\partial}{\partial \theta_k} \phi(\bu; \bttheta)_j$.

The $j$-th entry of $\phi(\bu; \btheta)$ contains the derivative of the log-likelihood of a parameter $\theta_{a,b; D}$, which is estimated using conditional variables $u_{a|D}, u_{b|D}$, both depending on $\btheta$.
Denote the derivative of the log-likelihood by $s$, so
\[
\phi(\bu; \btheta)_j = s(u_{a|D}(\btheta), u_{b|D}(\btheta); \theta_{a,b;D}).
\]
Both $u_{a|D}(\btheta)$ and $u_{b|D}(\btheta)$ are calculated recursively, with $u_{a|D}(\btheta) = h(u_{a|D\setminus\{ c \}}(\btheta), u_{c|D\setminus\{ c \}}(\btheta); \theta_{ac;D\setminus\{ c \}})$ etc.~for some $c \in D$.
Now assume that $\theta_k$ appears exactly once in this recursion.
We obtain the derivative of $\phi(\bu; \btheta)_j$ w.r.t.~$\theta_k$ using the chain rule and denoting
\[
h_i(u_1, u_2; \theta ) = \frac{\partial h(u_1, u_2; \theta )}{\partial u_i}, i = 1, 2, \quad h'(u_1, u_2; \theta ) = \frac{\partial h(u_1, u_2; \theta )}{\partial \theta},\quad s_i(u_1, u_2; \theta) =  \frac{\partial s(u_1, u_2; \theta )}{\partial u_i}, i = 1, 2,
\]
and write $h_{1|2}$ and $s_{1|2}$ for either $h_1$ or $h_2$ and $s_1, s_2$ respectively.
The derivative of $\phi(\bu; \btheta)_j$ w.r.t.~$\theta_k$ is then given by
\begin{equation}
\label{eq:deriv_product}
\frac{\partial \phi(\bu; \bttheta)_j}{\partial \theta_k} = s_{1|2}(u_{a|D}(\bttheta), u_{b|D}(\bttheta); \tilde \theta_{a,b;D}) \cdot h_{1|2}(\ldots) \cdots h_{1|2}(\ldots) \cdot h'(\ldots; \tilde \theta_k),
\end{equation}
For each of the $h_{1|2}(\ldots)$, the vine structure encodes whether to use $h_1$ or $h_2$ as well as their arguments. 
This information is typically encoded in a triangular matrix \citep[see for example][]{CzadoNagler}.
If $\theta_k$ appears several times in the recursions in $u_{a|D}(\btheta)$ or $u_{b|D}(\btheta)$, the multivariable chain rule implies that $\frac{\partial}{\partial \theta_k} \phi(\bu; \bttheta)_j$ is the sum of the derivatives w.r.t.~the ``appearances'' of $\theta_k$.
How do these derivatives consisting of (sums of) products behave? 
We turn to Gaussian vines to obtain more concrete examples.

\myparagraph{Gaussian vines} 
It is more convenient to parameterize a Gaussian vine using the transformed variables $x_{a|D} = \Phi^{-1}(u_{a|D})$.
This yields
\[
h_1(x_1, x_2; \rho) = \frac{1}{\sqrt{1-\rho^2}}, \quad h_2(x_1, x_2; \rho) = - \frac{\rho}{\sqrt{1-\rho^2}}, \quad
h'(x_1, x_2; \rho) = \frac{\rho x_1 - x_2}{(1-\rho^2)^{3/2}}, \quad s_1(x_1, x_2;\rho) = \frac{(1 + \rho^2) x_2 - 2 \rho x_1}{(1- \rho^2)^2},
\]
and $s_2(x_1, x_2; \rho) = s_1(x_2, x_1; \rho)$.
$h_1$ and $h_2$ do not depend on $x_1, x_2$, which makes the computation of expectations much easier than in the general case. 
Since $h_2(0) = 0$, we have $\frac{\partial}{\partial \rho_k} \phi(\bx; \bnull)_j = 0$ whenever $h_2$ appears at least once in the product \eqref{eq:deriv_product}.
In contrast, if many partial correlations are far away from 0, derivatives where $h_1$ occurs multiple times may diverge for $d \to \infty$.

If we evaluate $\frac{\partial}{\partial \rho_k} \phi(\bx; \brho)_j$ at the true parameter $\brho^*$, we can obtain analytical results for certain vine structures.
Consider a $d$-dimensional D-vine with parameters $\rho_{12,}, \rho_{23}, \ldots, \rho_{d-1 ,d}$ in the first tree and consider the derivatives of the last estimating equation, corresponding to the parameter $\rho_{1, d; 2, 3, \ldots, d-1}$.
Only derivatives w.r.t.~parameters corresponding to leaf nodes of the trees are non-zero (i.e., $\rho_{12}, \rho_{13;2}, \rho_{14;23}, \ldots, \rho_{1 (d-1); 2, \ldots, d-2}$ and $\rho_{d-1, d}, \rho_{d-2, d;d - 1}, \ldots$), because otherwise, all required expectations of products $\E[X_{1|2, \ldots, d - 1} X_{a|D}], \E[X_{d|2, \ldots, d - 1} X_{a|D}]$ are 0 since $a \in \{ 2, \ldots, d - 1\}$.
For example, the expectation of the derivative w.r.t.~$\rho_{24;3}$ involves the expectations $\E[X_{1|2, \ldots, d - 1} X_{2|3}], \allowbreak\E[X_{1|2, \ldots, d - 1} X_{4|3}], \allowbreak\E[X_{d|2, \ldots, d - 1} X_{2|3}],\allowbreak \E[X_{d|2, \ldots, d_n - 1} X_{4|3}]$, which are all 0 when the conditional $x_{a|D}$ are computed at the true parameter.
We set all $\alpha_{j,n}$ to 1 and assume that all parameters in the $t$-th tree are the same, denoted by $\rho_t$.
Plugging in the expected derivatives of the last estimating equation from \cref{lem:derivD-vine} in the bound in \eqref{eq:A3_bound} yields
\[
\sum_{t=1}^{d_n - 2}  \left| \E\left[ \frac{\partial}{\partial \rho_t} \phi(\bX; \brho^*)_{p_n} \right] \right| 
= 2 \sum_{t=1}^{d_n - 2}   \left|  \frac{\rho_t \cdot\rho_{d_n - 1} }{(1-\rho_t^2)(1-\rho_{d_n - 1}^2)} \right| 
= \left|  \frac{\rho_{d_n - 1} }{1-\rho_{d_n - 1}^2} \right|  \cdot 2 \sum_{t=1}^{d_n - 2}   \left|  \frac{\rho_t  }{1-\rho_t^2} \right|.
\]
Note that we omitted the $*$ marking the true parameter to improve readability.
It is easy to see that this expression diverges if all parameters are the same.
If $\rho_t$ decreases fast enough such that the above expression converges, we have to compare its limit to the expected derivative of the last estimating equation w.r.t.~the last parameter, $-(1+\rho_{d_n - 1}^2)/(1-\rho_{d_n - 1}^2)^2$ (see the proof of \cref{lem:CVineGauss}), which converges to $-1$ for $\rho_t \to 0$.

For $\rho_t = 1/\sqrt{t+1}$, we have
\begin{align*}
 \left|  \frac{\rho_{d_n - 1} }{1-\rho_{d_n - 1}^2} \right|  \cdot 2 \sum_{t=1}^{d_n - 2}   \left|  \frac{\rho_t  }{1-\rho_t^2} \right|
 & = \frac{1}{\sqrt{d_n}} \frac{d_n}{d_n - 1} \cdot 2  \sum_{t=2}^{d_n - 1}   \frac{1}{\sqrt{t}} \frac{t}{t -1} > \frac{2}{\sqrt{d_n}}  \cdot   \sum_{t=2}^{d_n - 1}   \frac{1}{\sqrt{t}}.
\end{align*}
Since $\sum_{t=2}^{d_n - 1}  t^{-1/2} \approx 2 \sqrt{d_n}$ (this follows from the Euler-Maclaurin-formula, see for example \citet[][Chapter 9.5]{ConcrMaths}), the obtained bound for \eqref{eq:A3_sum} (for $\bttheta = \btheta^*$) is positive, indicating that in this setting, (A3) might be violated for $\alpha_{j,n} = 1$, which is in line with the numerical validation in \cref{sec:ValAssumpA3}.
For $\rho_t = 1/(t+1)$, we obtain in a similar manner
\[
 \left|  \frac{\rho_{d_n - 1} }{1-\rho_{d_n - 1}^2} \right|  \cdot 2 \sum_{t=1}^{d_n - 2}   \left|  \frac{\rho_t  }{1-\rho_t^2} \right|
 \approx \frac{2}{d_n} \sum_{t=2}^{d_n - 1} \frac{1}{t}.
\]
Since $\sum_{t=2}^{d_n - 1}  t^{-1} \approx \ln d_n$ \citep[see for example][Theorem 0.8]{TenenB}, this expression converges to 0, which implies that \eqref{eq:A3_sum} is negative for $\bttheta = \btheta^*$.

These analytical derivations only hold when the expectations are evaluated at the true parameter, as this ensures that the recursively computed $x_{a|D}(\brho)$ behave the same as the true conditional $X_{a|D}$.
Obtaining brief analytical expressions for $\E\left[ \frac{\partial}{\partial \rho_k} \phi(\bX; \tilde \brho)_j \right]$ for a $\tilde \brho \neq \brho^*$ is impossible.
In this case, rules like $\E[X_{a|b}(\tilde \rho_{ab}) \cdot X_b] = 0$ do not hold, leading to more non-zero terms in \eqref{eq:A3_sum}.
Deriving negative upper bounds for \eqref{eq:A3_sum} with $\bttheta = \btheta^*$ is therefore not sufficient to ensure the validity of (A3), so these theoretical results are complemented by the numerical validation in \cref{sec:ValAssumpA3}.

\myparagraph{Derivatives w.r.t.~$\bu$}

The additional assumptions to ensure convergence and asymptotic normality of the estimator $\hbtheta_X$ with nonparametric estimation of margins mostly concern partial derivatives of $\phi(\bu; \btheta^*)$ w.r.t.~$\bu$.
These derivatives are sums of products similar to those in \eqref{eq:deriv_product}, with $h'$ replaced by the derivative of $h(u_1, u_2;\theta)$ w.r.t.~$u_1$ or $u_2$.
Again, independence pair copulas imply $h_2 = 0$, setting certain derivatives to 0.
Pair copulas that are close to independence are therefore more ``well-behaved'' than those with strong dependence, leading to milder assumptions on $d_n$.

\textcolor{red}{\section{Example: Assumptions for FGM vines}}

We now consider a $d_n$-dimensional vine with arbitrary R-vine structure, where each pair copula is a Farlie-Gumbel-Morgenstern (FGM) copula \citep[see for example][]{Joe2014} with cdf and density given by
\[
  C_\theta(u,v)=uv+\theta uv(1-u)(1-v), 
  \qquad c_\theta(u,v)=1+\theta a(u)a(v), \quad \text{where } \ a(u) = 1 - 2u
\]
for $\theta \in [0,1]$.
The score (derivative of the log density) and the $h$-function are 
\[
s_\theta(u,v)=\partial_\theta \log c(u,v;\theta)
  =\frac{a(u)a(v)}{1+\theta a(u)a(v)}, \qquad h_\theta(u,v)
  =u+\theta b(u)a(v), \quad \text{where } \ b(u) = u(1-u).
\]
A nice property of the FGM copula is that both score and $h$-function as well as all relevant first and second derivatives are bounded as long as $\theta$ is bounded away from $\pm 1$.
This leads to the following proposition:

\begin{proposition}\label{prop:FGM1}
    Consider a $d_n$-dimensional FGM vine, where all parameters in a given tree level $t$ share the same true value, denoted by $\theta_t^*$, and assume that
\[
\eta_{t,n} \coloneqq \sup_{\btheta \in \Theta_n} \max_{e \in E_t} |\theta_e | \le \rho < 1\qquad \text{for some fixed $\rho$}.
\]
For example, one may take $\eta_{t,n} = | \theta_t^* | + r_n C_n$ with $C_n \to \infty$ from the definition of $\Theta_n$ and $r_n=\sqrt{\log d_n/n}$.
Set
\[
  \delta_\rho \coloneqq 1-\rho,
  \quad
  \kappa_\rho \coloneqq (9(1+\rho))^{-1},
  \quad P_{t,n} \coloneqq \sum_{l=1}^t \prod_{m = l + 1}^t \left( 1 + \frac 3 2 \eta_{m, n} \right), \quad P_{t,n}^* \coloneqq \sum_{l=1}^t \prod_{m = l + 1}^t \left( 1 + \frac 3 2 | \theta_m^* | \right).
\]
and $\alpha_{j,n} =1$ for all $j, n$.
Let $A_n \in \R^{q \times p_n}$ with fixed $q$ and denote $a_n \coloneqq \| A_n \|_\infty$.

Then, (A1)--(A6) hold if
\begin{align*}
  & \max_{1 \le t \le d_n - 1} | \theta_t^*| \, P_{t-1,n}^* \le c \frac{1 - \rho}{9(1+\rho)} \ \text{ for some fixed $c < 1$}, \qquad
r_n C_n P_{d_n,n}^2 \to 0, \qquad  P_{d_n,n} p_n / n \to 0, \\
& a_n^2  P_{d_n,n}^2 p_n^2 \ln p_n / n \to 0, \qquad (a_n C_n \ln p_n  P_{d_n, n}^2 )^2 p_n / n \to 0, \qquad a_n^4/n \to 0.
\end{align*}
\end{proposition}

\begin{proposition}\label{prop:FGM2}
  For a $d_n$-dimensional FGM vine, where all parameters in a given tree level $t$ share the same true value, denoted by $\theta_t^*$, with
  \[
  |\theta_t^*| \le  \frac{1}{(t+1)^k}  \qquad \text{for some fixed $k > 1$,}
  \]
  the assumptions of Theorem 1 are satisfied if $d_n^4 \ln d_n / n \to 0$. 
  The assumptions of Theorem 2 are satisfied for matrices $A_n \in \R^{q \times p_n}$ with $\| A_n \|_\infty = O(1)$ if $d_n^6 \, (\ln d_n)^2 / n \to 0$.
\end{proposition}
 
\section{Proofs of Propositions}

\subsection{Proof of \autoref{prop:CVine}}

We parameterize the Gaussian copula density using the transformed variables $x_{a|D} = \Phi^{-1}(u_{a|D})$.
Let $1$ and $(1,2)$ be the root nodes in the first and second tree, i.e., the first tree contains the copulas $c_{1i}, i = 2, \ldots, d_n$, and the second tree contains the copulas $c_{2i; 1}, i = 3, \ldots, d_n$.
To simplify notation, we first assume that $\rho_{12}^* >0$ and that all true parameters $\rho^*_{2i; 1}, i = 3, \ldots, d_n$ in the second tree are equal and positive, and then extend the proof to $|\rho^*_{23;1}| = |\rho^*_{24;1}| = \ldots = |\rho^*_{2d_n;1}| \neq 0$.

Since
\[
\frac 1 2 \lambda_{\max} \lf(J(\brho^*) + J(\brho^*)^\top \ri) = \sup_{\| \bu \| = 1} \bu^\top J(\brho^*) \bu,
\]
we derive conditions necessary for $\sup_{\| \bu \| = 1} \bu^\top J(\brho^*) \bu$ not to diverge to $+\infty$.
Let $\bu$ be such that the entry corresponding to the estimating equation of $\rho_{12}$ is $u_1 > 0$, and the $d_n-2$ entries corresponding to the equations of $\rho_{2i;1}$ are all $u_2 < 0$.
All other entries of $\bu$ are set to zero.
Denote
\[
s(x_1, x_2; \rho) \coloneqq\frac{ \partial \log c(x_1, x_2;\rho)}{\partial \rho}
\]
and define
\[
a_1  \coloneqq\E\lf[\frac{\partial s_{12}(X_1, X_2; \rho_{12})}{\partial \rho_{12}} \ri], 
\; a_2  \coloneqq  \E\lf[\frac{\partial s_{2i;1}(X_{2|1}, X_{i|1}; \rho_{2i;1})}{\partial \rho_{2i;1}} \ri], 
\; b  \coloneqq  \E\lf[\frac{\partial s_{2i;1}(X_{2|1}, X_{i|1}; \rho_{2i;1})}{\partial \rho_{12}} \ri],
\]
where all expectations are computed plugging in the true $\rho^*$.
Since all $ \rho_{2i; 1}$ are the same, $a_2$ and $b$ do not depend on $i$.
As $\partial s_{2i;1}(X_{2|1}, X_{i|1}; \rho_{2i;1})/\partial \rho_{2j;1} = 0$ for $i \neq j$, we have
\begin{align*}
\bu^\top J(\brho^*) \bu & = u_1^2 a_1 + (d_n-2) u_2^2 a_2+ (d_n-2) u_1 u_2 b \\
 &= u_1^2 a_1 + (1 - u_1^2) a_2 - \sqrt{d_n-2} u_1 \sqrt{1 - u_1^2} b,
\end{align*}
where the second equality uses that $\| \bu \| = 1$ implies $u_2 = - \sqrt{(1-u_1^2)/(d_n-2)}$.
The first two terms are negative (since $a_1, a_2 < 0$, see \cref{lem:CVineGauss}) and do not depend on $d_n$ if $u_1$ is fixed.
Since
\[
b =  - \frac{\rho^*_{12} \rho^*_{2i;1}}{(1 - \rho^{*2}_{2i;1})(1 - \rho^{*2}_{12})} < 0
\]
by \cref{lem:CVineGauss} and $\rho^*_{12}, \rho^*_{2i;1} > 0$, the term $- \sqrt{d_n - 2}b$ diverges to $+ \infty$ for $d_n\to \infty$ unless $b = O(1/\sqrt{d_n})$.
The proof can be easily extended to the case where no assumptions are made on the sign of $\rho_{12}^*$ and where $|\rho^*_{23;1}| = |\rho^*_{24;1}| = \ldots = |\rho^*_{2d_n;1}| \neq 0$.

\textcolor{red}{
\subsection{Proof of \autoref{prop:FGM1}}}

    (A1) and (A2) hold since the FGM score $s_\theta (u,v)$ is bounded for $| \theta | \le \rho < 0$.
    For (A3), we have, since $\E [\phi(\bU; \btheta^*)_j] = 0$ and $|\Delta_k| \le \alpha_{k,n} = 1$ for all $k$, for all $j = 1, \ldots, p_n$ from any tree $t = 1, \ldots, d_n - 1$,
    \begin{align*}
      (r_n C_n)^{-1} \sign(\Delta_j) \E \left[\phi(\bU; \btheta^* + r_n C_n \bDelta)_j \right]   
    & \le  \E \left[  \frac{\partial }{\partial \theta_j } \phi(\bU; \btheta^* )_j \right] +  \sum_{k=1}^{j-1} \left| \E \left[  \frac{\partial }{\partial \theta_k } \phi(\bU; \btheta^* )_j\right] \right|  \\
     & \quad + r_n C_n \frac 1 2 \sum_{k=1}^j \sum_{k' = 1}^j \left| \E \left[  \frac{\partial^2 }{\partial \theta_k \partial \theta_{k'}} \phi(\bU; \btheta^* )_j\right] \right| \\
     & \le - \kappa_{\rho} + \delta_\rho^{-1} | \theta_t^*| \, P_{t-1, n}^* + 6 r_n C_n  \delta_\rho^{-3} (1 + P_{t-1,n})^2
    \end{align*}
    by \cref{lem:FGMbounds}, \cref{lem:FGM_J} and \cref{lem:FGM2deriv}.
    The assumption on $| \theta_t^*| \, P_{t-1, n}^*$ implies $\delta_\rho^{-1} | \theta_t^*| \, P_{t-1, n}^* \le c \kappa_{\rho}$ with some $c < 1$ for all $t$, so (A3) holds since
    \[
    - \kappa_{\rho} + \delta_\rho^{-1} | \theta_t^*| \, P_{t-1, n}^* + 6 r_n C_n  \delta_\rho^{-3} (1 + P_{t-1,n})^2
    \le (-1 + c)  \kappa_{\rho} + o(1)
    \]
    since $r_n C_n P_{d_n,n}^2 \to 0$.
    For (A4), \cref{lem:FGM1deriv} and \cref{lem:FGMbounds} ($| \partial_\theta s_\theta (u,v) | \le \delta_\rho^{-2}$) give the deterministic bound
    \[
    \max_{1 \le j \le p_n} \sup_{\btheta  \in \Theta_n} \sum_{k=1}^j \lf| \frac{\partial}{\partial \theta_k} \phi(\bu; \btheta)_j \ri| 
    \le  \delta_\rho^{-2} + \delta_\rho^{-2} \max_{1 \le t \le d_n - 1} P_{t,n},
    \]
    so we may choose $M_n = D_n = \delta_\rho^{-2} ( 1 + P_{d_n,n})$.
    With $k_n = p_n$ in (A4), this assumption is satisfied since $( 1 + P_{d_n,n})^2 p_n / n \to 0$.

    For (A5), note that
    \begin{align*}
        \|A_n[\phi(\bu;\btheta^* +  \bDelta) - \phi(\bu;\btheta^* +  \bDelta')] \|
     & \le \sqrt{q} \max_{1 \le i \le q} \left| \sum_{j=1}^{p_n} (A_n)_{i,j} [\phi(\bu;\btheta^* +  \bDelta)_j - \phi(\bu;\btheta^* +  \bDelta')_j]  \right| \\
     & \le \sqrt{q} \  \max_{1 \le i \le q} \sum_{j=1}^{p_n} | (A_n)_{i,j}|  \max_{1 \le j \le p_n} | \nabla_{\btheta} \phi(\bu;\tilde \btheta)_j (\bDelta - \bDelta') | \\
     & \le \sqrt{q}\ a_n \ \| \bDelta - \bDelta' \|\max_{1 \le j \le p_n} \| \nabla_{\btheta} \phi(\bu;\tilde \btheta)_j \|   \le \sqrt{q}\ a_n \ \| \bDelta - \bDelta' \| \ \delta_\rho^{-2} (1 + P_{d_n,n})
    \end{align*}
    by \cref{lem:FGM1deriv}.
    The left-hand side of the first condition in (A5) is therefore bounded by $q \, a_n^2  \, \delta_\rho^{-4} (1 + P_{d_n,n})^2$.
    With $\tilde r_n^2 = p_n \ln p_n / n$, the first condition is satisfied since $a_n^2 (1 + P_{d_n,n})^2 p_n^2 \ln p_n / n \to 0$.
   Similarly, we may choose $D_n = \sqrt{q} \, a_n \, \delta_\rho^{-2} (1 + P_{d_n,n})$ for the second condition in (A5).
   Since $\sqrt{n}/(\tilde r_n p_n) = n/(p_n^{3/2} \sqrt{\ln p_n})$ and $a_n \, (1 + P_{d_n,n}) p_n^{3/2} \sqrt{\ln p_n} / n \to 0$, the second condition is satisfied.
   For the third condition in (A5), denote $B_n = J(\btheta^* +  r_n C_n  \bDelta) - J(\btheta^*)$ and note that by \cref{lem:FGM2deriv}, the maximum absolute row sum of $B_n$ (i.e., $\| B_n \|_\infty$) is bounded by $ 12 r_n C_n \delta_\rho^{-3} (1 + P_{d_n,n})^2$. Now, standard inequalities for matrix norms give
   \[
   \| A_n B_n  \|_2 \le \sqrt{\| A_n B_n \|_1 \| A_n B_n \|_\infty } \le \sqrt{q \, \| A_n B_n \|_\infty^2 } = \sqrt{q} \| A_n B_n \|_\infty \le \sqrt{q} \| A_n \|_\infty \| B_n \|_\infty = 12 \sqrt{q} \, a_n r_n C_n \delta_\rho^{-3} (1 + P_{d_n,n})^2 .
   \]
   As $\sqrt{n} \tilde r_n = \sqrt{p_n \ln p_n}$ and $r_n = \sqrt{\ln p_n / n}$, the third condition in (A5) holds since $(a_n  C_n \ln p_n (1 + P_{d_n,n})^2)^2 p_n / n \to 0$.
   For (A6), each entry of $A_n \phi(\bu; \btheta)$ is $O(a_n)$ since the entries of $\phi(\bu; \btheta)$ are uniformly bounded, so $\| A_n \phi(\bu; \btheta) \| = O(\sqrt{q}\, a_n)$ and (A6) holds since $a_n^4/n \to \infty$.
   
\textcolor{red}{
\subsection{Proof of \autoref{prop:FGM2}}}

We may set $\eta_{t,n} = \theta_t^* + r_n C_n$ with $C_n \to \infty$ arbitrarily slowly and $r_n = \sqrt{\ln p_n / n}$.
We show that $P_{d_n,n} = O(d_n)$, as this implies that (A1)--(A4) hold if $d_n^4 \ln d_n / n \to 0$.
Using $(1+x) \le \exp(x)$ for $x > 0$, it holds that
\[
P_{t,n}  \le  \sum_{l=1}^t \prod_{m = l + 1}^t \exp\left(   \frac 3 2 \eta_{m, n} \right) = \sum_{l=1}^t \exp\left(\frac 3 2 \sum_{m = l + 1}^t  ( \theta_m^* + r_n C_n )\right) = \sum_{l=1}^t \exp\left(\frac 3 2 \sum_{m = l + 1}^t   \theta_m^* \right) \exp\left(\frac{3}{2} \, t \, r_n C_n \right) = O(t),
\]
since $d_n r_n C_n = d_n \sqrt{\ln d_n} C_n / n = O(1)$ and $\sum_{m = l + 1}^t   \theta_m^* = O(1)$.
To verify
\[
\max_{1 \le t \le d_n}  | \theta_t^*|\, P_{t-1,n}^* \le c \frac{1 - \rho}{9(1+\rho)} \ \text{ for some $c < 1$},
\]
note that the right-hand side is uniformly bounded and $ | \theta_t^*|\, P_{t-1,n}^* \to 0$ as $t \to \infty$.
It suffices to verify that this assumption holds for large enough $t > t_0$ for some fixed $t_0$, as one can always choose uniformly bounded sequences $\alpha_{j,n}$ (i.e., with no effect on the rate of convergence) such that (A3) holds for $t < t_0$, as long as $\E[\frac{\partial}{\partial \theta_j} \phi(\bU; \btheta^*)_j]$ is negative and bounded away from 0, which is the case here.

For Theorem 2, both remaining conditions in \cref{prop:FGM1} hold if $d_n^6 (\ln d_n)^2 /n \to 0$, since $P_{d_n, n} = O(d_n)$ and $a_n = O(1)$.
\vspace{0.5cm}

\section{Lemmas}

\subsection{Gaussian vines}

\begin{lemma}
\label{lem:CVineGauss}
For a Gaussian vine with $X_{a|D} = \phi^{-1}(U_{a|D}), X_{b|D} = \phi^{-1}(U_{b|D})$ and true parameter $\rho$, it holds that
\[
\E\lf[ \frac{\partial s_{ab;D}(X_{a|D}, X_{b|D}; \rho_{ab;D})}{\partial \rho_{ab;D}} \ri] = - \frac{1 + \rho_{ab;D}^2}{(1 - \rho_{ab;D}^2)^2}
\]
and 
\[
\E\lf[ \frac{\partial s_{2 i ; 1}(X_{2|1}, X_{i | 1}; \rho_{2i ; 1})}{\partial \rho_{12}} \ri] = - \frac{\rho_{12} \rho_{2i;1}}{(1 - \rho_{2i;1}^2)(1 - \rho^2_{12})}, \quad i = 3, \ldots, d.
\]
\end{lemma}

\begin{proof}
In a Gaussian bivariate copula, the log-likelihood of a parameter $\rho$ is given by
\[
\log c(x_1, x_2 ; \rho) = - \frac 1 2 \log(1 - \rho^2) - \frac{\rho^2 (x_1^2 + x_2^2) - 2 \rho x_1 x_2}{2 ( 1 - \rho^2)},
\]
see for example \citet{Czado2019} for the copula density.
The derivative w.r.t.~$\rho$ is
\[
s(x_1, x_2 ; \rho) \coloneqq\frac{\partial \log c(x_1, x_2 ; \rho) }{\partial \rho} 
= \frac{\rho}{1 - \rho^2} - \frac{\rho}{(1-\rho^2)^2} (x_1^2 + x_2^2) + \frac{1 + \rho^2}{(1 - \rho^2)^2} x_1 x_2.
\]
Note that these are the entries in $\phi(\bx; \brho)$ with $x_1, x_2$ replaced by conditional $x_{a|D}$, which are computed recursively using $\brho$ and the original $\bx$.
The derivative of $s(x_1, x_2 ; \rho)$ w.r.t.~$\rho$ is
\[
\frac{\partial s(x_1, x_2 ; \rho)}{ \partial \rho} = \frac{1 + \rho^2}{(1 - \rho^2)} - \frac{1 + 3 \rho^2}{(1 - \rho^2)^3}(x_1^2 + x_2^2)  + 2 \frac{3 \rho + \rho^3}{(1 - \rho^2)^3}  x_1 x_2
\]
and the derivative w.r.t.~$x_1$ is
\[
\frac{\partial s(x_1, x_2 ; \rho)}{ \partial x_1} = \frac{(1 + \rho^2) x_2 - 2 \rho x_1}{(1- \rho^2)^2}.
\]
The $h$-function is given by $h(x_1, x_2 ; \rho) = (x_1 - \rho x_2)/\sqrt{1 - \rho^2}$ \citep[see for example][]{Schepsm2014}, so the derivative of $h$ w.r.t.~$\rho$ is
\begin{align*}
\frac{\partial h(x_1, x_2 ; \rho) }{ \partial \rho} 
& = \frac{\rho x_1}{(1-\rho^2)^{3/2}} - \frac{\rho^2 x_2}{(1-\rho^2)^{3/2}} - \frac{x_2}{\sqrt{1-\rho^2}} \\
& =  \frac{\rho x_1 - \rho^2 x_2 - (1- \rho^2) x_2}{(1-\rho^2)^{3/2}}
=  \frac{\rho x_1 - x_2}{(1-\rho^2)^{3/2}}
\end{align*}
For the expectation of $\partial s(x_1, x_2 ; \rho) / \partial \rho$ (evaluated at the true parameter), we obtain
\begin{align*}
\E\lf[\frac{\partial s(X_1, X_2 ; \rho)}{ \partial \rho} \ri] 
& = \frac{1 + \rho^2}{(1 - \rho^2)} - 2 \frac{1 + 3 \rho^2}{(1 - \rho^2)^3} + 2 \frac{3 \rho^2 + \rho^4}{(1 - \rho^2)^3} \\
& = \frac{(1 - \rho^2)(1 + \rho^2) - 2 - 6 \rho^2 + 6 \rho^2 + 2 \rho^4}{(1 - \rho^2)^3}
= \frac{-1 + \rho^4}{(1 - \rho^2)^3} = - \frac{1 + \rho^2}{(1 - \rho^2)^2}
\end{align*}
using $\E(X_1^2 ) = \E(X_2^2) = 1$ and $\E(X_1 X_2) = \rho$ since $X_1, X_2$ are standard normally distributed with correlation $\rho$.
The statement for pair copulas $c_{ab;D}$ involving conditioning variables $D$ follows immediately since $\E(X_{a|D} X_{b|D}) = \rho_{ab;D}$.

For the second statement, the chain rule yields
\begin{align*}
\frac{\partial s_{2 i ; 1}(x_{2|1}, x_{i | 1}; \rho_{2i ; 1})}{\partial \rho_{12}}& = 
\frac{\partial s_{2 i ; 1}(x_{2|1}, x_{i | 1}; \rho_{2i ; 1})}{\partial x_{2|1}} \cdot \frac{\partial h(x_2, x_1; \rho_{12})}{\partial \rho_{12}} \\
&=\frac{(1 + \rho^2_{2i;1}) x_{i|1} - 2 \rho_{2i;1} x_{2|1}}{(1- \rho^2_{2i;1})^2}  \cdot \frac{\rho_{12} x_2 - x_1}{(1-\rho^2_{12})^{3/2}}.
\end{align*}
To compute its expectation, we need $\E[X_{i|1} X_2], \E[X_{i|1} X_1], \E[X_{2|1} X_1]$ and $\E[X_{2|1} X_2]$.
Plugging in the definition $x_{a|b} = h(x_a, x_b ; \rho_{ab})$ yields
\begin{align*}
\E[X_{i|1} X_1] & = \E \lf[\frac{X_i X_1 - \rho_{1i} X_1^2}{\sqrt{1 - \rho_{1i}^2}} \ri] 
= \frac{\rho_{1i} - \rho_{1i} }{\sqrt{1 - \rho_{1i}^2}} = 0 = \E[X_{2|1} X_1], \\
\E[X_{2|1} X_2] & = \E \lf[\frac{X_2^2 - \rho_{12} X_1 X_2}{\sqrt{1 - \rho_{12}^2}} \ri] 
= \frac{1 - \rho_{12}^2}{\sqrt{1 - \rho_{12}^2}} = \sqrt{1 - \rho_{12}^2}, \\
\E[X_{i|1} X_2] & = \E \lf[\frac{X_i X_2 - \rho_{1i} X_1 X_2}{\sqrt{1 - \rho_{1i}^2}} \ri] 
= \frac{\rho_{2i} - \rho_{1i}\rho_{12} }{\sqrt{1 - \rho_{1i}^2}} 
= \rho_{2i;1} \sqrt{1 - \rho_{12}^2},
\end{align*}
where we used the definition of the partial correlation \citep{Kurowicka2003} in the last step:
\[
 \rho_{2i;1} = \frac{\rho_{2i} - \rho_{1i}\rho_{12}}{\sqrt{(1 - \rho_{12}^2) (1 - \rho_{1i}^2})}.
\]
We now obtain
\begin{align*}
\E \lf[\frac{\partial s_{2 i ; 1}(X_{2|1}, X_{i | 1}; \rho_{2i ; 1})}{\partial \rho_{12} } \ri]
& = \E \lf[\frac{(1 + \rho^2_{2i;1}) X_{i|1} - 2 \rho_{2i;1} X_{2|1}}{(1- \rho^2_{2i;1})^2}  \cdot \frac{\rho_{12} X_2 - X_1}{(1-\rho^2_{12})^{3/2}}\ri] \\
& = \frac{(1 + \rho^2_{2i;1})  \rho_{12} \rho_{2i;1} \sqrt{1 - \rho_{12}^2} - 2  \rho_{2i;1} \rho_{12}  \sqrt{1 - \rho_{12}^2} }{(1- \rho^2_{2i;1})^2 (1-\rho^2_{12})^{3/2}}  \\
& = \frac{ \rho_{12} \rho_{2i;1} (\rho_{2i;1}^2 - 1 )}{(1- \rho^2_{2i;1})^2 (1-\rho^2_{12})}  
 = -  \frac{ \rho_{12} \rho_{2i;1} }{(1- \rho^2_{2i;1}) (1-\rho^2_{12})}  .
\end{align*}
\end{proof}

\textcolor{red}{\begin{lemma}
\label{lem:derivD-vine}
In a $d$-dimensional Gaussian D-vine with true parameter $\brho$, it holds that
\[
\E\left[\frac{\partial s(X_{1|2, \ldots, d-1}, X_{d|2, \ldots, d-1}; \rho_{1d;2, \ldots, d-1} )}{\partial \rho_k} \right] = - \frac{\rho_k \cdot\rho_{1d;2, \ldots, d-1} }{(1-\rho_k^2)(1-\rho_{1d;2, \ldots, d-1}^2)}
\]
for all $\rho_k$ that correspond to leaf nodes, i.e., $\rho_{12}, \rho_{13;2}, \rho_{14;23}, \ldots, \rho_{1 (d-1); 2, \ldots, d-2}$ and $\rho_{d-1, d}, \rho_{d-2, d;d - 1}, \ldots, \rho_{2, d; 3, \ldots, d-1}$.
\end{lemma}}

\begin{proof}
We only prove the statement for $\rho_k = \rho_{12}$.
For all other $\rho_k$, the statement either follows by symmetry (for $\rho_{(d-1, d}$) or from treating the vine as a lower-dimensional vine (i.e., the derivative w.r.t.~$\rho_{13;2}$ can be derived in the same way as the derivative of the estimating equation of $\rho_{1 (d-1); 2, \ldots, d-2}$ w.r.t.~$\rho_{12}$ etc.).

We use the above formulas for $s_1, h_1$ and $h'$.
This yields
\begin{align*}
 \frac{\partial s(x_{1|2, \ldots, d-1}, x_{d|2, \ldots, d-1}; \rho_{1 d;2, \ldots, d-1})}{\partial \rho_{12}} = & 
\frac{(1+\rho_{1 d;2, \ldots, d-1}^2)x_{d|2, \ldots, d-1} - 2\rho_{1 d;2, \ldots, d-1}x_{1|2, \ldots, d-1}}{(1-\rho_{1 d;2, \ldots, d-1}^2)^2} \\
 & \cdot \frac{1}{\sqrt{1 - \rho_{1(d-1);2, \ldots, d-2}^2}} \frac{1}{\sqrt{1 - \rho_{1(d-2);2, \ldots, d-3}^2}} \cdots \\& \cdots \frac{1}{\sqrt{1 - \rho_{13;2}^2}}  \frac{\rho_{12} x_1 - x_2}{(1-\rho_{12}^2)^{3/2}} \\
 = & \frac{(1+\rho_{1 d;2, \ldots, d-1}^2)x_{d|2, \ldots, d-1} - 2\rho_{1 d;2, \ldots, d-1}x_{1|2, \ldots, d-1}}{(1-\rho_{1 d;2, \ldots, d-1}^2)^2}  \cdot
\frac{\rho_{12} x_1 - x_2}{(1-\rho_{12}^2)} \cdot
\prod_{k = 2}^{d-1} \frac{1}{\sqrt{1 - \rho^2_{1k;2, \ldots, k-1}}}.
\end{align*}
We need the following expectations:
\begin{align*}
\E[X_{1|2, \ldots, d-1} \cdot X_1] & = \prod_{k = 2}^{d-1} \sqrt{1 - \rho^2_{1k; 2,\ldots, k-1} } ,\\
\E[X_{d|2, \ldots, d-1} \cdot X_1] & = \rho_{1d;2, \ldots, d-1}  \prod_{k = 2}^{d-1} \sqrt{1 - \rho^2_{1k; 2,\ldots, k-1} } ,\\
\E[X_{d|2, \ldots, d-1} \cdot X_2] & =  0, \\
\E[X_{1|2, \ldots, d-1} \cdot X_2] & = 0.
\end{align*}
The last two equations follow from the definition of conditional expectations.
The first two equations are shown in \cref{lem:expectD-vine}.
We therefore have
\begin{align*}
\E\left[  \frac{\partial s(x_{1|2, \ldots, d-1}, x_{d|2, \ldots, d-1}; \rho_{1 d;2, \ldots, d-1})}{\partial \rho_{12}}  \right] & = \frac{(1+\rho_{1 d;2, \ldots, d-1}^2) \rho_{12}\cdot \rho_{1d;2, \ldots, d-1} - 2 \rho_{12} \cdot \rho_{1d;2, \ldots, d-1}  }{ (1-\rho_{12}^2)(1-\rho_{1 d;2, \ldots, d-1}^2)^2} \\
& = - \frac{\rho_{12} \cdot \rho_{1d;2, \ldots, d-1} (1 - \rho_{1 d;2, \ldots, d-1}^2)}{ (1-\rho_{12}^2)(1-\rho_{1 d;2, \ldots, d-1}^2)^2} 
 = - \frac{\rho_{12} \cdot \rho_{1d;2, \ldots, d-1} }{(1-\rho_{12}^2)(1-\rho_{1 d;2, \ldots, d-1}^2) }.
\end{align*}

\end{proof}

\textcolor{red}{
\begin{lemma}
\label{lem:expectD-vine}
In a $d$-dimensional Gaussian D-vine with true parameter $\brho$, it holds that
\begin{align*}
\E[X_{1|2, \ldots, d-1} \cdot X_1] & = \prod_{k = 2}^{d-1} \sqrt{1 - \rho^2_{1k; 2,\ldots, k-1} } ,\\
\E[X_{d|2, \ldots, d-1} \cdot X_1] & = \rho_{1d;2, \ldots, d-1}  \prod_{k = 2}^{d-1} \sqrt{1 - \rho^2_{1k; 2,\ldots, k-1} } .
\end{align*}
\end{lemma}}
\begin{proof}
Both equalities can be shown by induction:

\textbf{Base case:} (d=3)
Since $\E[X_1^2] = 1$ and $\E[X_1 X_2] = \rho_{12}$, we have, using the definition of the Gaussian $h$-function,
\[
\E[X_{1|2} \cdot X_1]  = \E \lf[\frac{X_1 - \rho_{12} X_2}{\sqrt{1 - \rho_{12}^2} } \cdot X_1 \ri]
= \frac{1 - \rho_{12}^2}{\sqrt{1 - \rho_{12}^2} } = \sqrt{1 - \rho_{12}^2}.
\]
For the second equality, we have, using the definition of the partial correlation $\rho_{13;2}$,
\[
\E[X_{3|2} \cdot X_1] = \E\lf[\frac{X_3 - \rho_{23} X_2}{\sqrt{1 - \rho_{23}^2}} \cdot X_1 \ri] = \frac{\rho_{13} - \rho_{23} \rho_{12}}{\sqrt{1 - \rho_{23}^2}} = \rho_{13;2} \sqrt{1 - \rho_{12}^2}.
\]

\textbf{Induction step:}
We assume that the stated equalities hold for some fixed $d-1$ (induction hypothesis, IH).
It then holds that
\begin{align*}
\E[X_{1|2, \ldots, d-1} \cdot X_1] & = \E \lf[\frac{X_{1|2, \ldots, d-2} -  \rho_{1(d-1); 2, \ldots, d-2  } X_{d-1|2, \ldots, d-2} }{\sqrt{1 -  \rho_{1(d-1); 2, \ldots, d-2  }^2}} \cdot X_1 \ri]  = \frac{ \E[X_{1|2, \ldots, d-2} \cdot X_1] -  \rho_{1(d-1); 2, \ldots, d-2  } \E[ X_{d-1|2, \ldots, d-2} \cdot X_1]}{\sqrt{1 -  \rho_{1(d-1); 2, \ldots, d-2  }^2}} \\
& \overset{\text{IH}}{=} \frac{\prod_{k = 2}^{d-2} \sqrt{1 - \rho^2_{1k; 2,\ldots, k-1} } -  \rho_{1(d-1); 2, \ldots, d-2  }^2 \prod_{k = 2}^{d-2} \sqrt{1 - \rho^2_{1k; 2,\ldots, k-1} } }{\sqrt{1 -  \rho_{1(d-1); 2, \ldots, d-2  }^2}}  =  \prod_{k = 2}^{d-1} \sqrt{1 - \rho^2_{1k; 2,\ldots, k-1} } 
\end{align*}
and
\begin{align*}
& \E[X_{d|2, \ldots, d-1} \cdot X_1]  
 = \E \lf[\frac{X_{d|2, \ldots, d-2} - \rho_{(d-1), d; 2, \ldots, d-2} X_{(d-1)|2, \ldots, d-2} }{\sqrt{1 - \rho_{(d-1), d; 2, \ldots, d-2}^2}} \cdot X_1 \ri]  = \frac{\E[ X_{d|2, \ldots, d-2} \cdot X_1] - \rho_{(d-1), d; 2, \ldots, d-2} \E[ X_{(d-1)|2, \ldots, d-2} \cdot X_1]}{\sqrt{1 - \rho_{(d-1), d; 2, \ldots, d-2}^2}}  \\
& \overset{\text{IH}}{=}  \frac{\rho_{1d;2, \ldots, d-2} \prod_{k=2}^{d-2} \sqrt{1 - \rho^2_{1k;2, \ldots, k-1} }  - \rho_{(d-1), d; 2, \ldots, d-2}\cdot \rho_{1,(d-1);2, \ldots, d-2}  \prod_{k=2}^{d-2} \sqrt{1 - \rho^2_{1k;2, \ldots, k-1} } }{\sqrt{1 - \rho_{(d-1), d; 2, \ldots, d-2}^2}} \\
& = \frac{\rho_{1d;2, \ldots, d-2}  -  \rho_{(d-1), d; 2, \ldots, d-2} \cdot \rho_{1,(d-1);2, \ldots, d-2} } {\sqrt{1 - \rho_{(d-1), d; 2, \ldots, d-2}^2}}  \prod_{k=2}^{d-2} \sqrt{1 - \rho^2_{1k;2, \ldots, k-1} }  \\
& \overset{(*)}{=}  \rho_{1d;2, \ldots, d-1} \sqrt{1 - \rho^2_{1,(d-1); 2, \ldots, d-2}}  \prod_{k=2}^{d-2} \sqrt{1 - \rho^2_{1k;2, \ldots, k-1} }  = \rho_{1d;2, \ldots, d-1}  \prod_{k = 2}^{d-1} \sqrt{1 - \rho^2_{1k; 2,\ldots, k-1} },
\end{align*}
where we used the definition of the partial correlation $\rho_{1d;2, \ldots, d-1}$ in $(*)$.
\end{proof} 
\subsection{FGM vines}
\textcolor{red}{
\begin{lemma}\label{lem:FGMbounds} 
    For an FGM copula with $| \theta | \le \rho < 1$, the first- and second-order derivatives of the $h$-function w.r.t.~its first and second argument and $\theta$ satisfy
    \begin{align*}
  |\partial_1 h_\theta(u,v)| \le 1+|\theta |, \quad 
  |\partial_2 h_\theta(u,v)| \le | \theta |/2, \quad
  |\partial_\theta h_\theta(u,v)| \le 1/4,  \\
  |\partial_{11}h_\theta(u,v)| \le 2|\theta |,
  \qquad
  |\partial_{12}h_\theta(u,v)| \le 2|\theta |,
  \qquad
  \partial_{22}h_\theta(u,v)=0, \\
  |\partial_{1\theta}h_\theta(u,v)|\le1,
  \qquad
  |\partial_{2\theta}h_\theta(u,v)| \le1/2,
  \qquad
  \partial_{\theta\theta}h_\theta(u,v)=0. 
\end{align*}
All first- and second-order partial derivatives of $s_\theta(u,v)$ w.r.t.~$u,v,\theta$ are bounded in absolute value by
\[
| \partial_x s_\theta(u,v)|  \le 2 \delta_\rho^{-2}, 
   \qquad | \partial_{xy} s_\theta(u,v)|  \le 12 \delta_\rho^{-3},  \quad \text{where } \ x,y = v, u, \theta, \quad \delta_\rho \coloneqq 1-\rho,
\]
and it holds that
\[
  -\E[\partial_\theta s_\theta(U,V)]
  \ge \kappa_\rho \coloneqq (9(1+\rho))^{-1}.
\]
\end{lemma}}

\begin{proof}
    Recall the score and the $h$-function of an FGM copula:
    \[
    s_\theta(u,v)
  =\frac{a(u)a(v)}{1+\theta a(u)a(v)}, \qquad h_\theta(u,v)
  =u+\theta b(u)a(v), \quad \text{where } \ a(u) = 1 - 2u, \ b(u) = u(1-u).
    \]
    Writing $z=a(u)a(v)$, we have, using $| a(u) | \le 1$ and $\partial_u a(u) = -2$,
\[
  s_\theta(u,v)=\frac{z}{1+\theta z},
  \qquad
  |z|\le1,
  \qquad
  |\partial_u z|\le2,
  \qquad
  |\partial_v z|\le2,
  \qquad
  |\partial_{uv}z| \le4, \qquad \partial_{uu}z = 0.
\]
The first and second derivatives of $s_\theta(u,v)$ w.r.t.~$z$ and $\theta$ are
\begin{align*}
  & \partial_z\frac{z}{1+\theta z}
  = (1+\theta z)^{-1} - \theta z(1+\theta z)^{-2}
  =(1+\theta z)^{-2},
  \qquad
  \partial_\theta\frac{z}{1+\theta z}=-z^2(1+\theta z)^{-2}, \\
  & \partial_{zz}\frac{z}{1+\theta z}=-2\theta(1+\theta z)^{-3},
  \quad
  \partial_{z\theta}\frac{z}{1+\theta z}=-2z(1+\theta z)^{-3},
  \quad
  \partial_{\theta\theta}\frac{z}{1+\theta z}=2z^3(1+\theta z)^{-3}.
\end{align*}
The chain rule and the above bounds on the derivatives of $z$ yield
\begin{align*}
& | \partial_u s_\theta(u,v) | \le 2 \delta_\rho^{-2}, \quad | \partial_v s_\theta(u,v) | \le 2 \delta_\rho^{-2}, \quad | \partial_\theta s_\theta(u,v) | \le  \delta_\rho^{-2}, \quad | \partial_{\theta \theta} s_\theta(u,v) | \le 2 \delta_\rho^{-3}, \\
& | \partial_{uu} s_\theta(u,v) | = | \partial_u [\partial_z s_\theta(u,v) \: \partial_u z] |
= | (\partial_u z)^2 \partial_{zz} s_\theta(u,v) + \partial_z s_\theta(u,v) \: \partial_{uu} z | \le 8 \delta_\rho^{-3}, \\
& | \partial_{v\theta} s_\theta(u,v)| = | \partial_{z\theta} s_\theta(u,v) \: \partial_v z| \le 4 \delta_\rho^{-3},\\
& | \partial_{vu} s_\theta(u,v) | =  | \partial_v [\partial_z s_\theta(u,v) \: \partial_u z] |
= | \partial_v z  \: \partial_{zz} s_\theta(u,v) \: \partial_u z + \partial_z s_\theta(u,v) \: \partial_{vu} z | \le 12 \delta_\rho^{-3}, 
\end{align*}
using $\delta_\rho^{-2} \le \delta_\rho^{-3}$ since $\delta_\rho \in (0,1)$ in the last inequality.
For the $h$-function $h_\theta(u,v)=u+\theta b(u)a(v)$, it holds that
\[
  \partial_1 h_\theta(u,v)=1+\theta a(u)a(v),
  \quad
  \partial_2 h_\theta(u,v)=-2\theta b(u),
  \quad
  \partial_\theta h_\theta(u,v)=b(u)a(v),
\]
using $\partial_u b(u) = a(u)$ in the first equality.
The displayed bounds follow from $|a(u)| \le1$ and $| b(u) | \le1/4$.
For the second-order derivatives, we have
\begin{align*}
    & \partial_{11} h_\theta(u,v) = -2 \theta a(v), \quad \partial_{12} h_\theta(u,v) = -2 \theta a(u), \quad \partial_{22} h_\theta(u,v) = 0, \\
    & \partial_{1 \theta} h_\theta(u,v) = a(u)a(v), \quad \partial_{2\theta} h_\theta(u,v)= -2 b(u), \quad \partial_{\theta \theta} h_\theta(u,v) = 0.
\end{align*}
The stated bounds follow immediately.
It remains to show $-\E[\partial_\theta s_\theta(U,V)] \ge \kappa_\rho$.
Plugging in the density $c_\theta(u,v)=1+\theta a(u)a(v)$, we have
\[
-\E[\partial_\theta s_\theta(U,V)] = \int_0^1\int_0^1
    \frac{a(u)^2a(v)^2}{1+\theta a(u)a(v)}\mathrm d u \mathrm d  v \ge \frac{1}{9(1+\rho)} = \kappa_\rho
    \]
since $1+\theta a(u)a(v) \le 1 + \rho$ and $\int_0^1 a(u)^2 \mathrm d = 1/3$.
\end{proof}
\textcolor{red}{
\begin{lemma}\label{lem:FGM1deriv} 
  Consider a $d_n$-dimensional FGM vine, where all parameters in a given tree level $t$ share the same true value, denoted by $\theta_t^*$, and denote
\[
\eta_{t,n} \coloneqq \sup_{\btheta \in \Theta_n} \max_{e \in E_t} |\theta_e | \le \rho < 1, \qquad \delta_\rho \coloneqq 1 - \rho.
\]
  For any parameter $\theta_j$ in any tree $t = 1, \ldots, d_n - 1$, it holds that
  \[
  \sup_{\btheta \in \Theta_n} \sup_{\bu \in [0,1]^{d_n}} \sum_{k=1}^{j-1} \left| \frac{\partial}{\partial \theta_k} \phi(\bu; \btheta)_j \right| \le  \delta_\rho^{-2} \sum_{l=1}^{t-1}  \prod_{m = l + 1}^{t-1} \left( 1 + \frac 3 2 \eta_{m, n} \right).
  \]
  If we replace $\btheta \in \Theta_n$ by $\btheta^*$, we may replace $\eta_{m,n}$ in the bound by $|\theta_m^*|$.
\end{lemma}}
\begin{proof}
  The score $\phi(\bu; \btheta)_j$ of parameter $\theta_j$ belonging to edge $(a,b; D)$ in tree $t$ is given by $s_{\theta_j}(u_{a|D}, u_{b|D})$.
  To obtain the derivatives w.r.t.~parameters in trees $1, \ldots, t-1$, observe that
  \[
 \partial_{\theta_k} s_{\theta_j}(u_{a|D}, u_{b|D}) 
  = \partial_u s_{\theta_j}(u_{a|D}, u_{b|D}) \partial_{\theta_k} u_{a|D} + \partial_v s_{\theta_j}(u_{a|D}, u_{b|D}) \partial_{\theta_k} u_{b|D}.
  \]
  From \cref{lem:FGMbounds}, we know that $| \partial_u s_{\theta_j}(u_{a|D}, u_{b|D}) | \le 2 \delta_\rho^{-2}$.
  We therefore have
 \begin{equation}\label{eq:FGMtreeDerivSum}
  \sup_{\btheta \in \Theta_n} \sup_{\bu \in [0,1]^{d_n}} \sum_{k=1}^{j-1} | s_{\theta_j}(u_{a|D}, u_{b|D}) | \le 2 \delta_\rho^{-2} \left(\sup_{\btheta \in \Theta_n} \sup_{\bu \in [0,1]^{d_n}} \sum_{k=1}^{j-1}  | \partial_{\theta_k} u_{a|D} | + \sup_{\btheta \in \Theta_n} \sup_{\bu \in [0,1]^{d_n}} \sum_{k=1}^{j-1}  | \partial_{\theta_k} u_{b|D} |  \right).
 \end{equation}
  We show that 
\begin{equation} \label{eq:FGMtreeDerivAll}
  \sup_{\btheta \in \Theta_n} \sup_{\bu \in [0,1]^{d_n}} \sum_{k=1}^{j-1}  | \partial_{\theta_k} u_{a|D} | \le 
  \frac 1 4 \sum_{l=1}^{t-1}  \prod_{m = l + 1}^{t-1} \left( 1 + \frac 3 2 \eta_{m, n} \right)
  \end{equation}
  by showing via recursion that for each tree $l = 1, \ldots, t-1$, the sum of partial derivatives of $u_{a|D}$ w.r.t.~the parameters in tree $l$ is bounded by 
  \begin{equation} \label{eq:FGMtreeDeriv}
   \sup_{\btheta \in \Theta_n} \sup_{\bu \in [0,1]^{d_n}} \sum_{e \in E_{l} }| \partial_{\theta_e} u_{a|D} | \le
   \frac 1 4 \prod_{m = l + 1}^{t-1} \left( 1 + \frac 3 2 \eta_{m, n} \right), 
  \end{equation}
  which immediately implies \eqref{eq:FGMtreeDerivAll}.
  The same bound holds for $| \partial_{\theta_e} u_{b|D} |$.
  Plugging them into \eqref{eq:FGMtreeDerivSum} yields the claim.

  To prove \eqref{eq:FGMtreeDeriv}, note that in tree $t-1$, only one parameter $\theta_k$ affects $u_{a|D} = h_{\theta_k}(u_{\tilde a | \tilde D}, u_{\tilde b | \tilde D})$, belonging to edge $(\tilde a, \tilde b; \tilde D)$.
  All other derivatives of $u_{a|D}$ w.r.t.~parameters in tree $t-1$ are 0. 
 The bound from \cref{lem:FGMbounds} gives
  \[
  |\partial_{\theta_k} u_{a|D}| = | \partial_{\theta_k} h_{\theta_k}(u_{\tilde a | \tilde D}, u_{\tilde b | \tilde D}) | \le 1/4, \qquad \text{ so } \qquad \sup_{\btheta \in \Theta_n} \sup_{\bu \in [0,1]^{d_n}} \sum_{e \in E_{t-1} }| \partial_{\theta_e} u_{a|D} | \le 1/4,
  \]
  which proves \eqref{eq:FGMtreeDeriv} for $l = t-1$.
  In tree $t-2$, there is exactly one parameter $\theta_{k1}$ that affects $u_{a|D}$ through $u_{\tilde a | \tilde D}$ and one $\theta_{k2}$ that affects $u_{a|D}$ through $u_{\tilde b | \tilde D}$: 
  \begin{align*}
      |\partial_{\theta_{k1}} u_{a|D}| +  |\partial_{\theta_{k2}} u_{a|D}| 
   & = |  \partial_1 h_{\theta_k}(u_{\tilde a | \tilde D}, u_{\tilde b | \tilde D})  \partial_{\theta_{k1}} h_{\theta_{k1}}(\ldots) | + | \partial_2 h_{\theta_k}(u_{\tilde a | \tilde D}, u_{\tilde b | \tilde D}) \partial_{\theta_{k2}} h_{\theta_{k2}}(\ldots)  | \\
   &  \le \frac{1}{4} (|  \partial_1 h_{\theta_k}(u_{\tilde a | \tilde D}, u_{\tilde b | \tilde D}) | + |  \partial_2 h_{\theta_k}(u_{\tilde a | \tilde D}, u_{\tilde b | \tilde D}) |) \le \frac{1}{4} (1 + \eta_{t-1, n} + \eta_{t-1, n}/2) = \frac{1}{4} \left(1 + \frac 3 2 \eta_{t-1, n}\right),
  \end{align*}
  using the bounds on $|\partial_1 h_\theta |, |\partial_2 h_\theta |$ from \cref{lem:FGMbounds} and $\sup_{\btheta \in \Theta_n} | \theta_k| \le \eta_{t-1, n}$ for $\theta_k$ in tree $t-1$.
  This gives \eqref{eq:FGMtreeDeriv} for $l = t-2$.

Similarly, one can prove the bound for a general tree $l$ via induction:
Assume that the bound in \eqref{eq:FGMtreeDeriv} holds for tree $l+1$.
According to the chain rule, the sum $\sum_{e \in E_{l+1} }| \partial_{\theta_e} u_{a|D} |$ is a sum of products of $\partial_1 h$ and $\partial_2 h$ terms and a $\partial_{\theta_e} h_{\theta_e}$ term.
A parameter $\theta_e$ may correspond to more than one of these products, depending on the vine structure (or, more concretely, depending on the number of ``appearances'' of $\theta_e$ in the recursive computation of $u_{a|D}$, see also \cref{sec:A3theory}).
Let $\theta_k$ be a fixed parameter in tree $l+1$, corresponding to edge $(a', b'; D')$. 
The last term of each of the products corresponding to $\theta_k$ in  $\sum_{e \in E_{l+1} }| \partial_{\theta_e} u_{a|D} |$ is $\partial_{\theta_k} h_{\theta_k} (u_{a'| D'}, u_{b'| D'} )$ (or $\partial_{\theta_k} h_{\theta_k} (u_{b'| D'}, u_{a'| D'} )$).
Each of these products yields two products in $\sum_{e \in E_l }| \partial_{\theta_e} u_{a|D} |$ by replacing $\partial_{\theta_k} h_{\theta_k} (u_{a'| D'}, u_{b'| D'} )$ by $\partial_1 h_{\theta_k} (u_{a'| D'}, u_{b'| D'} ) \partial_\theta h_\theta(u,v)$ and $\partial_2 h_{\theta_k} (u_{a'| D'}, u_{b'| D'} ) \partial_{\tilde \theta} h_{\tilde \theta}(\tilde u, \tilde v)$ (since we differentiate $h_{\theta_k} (u_{a'| D'}, u_{b'| D'} )$ w.r.t.~its first and second argument).
$\theta, \tilde \theta$ are the parameters corresponding to $u_{a'| D'}$ and $u_{b'| D'}$, and $u,v, \tilde u, \tilde v$ are the corresponding conditional variables.
If the given term in $\sum_{e \in E_l }| \partial_{\theta_e} u_{a|D} |$ was bounded by some $c$, the sum of the two new terms is therefore now bounded by 
\[
c \,(| \partial_1 h_{\theta_k} (u_{a'| D'}, u_{b'| D'} ) | + | \partial_2 h_{\theta_k} (u_{a'| D'}, u_{b'| D'} ) |) \le c \, (1 + \eta_{l+1, n} + \eta_{l+1,n}/2 )= c\left(1 + \frac 3 2 \eta_{l+1, n} \right).
\]
Taking the sum over all terms in $\sum_{e \in E_l }| \partial_{\theta_e} u_{a|D} |$, we obtain 
\[
\sum_{e \in E_l }| \partial_{\theta_e} u_{a|D} | \le \sum_{e \in E_{l + 1} }| \partial_{\theta_e} u_{a|D} | \left(1 + \frac 3 2 \eta_{l+1, n} \right) = \frac 1 4 \prod_{m = l + 2}^{t-1} \left( 1 + \frac 3 2 \eta_{m, n} \right)  \left(1 + \frac 3 2 \eta_{l+1, n} \right) =  \frac 1 4 \prod_{m = l + 1}^{t-1} \left( 1 + \frac 3 2 \eta_{m, n} \right).
\]
This concludes the proof of \eqref{eq:FGMtreeDeriv}.
A brief example illustrates the induction step in more detail:
Take a D-vine and consider a term in $\sum_{e \in E_2 }| \partial_{\theta_e} u_{a|D} |$ that is a derivative w.r.t.~$\theta_{13;2}$.
This term has the form
\[
\cdots \overbrace{| \partial_{\theta_{13;2}} h_{\theta_{13;2}}(u_{1|2}, u_{3|2})|}^{ \le 1/4}.
\]
This term yields two terms in the sum of derivatives w.r.t.~parameters in tree 1:
\[
\cdots \underbrace{\left[ \overbrace{| \partial_1 h_{\theta_{13;2}}(u_{1|2}, u_{3|2}) |}^{\le 1 + \eta_{2,n}} \overbrace{ |\partial_{\theta_{12}} h_{\theta_{12}}(u_1, u_2)|}^{\le 1/4} + \overbrace{| \partial_2 h_{\theta_{13;2}}(u_{1|2}, u_{3|2}) |}^{\le \eta_{2,n}/2} \overbrace{ |\partial_{\theta_{23}} h_{\theta_{23}}(u_3, u_2)|}^{\le 1/4} \right]}_{\le 1/4 (1 + (3/2) \eta_{2,n})} .
\]
Applying this bound to all terms in $\sum_{e \in E_2 }| \partial_{\theta_e} u_{a|D} |$ yields $\sum_{e \in E_1 }| \partial_{\theta_e} u_{a|D} | \le \sum_{e \in E_2 }| \partial_{\theta_e} u_{a|D} |(1 + (3/2) \eta_{2,n})$. 
\end{proof}
 \textcolor{red}{
\begin{lemma}\label{lem:FGM_J} 
  Consider a $d_n$-dimensional FGM vine, where all parameters in a given tree level $t$ share the same true value, denoted by $\theta_t^*$. Assume $\max_{1 \le t \le d_n - 1} | \theta_t^*| \le \rho$ and let $\delta_\rho \coloneqq 1 - \rho$.
  For any parameter $\theta_j$ in any tree $t = 1, \ldots, d_n - 1$, it holds that
  \[
  \sum_{k=1}^{j-1} \left| \E\left[ \frac{\partial}{\partial \theta_k} \phi(\bU; \btheta^*)_j  \right] \right|
  \le \delta_\rho^{-1} | \theta_t^*| \sum_{l=1}^{t-1}  \prod_{m = l + 1}^{t-1} \left( 1 + \frac 3 2 |\theta_m^*| \right).
  \]
\end{lemma}}
\begin{proof}
  For $\theta_j$ in tree $t$ belonging to edge $(a,b; D)$ and a fixed lower-tree parameter $\theta_k$, using the formulas from \cref{lem:FGMbounds}, it holds that
  \[
  \frac{\partial}{\partial \theta_k} \phi(\bu; \btheta^*)_j
  =  \partial_{\theta_k} \frac{z(\bu)}{1+\theta_t^* z(\bu)} = \partial_z \frac{z(\bu)}{1+\theta_t^* z(\bu)}  \partial_{\theta_k} z(\bu)
  = \frac{\partial_{\theta_k} z(\bu)}{(1+\theta_t^* z(\bu))^2} 
  \]
  with $z(\bu) = a(u_{a|D}) a(u_{b|D}) \in [-1, 1]$ and $a(u) = 1-2u$.
  Let $c_{\btheta^*}$ denote the density of the true model and let $c_{\btheta^*, \theta_j^* = 0}$ be the density of the same model except for $\theta_j^* = 0$.
  The densities are the products of the pair copula densities and since $c_\theta(u,v)=1+\theta z$, we have $c_{\btheta^*} = c_{\btheta^*, \theta_j^* = 0} \cdot (1 + \theta_j^* z)$.
  Let $\E_0$ denote the expectation under this model.
  Note that $\E_0[z(\bu)] = 0$, as under this model, $a(u_{a|D})$ and $a(u_{b|D})$ are independent uniform random variables on $[-1, 1]$.
  Now
  \[
  \E\left[ \frac{\partial}{\partial \theta_k} \phi(\bU; \btheta^*)_j\right] 
  = \int_{[0,1]^{d_n}}  \frac{\partial_{\theta_k} z(\bu)}{(1+\theta_t^* z(\bu))^2} c_{\btheta^*}(\bu) \mathrm d \bu
  = \int_{[0,1]^{d_n}}  \frac{\partial_{\theta_k} z(\bu)}{1+\theta_t^* z(\bu)} c_{\btheta^*, \theta_j^* = 0}(\bu) \mathrm d \bu = \E_0\left[ \frac{\partial_{\theta_k} z(\bU)}{1+\theta_t^* z(\bU)} \right].
  \]
  We may substract $\E_0[\partial_{\theta_k} z(\bU)] = \partial_{\theta_k} \E_0[z(\bU)] = 0$ and obtain
  \[
  \left| \E_0\left[ \frac{\partial_{\theta_k} z(\bU)}{1+\theta_t^* z(\bU)} - \partial_{\theta_k} z(\bU) \right]  \right|
  = \left| \E_0\left[ \left( \frac{1}{1+\theta_t^* z(\bU)} -  1 \right) \partial_{\theta_k} z(\bU) \right] \right|
  = \left| \E_0\left[  \frac{- \theta_t^* z(\bU)}{1+\theta_t^* z(\bU)}    \partial_{\theta_k} z(\bU) \right] \right|
  \le \frac{|\theta_t^*|}{\delta_\rho } \E_0 [| \partial_{\theta_k} z(\bU) |],
  \]
  since $| \theta_t^* z(\bu) | \le | \theta_t^* |$ and $|1 + \theta_t^* z(\bu) | \ge | 1 - |\theta_t^*| \ge \delta_\rho$.
  By the definition of $z(\bu)$ and $\partial_u a(u) = -2$, we have
  \[
  | \partial_{\theta_k} z(\bu) | \le 2 | \partial_{\theta_k} u_{a|D} | + 2 | \partial_{\theta_k} u_{b|D} | .
  \]
  Summing over all parameters in all trees $l = 1, \ldots, t-1$ and applying the bound from \eqref{eq:FGMtreeDerivAll} in \cref{lem:FGM1deriv} with $\eta_{t,n}$ replaced by $|\theta_t^*|$, we obtain
  \[
  \sum_{k=1 }^{j-1} | \partial_{\theta_k} z(\bu) | \le \sum_{l=1}^{t-1}  \prod_{m = l + 1}^{t-1} \left( 1 + \frac 3 2 |\theta_m^*| \right),
  \]
  which implies the claim.
\end{proof}

\textcolor{red}{
  \begin{lemma}\label{lem:FGM2deriv} 
Consider a $d_n$-dimensional FGM vine, where all parameters in a given tree level $t$ share the same true value, denoted by $\theta_t^*$, and denote
\[
\eta_{t,n} \coloneqq \sup_{\btheta \in \Theta_n} \max_{e \in E_t} |\theta_e | \le \rho < 1, \qquad \delta_\rho \coloneqq 1 - \rho, \qquad P_{t,n} \coloneqq \sum_{l=1}^t \prod_{m = l + 1}^t \left( 1 + \frac 3 2 \eta_{m, n} \right).
\]
For any parameter $\theta_j$ in any tree $t = 1, \ldots, d_n - 1$, it holds that
\[
\sup_{\btheta \in \Theta_n} \sup_{\bu \in [0,1]^{d_n}} \sum_{k=1}^j \sum_{k'=1}^j \left| \frac{\partial^2}{\partial \theta_k \partial \theta_{k'} } \phi(\bu; \btheta)_j \right| \le 12 \delta_\rho^{-3}   (1 + P_{t-1, n})^2.
\]
\end{lemma} }

\begin{proof}
  For $\theta_j$ in tree $t$ belonging to edge $(a,b; D)$ and $k = k' = j$, \cref{lem:FGMbounds} gives $| \partial_{\theta_j \theta_j} s_{\theta_j} (u_{a|D} , u_{b|D}) | \le 12 \delta_\rho^{-3}$.
  For $k' = j, k \neq j$, we have, using the bound for$\sum_{k=1}^{j-1} | \partial_{\theta_k} u_{a|D} | \le P_{t-1, n}/4$ from \eqref{eq:FGMtreeDerivAll} in \cref{lem:FGM1deriv},
  \begin{align*}
   \sum_{k=1}^{j-1}| \partial_{\theta_k \theta_j} s_{\theta_j} (u_{a|D} , u_{b|D}) |
  & = \sum_{k=1}^{j-1} | \partial_{u \theta_j} s_{\theta_j} (u_{a|D} , u_{b|D}) \partial_{\theta_k} u_{a|D} +  \partial_{v \theta_j} s_{\theta_j} (u_{a|D} , u_{b|D}) \partial_{\theta_k} u_{b|D} | \\
  & \le 12 \delta_\rho^{-3} \left(\sum_{k=1}^{j-1}| \partial_{\theta_k} u_{a|D} | + \sum_{k=1}^{j-1}| \partial_{\theta_k} u_{b|D} |\right)
  \le  6 \delta_\rho^{-3} P_{t-1, n}.
  \end{align*}
  The same argument holds for $k = j, k' \neq j$, so the sum over all second derivatives with $(k = j \lor k'= j) \land k \neq k'$ is bounded by $12 \delta_\rho^{-3} P_{t-1, n}$.

  It remains to treat the case where both $\theta_k, \theta_{k'}$ belong to lower trees.
  In this case, the chain rule yields
  \[
  \partial_{\theta_k \theta_{k'}} s_{\theta_j} (u_{a|D} , u_{b|D}) 
  = \partial_{\theta_k} [\partial_u s_{\theta_j} (u_{a|D} , u_{b|D}) \partial_{\theta_{k'}} u_{a|D} + \partial_v s_{\theta_j} (u_{a|D} , u_{b|D}) \partial_{\theta_{k'}} u_{b|D} ]
  \]
  and, using the bounds from \cref{lem:FGMbounds},
  \begin{align*}
   & \partial_{\theta_k} [\partial_u s_{\theta_j} (u_{a|D} , u_{b|D}) \partial_{\theta_{k'}} u_{a|D}] \\
   & = \partial_u s_{\theta_j} (u_{a|D} , u_{b|D}) \partial_{\theta_k \theta_{k'}} u_{a|D} + \partial_{\theta_k u } s_{\theta_j} (u_{a|D} , u_{b|D}) \partial_{\theta_{k'}} u_{a|D}  \\
   & = \underbrace{ \partial_u s_{\theta_j} (u_{a|D} , u_{b|D}) }_{| \ldots | \le 2 \delta_\rho^{-2} } \partial_{\theta_k \theta_{k'}} u_{a|D}
   + \underbrace{\partial_{uu} s_{\theta_j} (u_{a|D} , u_{b|D})}_{| \ldots |\le 12 \delta_\rho^{-3}} \partial_{\theta_k} u_{a|D} \partial_{\theta_{k'}} u_{a|D}
   + \underbrace{\partial_{vu} s_{\theta_j} (u_{a|D} , u_{b|D})}_{| \ldots | \le 12 \delta_\rho^{-3}} \partial_{\theta_k} u_{b|D} \partial_{\theta_{k'}} u_{a|D},
  \end{align*}
  and similar for $\partial_{\theta_k} [\partial_v s_{\theta_j} (u_{a|D} , u_{b|D}) \partial_{\theta_{k'}} u_{b|D}]$.
  Summing over all $k, k'$ and applying $\sum_{k=1 }^{j-1} | \partial_{\theta_k}  u_{a|D} | \le P_{t-1, n}/4$ from \eqref{eq:FGMtreeDerivAll} in \cref{lem:FGM1deriv} and \cref{lem:FGM2deriv2}, we obtain
  \[
  \sum_{k=1}^{j-1} \sum_{k' = 1}^{j-1} | \partial_{\theta_k \theta_{k'}} s_{\theta_j} (u_{a|D} , u_{b|D}) |
  \le 2 \left( 2 \delta_\rho^{-2}  \sum_{k=1}^{j-1} \sum_{k' = 1}^{j-1} | \partial_{\theta_k \theta_{k'}} u_{a|D} |  
  + 12 \delta_\rho^{-3} 2 (P_{t-1, n} /4)^2\right)
  = 4 \delta_\rho^{-2} P_{t-1, n}^2
  + 3 \delta_\rho^{-3}P_{t-1, n}^2 \le 7 \delta_\rho^{-3}P_{t-1, n}^2.
  \]
  Combining the three bounds, we have
  \begin{align*}
    \sum_{k=1}^j \sum_{k' = 1}^j | \partial_{\theta_k \theta_{k'}} s_{\theta_j} (u_{a|D} , u_{b|D}) |
    & = | \partial_{\theta_j \theta_j} s_{\theta_j} (u_{a|D} , u_{b|D}) | + 2 \sum_{k=1}^{j-1}  | \partial_{\theta_k \theta_j} s_{\theta_j} (u_{a|D} , u_{b|D}) |  + \sum_{k=1}^{j-1} \sum_{k' = 1}^{j-1} | \partial_{\theta_k \theta_{k'}} s_{\theta_j} (u_{a|D} , u_{b|D}) | \\
    & \le 12 \delta_\rho^{-3} + 12 \delta_\rho^{-3} P_{t-1, n} + 7 \delta_\rho^{-3} P_{t-1, n}^2
    \le 12 \delta_\rho^{-3} (1 + 2 P_{t-1, n} + P_{t-1, n}^2) = 12 \delta_\rho^{-3}   (1 + P_{t-1, n})^2.
  \end{align*}
  Note that one could improve the constant, but this is not relevant for our application.
\end{proof}

\textcolor{red}{
\begin{lemma}\label{lem:FGM2deriv2}
  Consider a $d_n$-dimensional FGM vine, where all parameters in a given tree level $t$ share the same true value, denoted by $\theta_t^*$, and denote
  \[
  \eta_{t,n} \coloneqq \sup_{\btheta \in \Theta_n} \max_{e \in E_t} |\theta_e | \le \rho < 1, \qquad P_{t,n} \coloneqq \sum_{l=1}^t \prod_{m = l + 1}^t \left( 1 + \frac 3 2 \eta_{m, n} \right).
  \]
For any parameter $\theta_j$ belonging to edge $(a,b;D)$ in any tree $t = 1, \ldots, d_n - 1$, it holds that
  \[
   \sup_{\btheta \in \Theta_n} \sup_{\bu \in [0,1]^{d_n}}  \sum_{k=1}^j \sum_{k' = 1}^j | \partial_{\theta_k \theta_{k'}} h_{\theta_j}(u_{a|D}, u_{b|D}) |  \le P_{t , n}^2.
  \]
\end{lemma}}

\begin{proof}
  We prove the claim by induction.

  \textbf{Base case:}
  In tree $t=1$, we have 
  \[
  \sum_{k=1}^j \sum_{k' = 1}^j | \partial_{\theta_k \theta_{k'}} h_{\theta_j}(u_a, u_b)|
  = | \partial_{\theta_j \theta_j} h_{\theta_j}(u_a, u_b) | = 0
  \]
  by \cref{lem:FGMbounds}.

  \textbf{Induction step:}
  Assume that the claim holds for all parameters in tree $t - 1$.
  For some $\theta_j$ in tree $t$ with $h_{\theta_j}(u, v)$ with conditional $u,v$, it holds that
  \begin{align*}
    \partial_{\theta_k \theta_{k'}} h_{\theta_j}(u, v)
    & =  \partial_{\theta_k} [\partial_1 h_{\theta_j}(u, v) \partial_{\theta_{k'}} u + \partial_2 h_{\theta_j}(u, v) \partial_{\theta_{k'}} v] \\
    & \quad  + \partial_{\theta_k} [\partial_{\theta_j}h_{\theta_j}(u, v) ] \ind_{k' = j \land k \neq j}
    + \partial_{\theta_{k'}} [\partial_{\theta_j}h_{\theta_j}(u, v) ] \ind_{k = j \land k' \neq j } \\
    & \quad + \underbrace{\partial_{\theta_j \theta_j} h_{\theta_j}(u, v)\ind_{k = j \land k' = j}}_{=0 \text{ by \cref{lem:FGMbounds}}}.
  \end{align*}
  Writing $h \coloneqq h_{\theta_j}(u, v)$, the chain rule yields
  \begin{align*}
      \partial_{\theta_k} [\partial_1 h \, \partial_{\theta_{k'}} u + \partial_2 h \, \partial_{\theta_{k'}} v] 
   & =  \partial_{11} h \, \partial_{\theta_k} u \, \partial_{\theta_{k'}} u 
   + \partial_{21} h \,\partial_{\theta_k} v \,\partial_{\theta_{k'}} u 
   + \partial_1 h \, \partial_{\theta_k \theta_{k'}} u \\
   & \quad + \partial_{22} h \, \partial_{\theta_k} v \, \partial_{\theta_k'} v 
   + \partial_{12} h \, \partial_{\theta_k} u \, \partial_{\theta_k'} v + \partial_2 h \, \partial_{\theta_k \theta_{k'}} v
  \end{align*}
  and 
  \begin{align*}
  \partial_{\theta_k} [\partial_{\theta_j}h_{\theta_j}(u, v) ] \ind_{k' = j \land k \neq j}
  & = ( \partial_{1 \theta_j} h \, \partial_{\theta_k} u  + \partial_{2 \theta_j} h \, \partial_{\theta_k} v ) \ind_{k' = j \land k \neq j} , \\
  \partial_{\theta_{k'}} [\partial_{\theta_j}h_{\theta_j}(u, v) ] \ind_{k = j \land k' \neq j }
  & = ( \partial_{1 \theta_j} h \, \partial_{\theta_{k'}} u  + \partial_{2 \theta_j} h \, \partial_{\theta_{k'}} v ) \ind_{k = j \land k' \neq j }. 
  \end{align*}
  \cref{lem:FGMbounds} gives the bounds
  \[
  |\partial_1 h| + |\partial_2 h | \le 1 + \frac 3 2 \eta_{t, n}, \quad | \partial_{11} h | \le 2 \eta_{t, n}  , \quad  \partial_{22} h  = 0, \quad | \partial_{21} h | \le 2 \eta_{t, n} , \quad  |\partial_{1 \theta_j} h | \le 1  , \quad | \partial_{2 \theta_j} h | \le 1/2 .
  \]
  Denote by $I_t$ the set of all indices $k$ with $\theta_k \in E_l, l < t$.
  Summing the previous displays over all pairs $(k, k')$, using the induction hypothesis and applying $\sum_{k=1}^{j-1} | \partial_{\theta_k} u_{a|D} | \le P_{t - 1, n}/4$ from \eqref{eq:FGMtreeDerivAll} in \cref{lem:FGM1deriv}, we obtain
 \begin{align*}
  \sum_{k = 1}^j \sum_{k' = 1}^j | \partial_{\theta_k \theta_{k'}} h_{\theta_j}(u, v) |
  & \le 2 (|\partial_{1 \theta_j} h | + |\partial_{2 \theta_j} h |)P_{t - 1, n}/4 
  + (| \partial_{11} h | + 2 | \partial_{12} h |) (P_{t - 1, n}/4)^2 \\
  & \quad + (|\partial_1 h| + |\partial_2 h |) \sum_{k \in I_t} \sum_{k'  \in I_t}  | \partial_{\theta_k \theta_{k'}} u | \\
  & \le \frac 3 4  P_{t - 1, n} + \frac {3}{8}  \eta_{t,n} P_{t - 1, n}^2 + \left( 1 + \frac 3 2 \eta_{t, n}\right) P_{t-1, n}^2
 \end{align*}
 Denote $a_{t,n} = 1 + (3/2) \eta_{t,n}$ and note that
 \[
 P_{t,m} = \sum_{l=1}^{t-1} \prod_{m = l + 1}^{t-1} a_{m,n}  a_{t,n} + \prod_{m=t+1}^t a_{m,t} = a_{t.n} P_{t-1,m} + 1 , \qquad P_{t,m}^2 = a_{t,n}^2 P_{t-1,n}^2 + 2 a_{t,n} P_{t-1, n} + 1, 
 \]
 and $a_{t,n}^2 = 1 + 3 \eta_{t,n} + (9/4) \eta_{t,n}^2$.
 We therefore have
 \begin{align*}
   \sum_{k = 1}^j \sum_{k' = 1}^j | \partial_{\theta_k \theta_{k'}} h_{\theta_j}(u, v) |
  & \le P_{t-1,n}^2 ( 1 + (15/8) \eta_{t,n}) + (3/4) P_{t-1, n}
  \le a_{t,n}^2 P_{t-1,n}^2 + 2 a_{t,n} P_{t-1, n} + 1 = P_{t,n}^2.
 \end{align*}

\end{proof} \newpage
\section{Simulation study--additional figures}

The supplement contains the following additional figures for the simulation study:
\begin{itemize}
  \item \cref{fig:par_gauss_gumbel_norm}: Parameter estimation for Gaussian and Gumbel vines, mean maximum norm of estimation error for different proportions of $d$ and $n$. 
  \item \cref{fig:par_dvine}: Parameter estimation for Gaussian D-vine with $\theta_t^* = 0.5/\sqrt{t+1}$, mean maximum norm of estimation error with different normalizations.
  \item \cref{fig:par_student_norm}: Parameter estimation for Student's t vines, mean maximum norm of estimation error for different proportions of $d$ and $n$.
  \item \cref{fig:par_student}: Parameter estimation for Student's t vines, sum of estimation errors.
  \item \cref{fig:par_estim_trunc}: 2-truncated Gaussian C-vine, sum of estimation errors. 
  \item \textcolor{red}{\cref{fig:par_mix}: Parameter estimation for a mixture of two rotated Gumbels, mean maximum norm of estimation error for different proportions of $d$ and $n$.}
  \item \textcolor{red}{\cref{fig:par_mix2}: Parameter estimation for a mixture of two rotated Gumbels, sum of estimation errors.}
  \item \cref{fig:par_margins_student_gumbel}: Parameter estimation for Gumbel and Student’s t vines with nonparametric estimation of margins, sum of estimation errors.
\end{itemize}

\begin{figure}
  \centering
  \includegraphics[width = \textwidth]{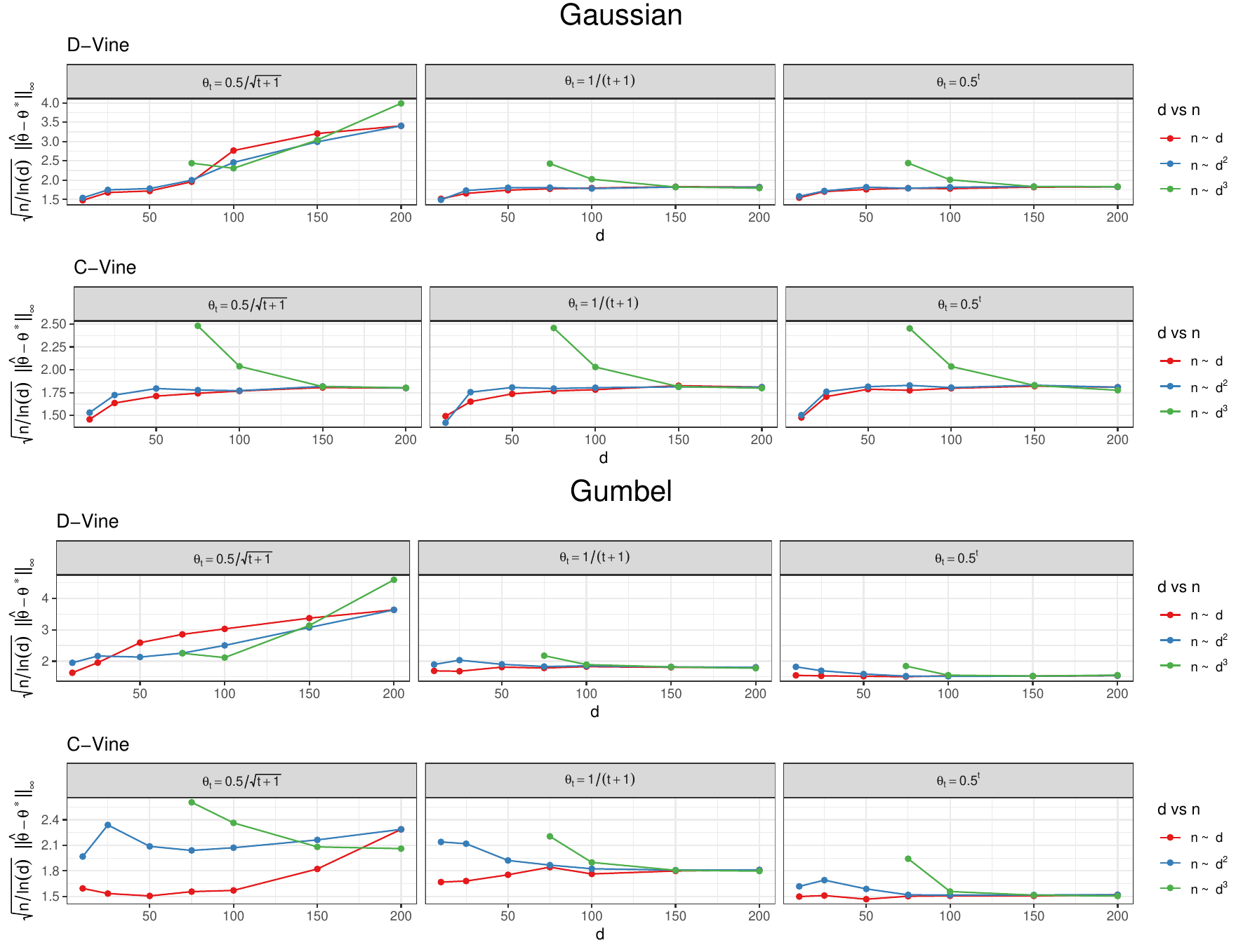}
  \caption{Parameter estimation for Gaussian and Gumbel vines, mean maximum norm of estimation error for different proportions of $d$ and $n$. Parameterization: $\theta = \rho$ for Gaussian and $\text{Gumbel}(\theta + 1)$ for $\theta \ge 0$, $\text{Gumbel}_{90}(-\theta+1)$ for $\theta <0$. }
  \label{fig:par_gauss_gumbel_norm}
\end{figure}

\begin{figure}
  \centering
  \includegraphics[width = \textwidth]{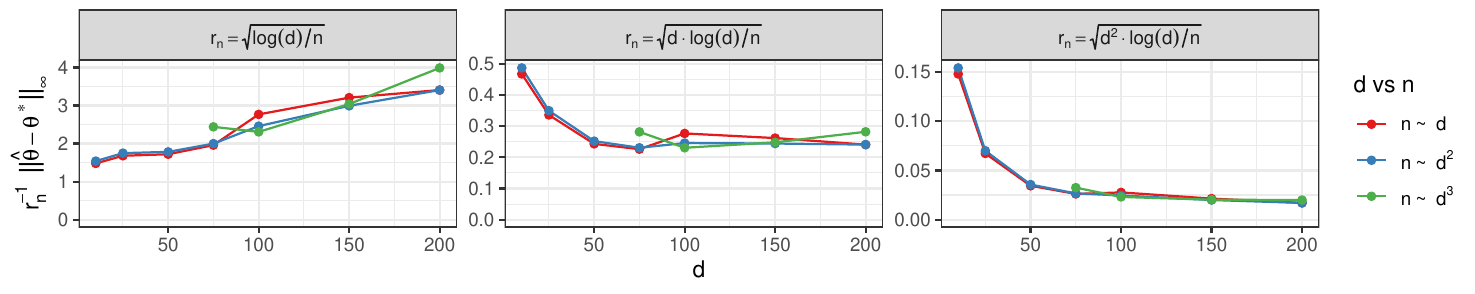}
  \caption{Parameter estimation for Gaussian D-vine with $\theta_t^* = 0.5/\sqrt{t+1}$, mean maximum norm of estimation error with different normalizations.}
  \label{fig:par_dvine}
\end{figure}

\begin{figure}
  \centering
  \includegraphics[width = \textwidth]{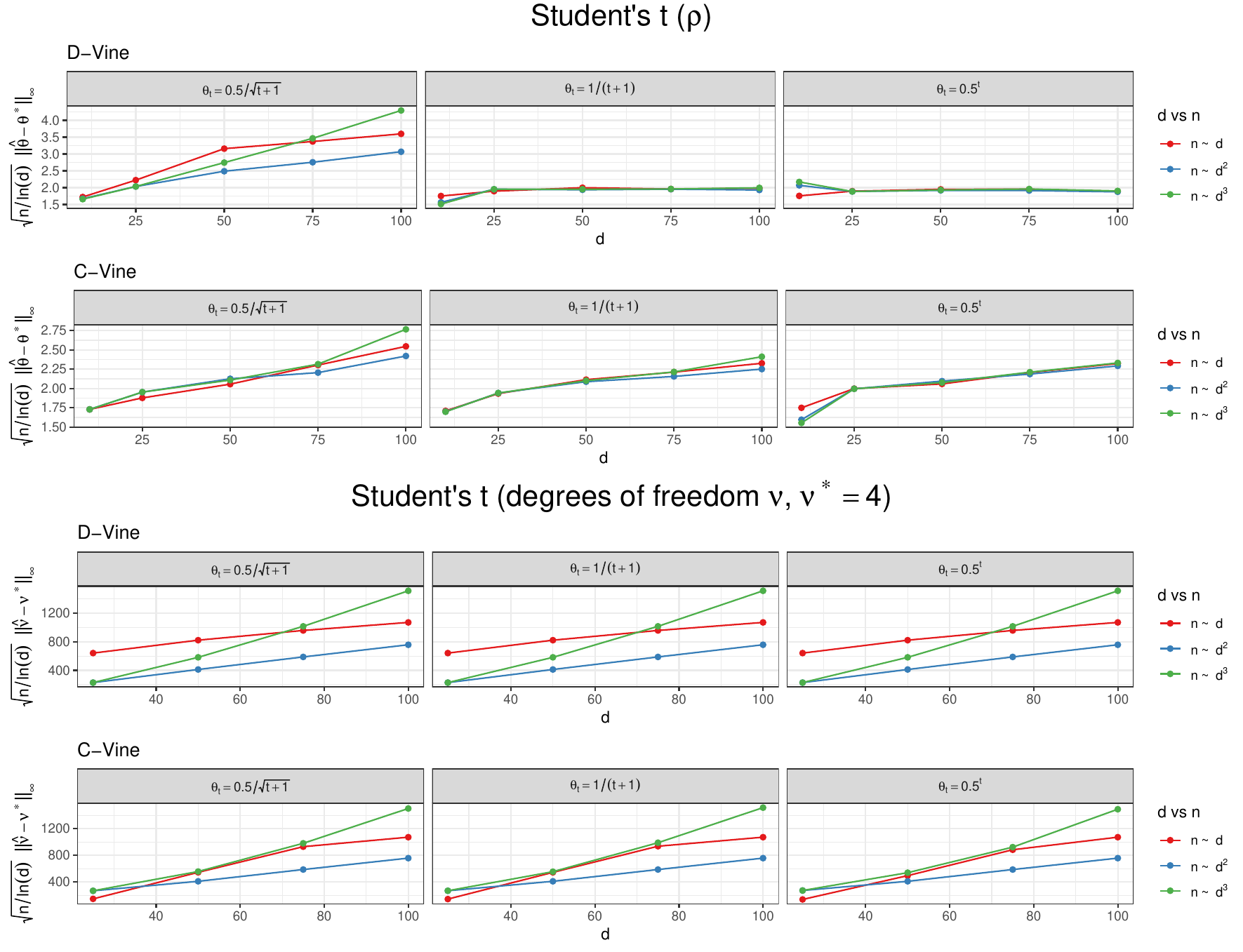}
  \caption{Parameter estimation for Student's t vines, mean maximum norm of estimation error for different proportions of $d$ and $n$. }
  \label{fig:par_student_norm}
\end{figure}

\begin{figure}
  \centering
  \includegraphics[width = \textwidth]{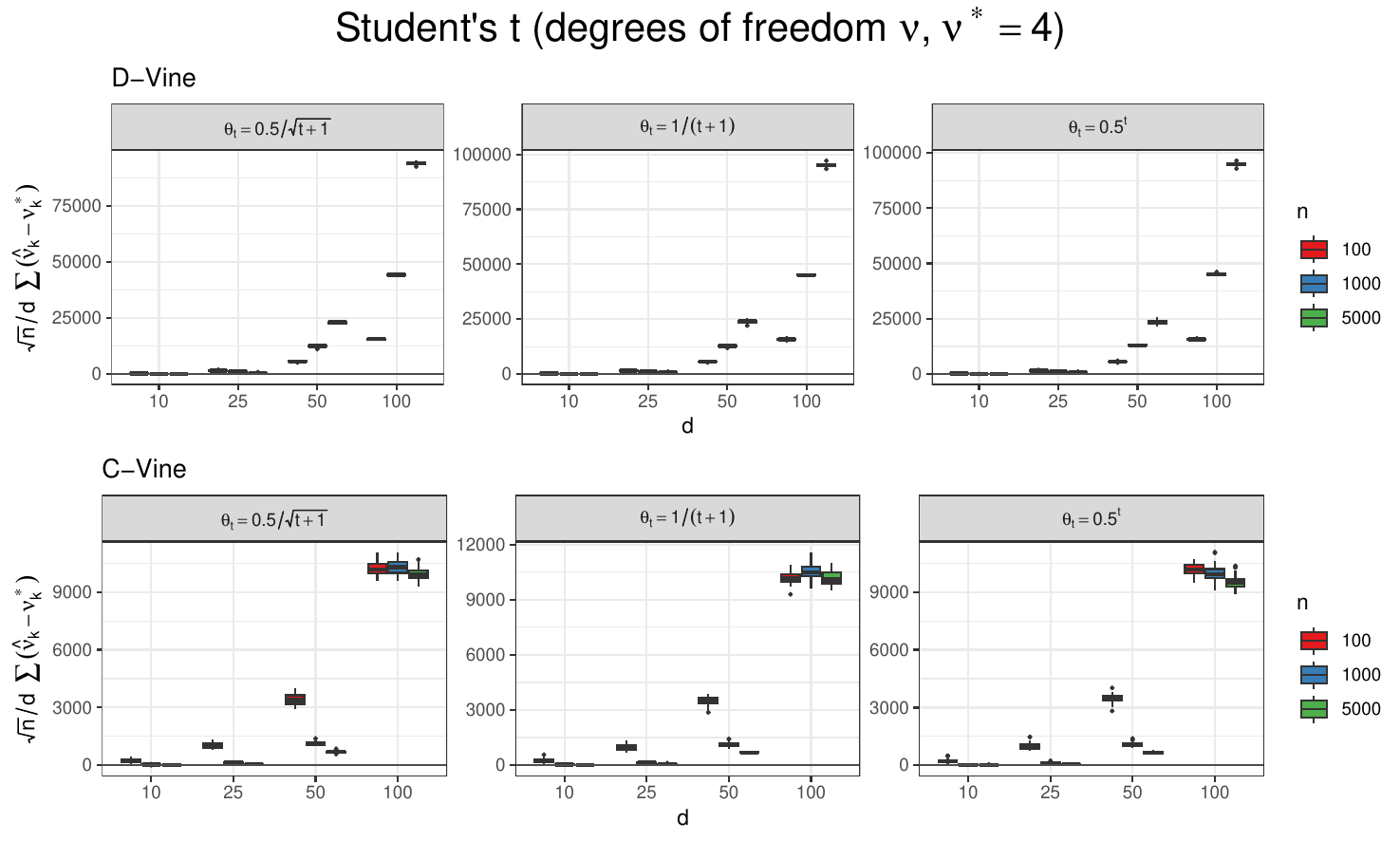}
  \caption{Parameter estimation for Student's t vines, sum of estimation errors. Each boxplot represents $N=100$ replications.}
  \label{fig:par_student}
\end{figure}

\begin{figure}
  \centering
  \includegraphics[width = 0.5\textwidth]{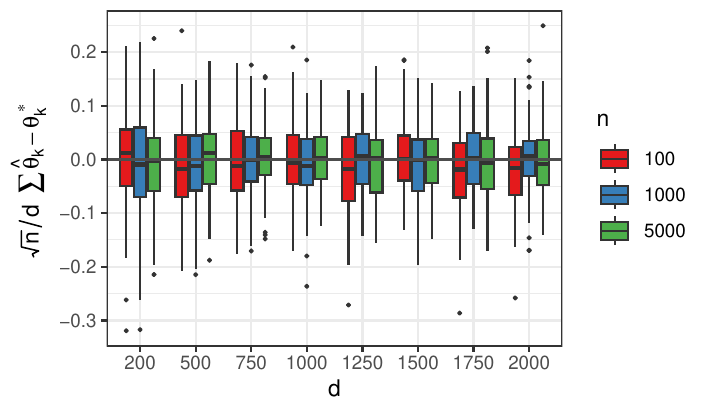}
  \caption{2-truncated Gaussian C-vine with $\theta_t^* = 1/(t+1)$, sum of estimation errors. Each boxplot represents $N=100$ replications.}
  \label{fig:par_estim_trunc}
\end{figure}

\begin{figure}
  \centering
  \includegraphics[width = \textwidth]{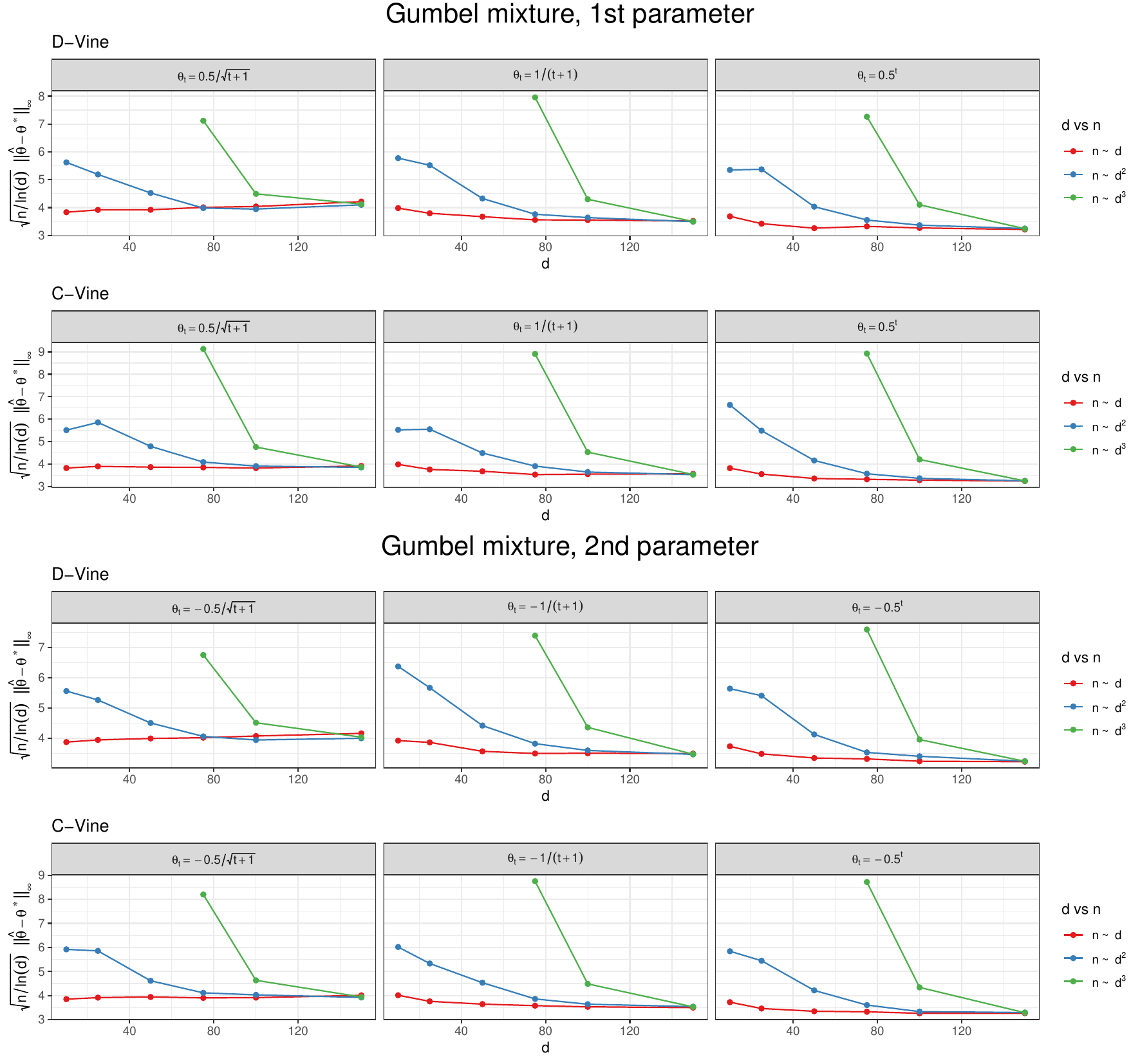}
  \caption{Parameter estimation for a mixture of two rotated Gumbels, mean maximum norm of estimation error for different proportions of $d$ and $n$. Parameterization: each pair copula is a 50--50 mixture of $\text{Gumbel}(\theta_1 + 1)$ for $\theta_1 \ge 0$, $\text{Gumbel}_{90}(-\theta_1+1)$ for $\theta_1 <0$ (first component) and $\text{Gumbel}_{180}(\theta_2 + 1)$ for $\theta_2 \ge 0$, $\text{Gumbel}_{270}(-\theta_2+1)$ for $\theta_2 <0$ (second component). We always set $\theta^*_1 = \theta^*_2$.}
  \label{fig:par_mix}
\end{figure}

\begin{figure}
  \centering
  \includegraphics[width = \textwidth]{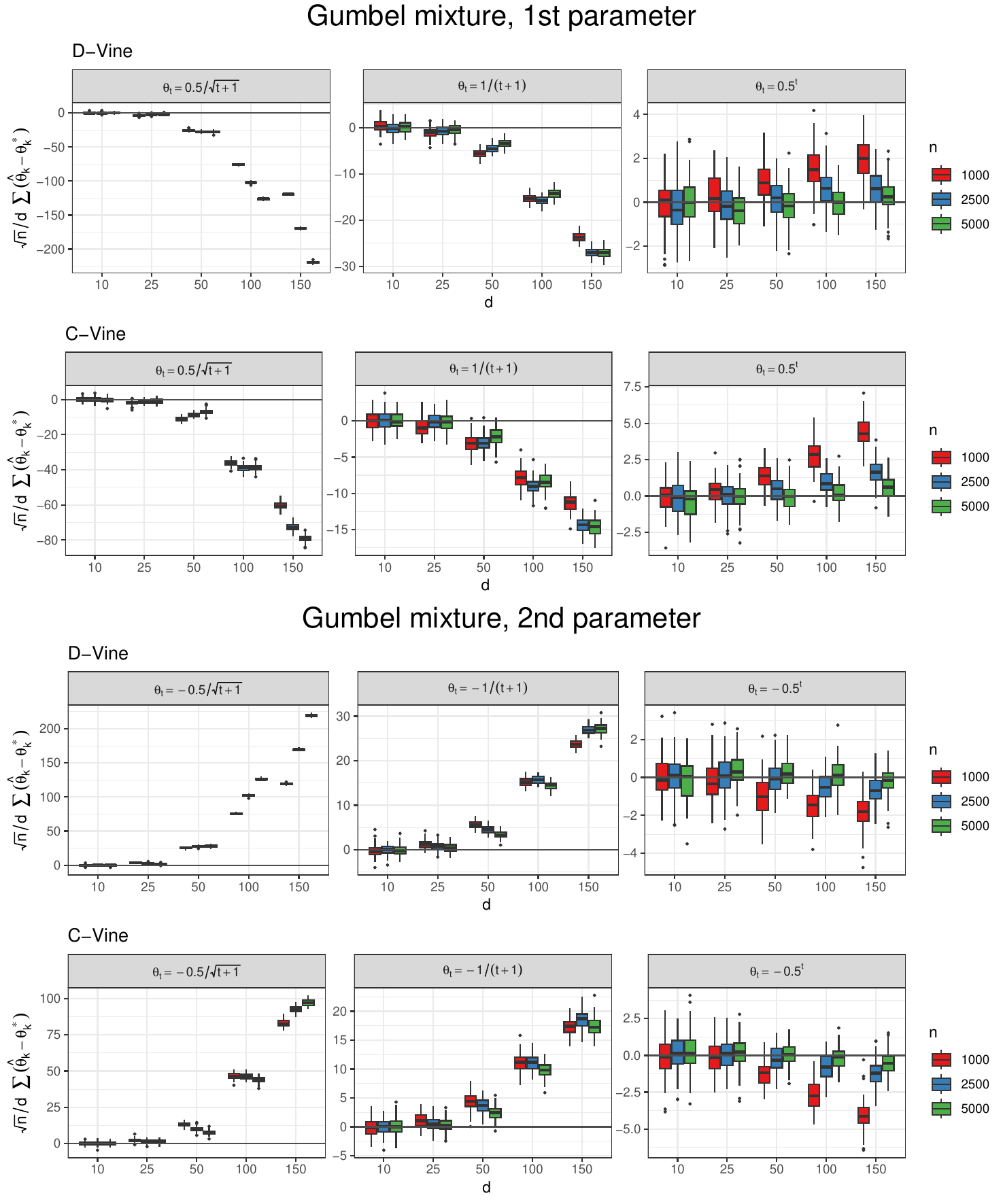}
  \caption{Parameter estimation for a mixture of two rotated Gumbels, sum of estimation errors. Each boxplot represents $N=100$ replications. Parameterization: each pair copula is a 50--50 mixture of $\text{Gumbel}(\theta_1 + 1)$ for $\theta_1 \ge 0$, $\text{Gumbel}_{90}(-\theta_1+1)$ for $\theta_1 <0$ (first component) and $\text{Gumbel}_{180}(\theta_2 + 1)$ for $\theta_2 \ge 0$, $\text{Gumbel}_{270}(-\theta_2+1)$ for $\theta_2 <0$ (second component). We always set $\theta^*_1 = \theta^*_2$.}
  \label{fig:par_mix2}
\end{figure}

\begin{figure}
  \centering
  \includegraphics[width = \textwidth]{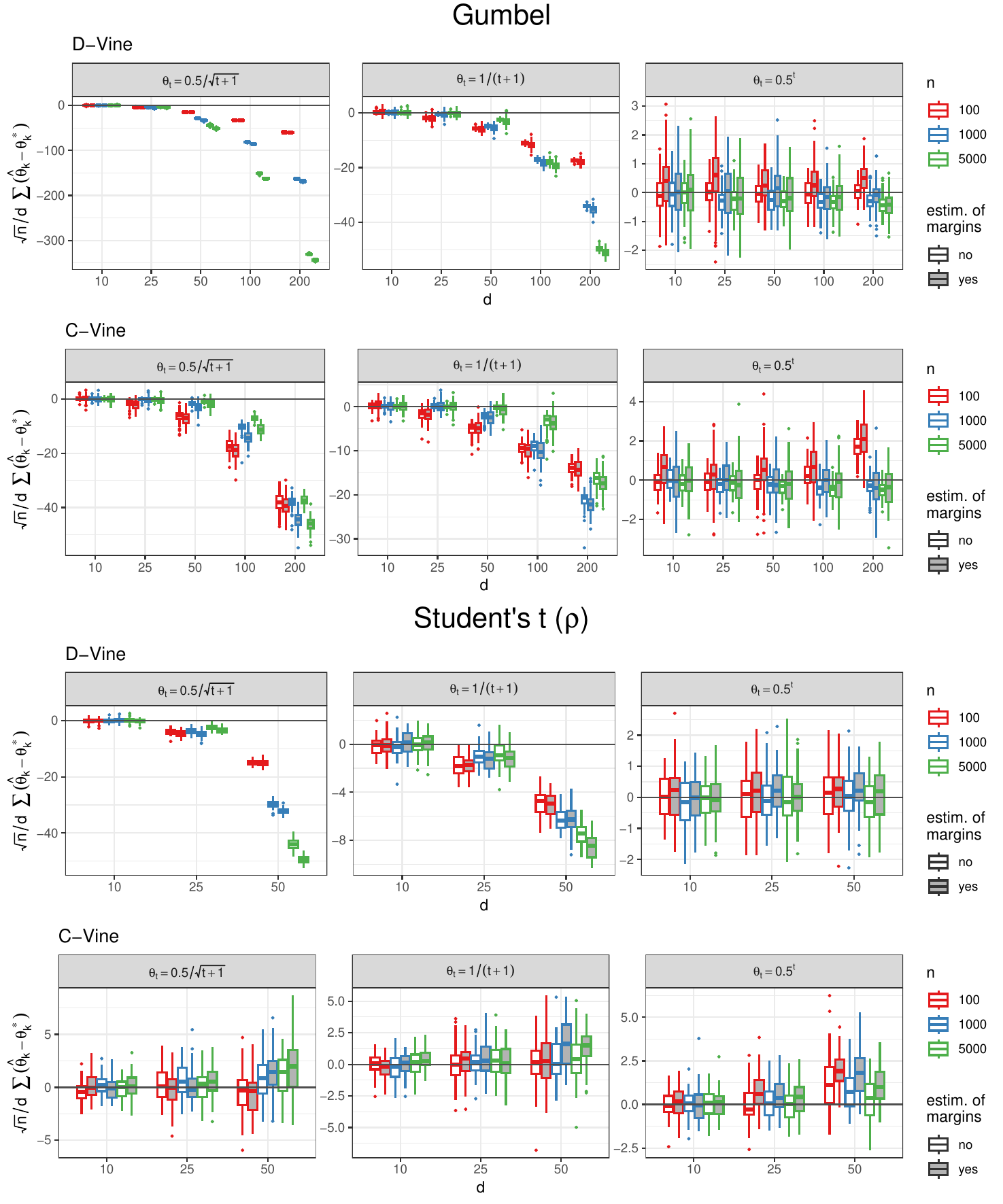}
  \caption{Parameter estimation for Gumbel and Student’s t vines with nonparametric estimation of margins, sum of estimation errors. Each boxplot represents $N=100$ replications.}
  \label{fig:par_margins_student_gumbel}
\end{figure}

\end{document}